\documentclass[12]{amsart}
\usepackage{amsmath}
\usepackage{amssymb}
\usepackage{tabularx}
\usepackage{enumerate}
\usepackage{graphicx}
\usepackage{texdraw}
\usepackage{hyperref}
\usepackage{color}

\usepackage{mathrsfs}

\usepackage{amsfonts,amssymb,amsmath}
\usepackage{epsfig}

\newtheorem{theorem}{Theorem}[section]
\newtheorem{lemma}[theorem]{Lemma}
\newtheorem{proposition}[theorem]{Proposition}
\newtheorem{corollary}[theorem]{Corollary}

\newcommand{\R}{\mathbb{R}}

\numberwithin{equation}{section}
\numberwithin{figure}{section}

\theoremstyle{remark}

\numberwithin{equation}{section}
\newcommand\bes{\begin{eqnarray}}
\newcommand\ees{\end{eqnarray}}
\newcommand\bess{\begin{eqnarray*}}
	\newcommand\eess{\end{eqnarray*}}
\newcommand{\lf}{\left}
\newcommand{\rr}{\right}
\newcommand{\dd}{\displaystyle}

\newcommand{\td}{\tilde}
\newcommand{\wtd}{\widetilde}
\newcommand\yy{\infty}

\newcommand{\ol}{\overline}
\newcommand{\rd}{{\rm d}}

\newcommand{\prt}{\partial_t}

\begin{document}
	
	\title[Biological propagation with nonlocal diffusion and free boundary]{\underline{\Large Lecture Notes:} \medskip\\
	Biological propagation via reaction-diffusion equations with nonlocal diffusion and free boundary} 
	\author[Y. Du]{{\bf Yihong Du}$^\dag$\medskip\\
	University of New England, Australia}
	\thanks{$^\dag$School of Science and Technology, University of New England, Armidale, NSW 2351, Australia}
	\thanks{\mbox{\bf \ Email:} ydu@une.edu.au}
	\date{\today}
	
	\begin{abstract}
	These notes are based on the  lectures given in a mini-course at VIASM (Vietnam Institute for Advanced Study in Mathematics) 2025 Summer School. They give a brief account of the theory (with detailed proofs) for propagation governed by a nonlocal reaction-diffusion model with free boundaries in one space dimension. The main part is concerned with a KPP reaction term, though the basic results on the existence and uniqueness of solutions as well as on the comparison principles are for more general situations. The contents are mostly taken from published recent works of the author with several collaborators, where the kernel function was assumed to be symmetric: $J(x)=J(-x)$. When $J(x)$ is not symmetric, significant differences may arise in the dynamics of the model, as shown in several preprints quoted in the references at the end of these notes, but many of the existing techniques can be easily extended to cover the ``weakly non-symmetric case", and this is done here with all the necessary details.
	\end{abstract}
	\maketitle
	
	\tableofcontents
	
	\section{Maximum principle and comparison results}

\subsection{A maximum principle}
Suppose the kernel function $J_i:\R\to\R$ $(i=1,2,... n)$ satisfy 
\begin{itemize}
	\item[{\rm $(\mathbf{J})$}:]   $J_i\in C(\R)\cap L^\infty(\R)$ is nonnegative, $J_i(0)>0$,   $\displaystyle\int_\R J_i(x) \rd x=1$,\; $ i=1,2,... n$.
\end{itemize}
The associated nonlocal diffusion operator $\mathcal L_i$ is defined by
\begin{align}\label{Li}
\mathcal{L}_i[u](t,x)=\int_{\R}J_i(x-y)u(t,y)\rd y-u(t,x), \  \ \ \ \; i=1,2,..., n.
\end{align}

Let $T>0$ and $\xi\in C([0,T])$. We define the set of {\color{red}strict local semi-maximum points} of $\xi$ by
\[
\Sigma_{max}^\xi:=\{t\in (0, T]: \mbox{ There exists $\epsilon>0$ such that $\xi(t)>\xi(s)$ for $s\in [t-\epsilon, t)$}\}.
\]
Similarly the set of {\color{red}strict local semi-minimum points} of $\xi$ is given by
\[
\Sigma_{min}^\xi:=\{t\in (0, T]: \mbox{ There exists $\epsilon>0$ such that $\xi(t)<\xi(s)$ for $s\in [t-\epsilon, t)$}\}.
\]
If $\xi$ is strictly increasing, then clearly $\Sigma_{max}^\xi=(0, T]$, if $\xi$ is nondecreasing, then $ \Sigma_{min}^\xi=\emptyset$. 
In particular, if $\xi$ is a constant function, then  $\Sigma_{max}^\xi=\Sigma_{min}^\xi=\emptyset$.

\begin{theorem}[Maximum Principle]\label{lemma2.1} Let $T, h_0>0$, $g, h\in C([0,T])$ satisfy $g(t)<h(t)$ and $-g (0)=h(0)=h_0$. Denote $D_T:=\{(t,x): t\in (0, T],\; g(t)<x<h(t)\}$ and suppose that for $i,j\in \{1,2, ..., n\}$,  $\phi_i$, $\partial_t\phi_i\in C(\ol D_T)$, $d_i$, $c_{ij}\in L^\yy(D_T)$, ${\color{red}d_i\geq  0}$,  and 
	\begin{equation}\label{mp-upper}
	\begin{cases}
	\dd \partial_t\phi_i\geq d_i\mathcal{L}_i[\phi_i]+\sum_{j=1}^n c_{ij}\phi_j, & (t,x)\in D_T,\; \\[-2mm]
	\phi_i(t,x)=0, & t\in (0, T], \ x\not\in [g(t), h(t)],\\
\phi_i(t,g(t))\geq 0,& t\in\Sigma_{min}^g,\\
 \phi_i(t,h(t))\geq 0,  & t \in \Sigma_{max}^h,\; \\
\phi_i(0,x)\geq 0, &x\in [-h_0,h_0],\; 
	\end{cases}
	\end{equation}
where $\mathcal{L}_i$ is given by \eqref{Li} with every $J_i$ $(i=1,..., n)$ satisfying {\bf (J)}. Then the following conclusions hold:
\begin{itemize}
\item[\rm (i)] If $c_{ij}\geq 0$ on $ D_T$ for $i, j\in \{1,..., n\}$ and {\color{red}$i\not=j$}, then $\phi_i\geq 0$ on $\ol D_T$ for $i\in\{1,..., n\}$.
\item[\rm (ii)] If in addition $d_{i_0}>0$ in $D_T$, $\phi_{i_0}(0,x)\not\equiv0$ in $[-h_0,h_0]$, then $\phi_{i_0}> 0$ in $D_T$.
\end{itemize}
\end{theorem}
\begin{proof} Since $\phi_i(t,x)=0$ for $x\not\in [g(t), h(t)]$, we have
\[
\mathcal{L}_i[\phi_i](t,x)=\int_{g(t)}^{h(t)}J_i(x-y)\phi_i(t,y)\rd y-\phi_i(t,x),\ i=1,..., n.
\]

{\bf Proof of part (i)}. We prove part (i) in two steps.\medskip

{\bf Step 1}. We first prove that if $(\phi_1,...,\phi_n)$  satisfies
\begin{equation}\label{1.2}
\begin{cases}
\dd \partial_t\phi_i> d_i\mathcal{L}_i[\phi_i]+\sum_{j=1}^nc_{ij}\phi_j, & (t,x)\in \ol D_T,\; i\in\{1,..., n\}\\
\phi_i(t,g(t))> 0, & t\in\Sigma_{min}^g,\; i\in\{1,..., n\},\\
 \phi_i(t,h(t))> 0,  & t \in \Sigma_{max}^h,\; i\in \{1,..., n\},\\
\phi_i(0,x)> 0 &x\in [-h_0,h_0],\; i\in\{1,..., n\},
\end{cases}
\end{equation}
then $\phi_i> 0$ on $\ol D_T$. 

Define 
\begin{align*}
T_1=\sup\{0<t\leq T: \phi_i(s,x)>0\ {\rm for }\ (s,x)\in D_t,\; i=1,..., n\}. 
\end{align*}
We have $T_1>0$ since $\phi_i(0,x)>0$ in $[-h_0,h_0]$ and $\phi_i$  is continuous for $i=1,..., n$.  If $T_1<T$, then there exists $i_0\in \{1,..., n\}$ and   $x_1\in [g(T_1),h(T_1)]$ such that
\begin{align}\label{2.3}
\phi_{i_0}(T_1,x_1)=0, \ \ {\rm and}\ \phi_{i}(t,x)\geq 0 \  {\rm for}\  (t,x)\in D_{T_1},\; i=1,..., n.
\end{align} 
We claim that
\bes\label{part_t}
\partial_t\phi_{i_0}(T_1,x_1)\leq 0.
\ees
This fact is evident if $x_1\in(g(T_1), h(T_1))$. If $x_1=g(T_1)$, then from \eqref{1.2} we can conclude that $T_1\not\in \Sigma_{min}^g$,
and hence there exists an increasing sequence $t_k\to T_1$ such that $g(t_k)\leq g(T_1)$. It follows that for all large $k$,  
$x_1=g(T_1)\in [g(t_k), h(t_k)]$ and hence $\phi_{i_0}(t_k, x_1)\geq 0$. This and the assumption $\partial_t\phi_i\in C(\overline D_T)$ clearly imply \eqref{part_t}. If $x_1=h(T_1)$, the proof of \eqref{part_t} is analogous.

Without loss of generality, we assume from now on $i_0=1$.    From $c_{1j}\geq 0$ for $j=2,..., n$, \eqref{2.3}, \eqref{part_t} and the first inequality of \eqref{1.2}, we obtain
\begin{align*}
0\geq \partial_t\phi_{1}(T_1,x_1)>d_1(T_1, x_1)\int_{g(T_1)}^{h(T_1)}J_1(x-y)\phi_1(T_1,y)\rd y+\sum_{j=2}^nc_{1j}\phi_j(T_1,x_1)\geq 0,
\end{align*}
which is a contradiction. Hence 
\[
\mbox{$T_1=T$,  $\phi_i(t,x)>0$ for $t\in [0, T)$, $x\in [g(t), h(t)]$, $i=1,..., n$.}
\]
It follows that $ \phi_{i}(t,x)\geq 0 \  {\rm for}\  (t,x)\in D_{T},\; i=1,..., n$. To complete the proof of Step 1, it remains to show
\[
 \phi_{i}(T,x)> 0 \  {\rm for}\  x\in [g(T), h(T)],\; i=1,..., n.
\]
If there exists $i_0\in \{1,..., n\}$ and   $x_0\in [g(T),h(T)]$ such that
$\phi_{i_0}(T,x_0)=0$, then we can repeat the above argument with $T_1=T$ to derive a contradiction. \bigskip

{\bf Step 2}. We  apply the conclusion in Step 1 to show the desired results. 

For $i\in\{1,..., n\}$, let $\psi_i(t,x)=\phi_i(t,x)+\epsilon e^{At}$ for some positive constants $\epsilon$ and $A$. Then 
\begin{equation*}
\begin{cases}
\psi_i(t,g(t))=\phi_i(t,g(t))+\epsilon e^{At}> 0,  & t\in\Sigma_{min}^g,\\
\psi_i(t,h(t))=\phi_i(t,h(t))+\epsilon e^{At}> 0,  & t\in\Sigma_{max}^h,\\
\psi_i(0,x)=\phi_i(0,x)+\epsilon\geq \epsilon> 0, &x\in [-h_0,h_0].
\end{cases}
\end{equation*}
Moreover,
\begin{align*}
&\dd \partial_t\psi_i-d_i\mathcal{L}_i[\psi_i] -\sum_{j=1}^nc_{ij}\psi_j\\
&=\dd \partial_t\phi_i- d_i\mathcal{L}_i[\phi_i] -\sum_{j=1}^nc_{ij}\phi_j
 +\epsilon Ae^{At}-d_i\epsilon e^{At}\Big[\int_{g(t)}^{h(t)}J_i(x-y)\rd y-1\Big]-\epsilon e^{At}\sum_{j=1}^nc_{ij}\\
&\geq \Big(A-d_i-\sum_{j=1}^n c_{ij}\Big)\epsilon e^{At}>0\ \   {\rm for}\ (t,x)\in \ol D_T,
\end{align*}
provided that $\dd A> \max_{1\leq i,j\leq n}\{||c_{ij}||_{L^{\yy}(D_T)}+d_i\}$. It then follows from Step 1 that  for any $\epsilon>0$ and $\dd A> \max_{1\leq i,j\leq n}\{||c_{ij}||_{L^{\yy}(D_T)}+d_i\}$, 
\begin{align*}
\psi_i(t,x)=\phi_i(t,x)+\epsilon e^{At}>0\ \ \  \ \ \  {\rm for}\ (t,x)\in \ol D_T,\; i=1,..., n.
\end{align*}
Fix $A$ and let  $\epsilon\to 0$, it gives $\phi_i\geq 0$ on $\ol D_T$ for $i=1,..., n$. This completes the proof of (i).
\medskip

{\bf Proof of part (ii)}. We now prove part (ii), that is,  $\phi_{i_0}> 0$ on $D_T$ under the additional conditions 
\begin{equation}\label{extra-c}
\mbox{$d_{i_0}(t,x)>0$ in $D_T$, \ $\phi_{i_0}(0,x)\not\equiv 0$  in $[-h_0,h_0]$.}
\end{equation}
  Suppose, on the contrary, 
  \[
  \mbox{\it there  exists a point $(T_1,x_1)\in D_T$ such  that $\phi_{i_0}(T_1,x_1)=0$.}
  \]
   To simplify notations, without loss of generality, let us again assume  $i_0=1$.

First, we claim that
\begin{align}\label{T1}
\phi_1(T_1,x)=0\ \ \  \ \  \ {\rm for}\ x\in (g(T_1),h(T_1)). 
\end{align}
If this is not true, then $\phi_1(T_1,\hat x_1)>0$ for some $\hat x_1\in (g(T_1), h(T_1))$. Let $I$ be the maximal open interval containing $\hat x_1$ such that $\phi_1(T_1, x)>0$ for $x\in I$. Then the existence of $x_1$ implies that at least one of the two boundary points of $I$ must be in $(g(T_1), h(T_1))$. So there exists
\begin{align*}
\tilde x_1\in (g(T_1),h(T_1))\cap \partial\{x\in(g(T_1),h(T_1)):  \phi_1(T_1,x)>0\}.
\end{align*}
Then it follows from $\phi_1(T_1,\tilde x_1)=0$, $c_{ij}\geq 0$ and $\phi_j\geq 0$ for $j\in\{2,..., n\}$ that
\begin{align*}
0&\geq \partial_t \phi_1(T_1,\tilde x_1)\geq d_1(T_1, \tilde x_1)\int_{g(T_1)}^{h(T_1)}J_1(\tilde x_1-y)\phi_1(T_1,y)\rd y+\sum_{j=2}^nc_{1j}(T_1,\tilde x_1)\phi_j(T_1,\tilde x_1)\\
&\geq  d_1(T_1, \tilde x_1)\int_{g(T_1)}^{h(T_1)}J_1(\tilde x_1-y)\phi_1(T_1,y)\rd y>0,\ \ \mbox{\color{red} [strict inequality due to $J(0)>0$]}
\end{align*}
which is a contradiction. Hence, $\phi_1(T_1,x)=0$ for all $x\in[g(T_1),h(T_1)]$.  \medskip

  Define
\[\mbox{
 $\Phi_i(t,x):=e^{Kt}\phi_i(t,x)$ with $K=d_1+\|c_{11}\|_\infty,\ i=1,..., n.$}
 \]
  Then $\Phi_1(t,x)$ satisfies $\Phi_1(t,x)\geq 0$,
\begin{align}\label{3.5a}
\Phi_1(T_1,x)=0\ \ \ \ {\rm for}\ x\in[g(T_1),h(T_1)]
\end{align}
and
\begin{equation}\label{3.6a}\begin{aligned}
\partial_t \Phi_1&=e^{Kt}(\phi_1)_t+K\Phi_1\\
&\geq e^{Kt}\Big[d_1\int_{g(t)}^{h(t)}J_1(x-y)\phi_1(t,y)\rd y+(-d_1+c_{11})\phi_1+\sum_{j=2}^nc_{1j}\phi_j\Big]+K\Phi_1\\
&= d_1\int_{g(t)}^{h(t)}J_1(x-y)\Phi_1(t,y)\rd y+(K-d_1+c_{11})\Phi_1+\sum_{j=2}^nc_{1j}\Phi_j\\
& \geq 0\ \mbox{ for $(t,x)\in  D_T$.}
\end{aligned}
\end{equation}

 We next use \eqref{3.5a} and \eqref{3.6a} to drive a contradiction. 
By \eqref{extra-c}
\[
\Omega_0:=\{x\in (-h_0, h_0): \Phi_1(0,x)>0\}\not=\emptyset.
\]
Since $g$ and $h$ are continuous and satisfy $g(t)<h(t)$  for all $t\in [0,T]$,  for any fixed $y_0\in\Omega_0$ there is a small constant $t_0 \in (0, T_1)$    such that 
\begin{align*}
g(t)<y_0<h(t) \mbox{ for } t\in[0, t_0].
\end{align*}
  
We claim that
\bes\label{y0}
\Phi_1(t_0,y_0)=0.
\ees

If this claim is proved, then  by \eqref{3.6a}, $\prt \Phi_1(t,y_0)\geq 0$ for $t\in [0,t_0]$, i.e., $\Phi(t,y_0)$ is nondecreasing for $t\in [0,t_0]$, which indicates that $\Phi_1(0,y_0)\leq 0$. However, this contradicts with  $y_0\in\Omega_0$. 

Therefore, to complete the proof, it suffices to show \eqref{y0}.
For clarity, we carry out the proof of \eqref{y0} according to two cases.

{\bf Case 1}. $\cap_{t\in [t_0,T_1]} (g(t),h(t))\neq \emptyset$.

In this case, we  take
\[
 y_1\in \cap_{t\in [t_0,T_1]} (g(t),h(t)),\ \ \ \
 \]
 and recall from \eqref{3.5a} and \eqref{3.6a} that $\Phi_1(T_1,y_1)=0$,  $\prt \Phi_1(t,y_1)\geq 0$ for $t\in [t_0,T_1]$. We then immediately see from $\Phi_1(t_0,y_1)\geq 0$ that $\Phi_1(t_0,y_1)= 0$. Now we may repeat the argument used to prove \eqref{T1} to conclude that 
 $\Phi_1(t_0,x)=0$ for $x\in [g(t_0),h(t_0)]$. In particular
$
\Phi_1(t_0,y_0)=0$, as desired.
This completes the proof in Case 1.\medskip

 {\bf Case 2}. $\cap_{t\in [t_0,T_1]} (g(t),h(t))=\emptyset$.

In this case we use a geometric argument in the  two-dimensional plane with $x$ and $t$ being the horizontal and vertical axis respectively.   Since $g(t)<h(t)$,
the continuous path given by
\[
\gamma_0:=\{(x,t): x=\xi(t)=\frac 12 [g(t)+h(t)],\; t\in [t_0, T_1]\},
\]
 is contained in the region $G:=\{(x,t): x\in (g(t),h(t)),\ t\in [t_0,T_1]\}$.
Clearly a small tubular neighbourhood of $\gamma_0$ still lies in $G$, and hence we can find a continuous path $\gamma_1$ close to $\gamma_0$ such that
$\gamma_1$ lies in $G$,  it consists of finitely many line segments in the $xt$-plane, and 
\[
\gamma_1\cap\{t=T_1\}=(\xi(T_1), T_1),\; \gamma_1\cap \{t=t_0\}=(\xi(t_0), t_0).
\]
For example, we could take $\gamma_1$ the piecewise linear curve connecting the points $p_i\in\gamma_0$, where $p_i=(\xi(s_i),s_i)$ with $s_i=\frac ik (T_1-t_0)+t_0$, $i=0,..., k$, for a large enough positive integer $k$.\medskip

 Similarly, a small tubular neighbourhood of $\gamma_1$ still lies in $G$, which allows us to find a
 continuous path $\gamma_2$ close to $\gamma_1$ with the following two properties:\medskip
 
\begin{itemize}
\item[(i)] $\gamma_2\subset G$, and $\gamma_2\cap\{t=T_1\}=(\xi(T_1), T_1),\; \gamma_2\cap \{t=t_0\}=(\xi(t_0), t_0)$,
\item[(ii)] $\gamma_2$ consists of finitely many line segments which are either vertical or horizontal. 
\end{itemize}

Let the horizontal line segments of $\gamma_2$ be denoted by $H_i$, $i=1,..., m$.
Then we can find $t_0<t_1< ... <t_m<t_{m+1}=T_1$ such that $H_i\subset \{ t=t_i\}$, $i=1, ..., m$. Let $V_j$ denote the vertical line segments of $\gamma_2$
that lies between $t_{j-1}$ and $t_j$, $j=1,..., m+1$,  then there exists $ x_j\in (g(t_j), h(t_j))$ such that 
$V_j\subset\{x=x_j\}$, $j=1,..., m+1$.
Thus
	\begin{equation*}
\begin{cases}
 \mbox{ the two end points of $V_i$ are } (x_i,t_{i-1})\ {\rm and\ } (x_i,t_{i}),& 1\leq i\leq m+1,\\
\mbox{ the two end points of $H_i$ are } (x_i,t_{i})\ \mbox{ and\ } (x_{i+1},t_{i}),&1\leq i\leq m.
\end{cases}
	\end{equation*}

We show that $\Phi_1(t_{m},x)=0$ for $x\in [g(t_m),h(t_m)]$. Thanks to  \eqref{3.5a} and \eqref{3.6a}, we have $\Phi_1(T_1,x_{m+1})=0$ and $\prt \Phi_1(t,x_{m+1})\geq 0$ for $t\in [t_m,T_1]$. This, 
combined with  $\Phi_1(t_m,x_{m+1})\geq 0$, yields $\Phi_1(t_m,x_{m+1})= 0$. The arguments leading to \eqref{T1} now infers that $\Phi_1(t_m, x)\equiv 0$ for $x\in 
[g(t_m), h(t_m)]$.
In other words, $\Phi_1(t_{m+1},x_{m+1})=0$ implies $\Phi_1(t_{m},x)=0$ for $x\in [g(t_m),h(t_m)]$.\medskip

Repeating the above argument, we can show $\Phi_1(t_i, x_i)=0$ implies $\Phi_1(t_{i-1}, x)=0$ for $x\in [g(t_{i-1}, h(t_{i-1})]$, $i=m,..., 1$. Thus we have  $\Phi_1(t_0,x)=0$  for $x\in [g(t_{0}),h(t_{0})]$, which clearly implies \eqref{y0}. The proof is now complete.
\end{proof}

\subsection{An example} A free boundary model for West Nile virus \cite{dn20}:
\begin{equation}\label{1}
\begin{cases}
		\dd H_t=d_1 \mathcal{L}_1[H](t,x)+a_1(e_1-H)V-b_1 H, & x\in (g(t), h(t)),\ t>0,\\
		\dd V_t=d_2 \mathcal{L}_2[V](t,x)+a_2(e_2-V)H-b_2 V , & x\in (g(t), h(t)),\ t>0,\\
		H(t,x)= V(t,x)=0, & t >0, \; x\in\{g(t), h(t)\},\\
	\dd h'(t)= \mu \int_{g(t)}^{h(t)}\int_{h(t)}^{\yy}J_1(y-x)H(t,y)\rd x, & t >0,\\[3mm]		
	\dd g'(t)= -\mu \int_{g(t)}^{h(t)}\int_{-\yy}^{g (t)}J_1(y-x)H(t,y)\rd x, & t >0,\\
	H(0,x)=u_1^{0}(x),\ V(0,x)=u_2^{0}(x), \ &x\in [-h_0,h_0].
\end{cases}
\end{equation}
Here  $H(t,x)$ and $V(t,x)$ stand for the densities of the infected bird (host) and mosquito (vector) populations at  time $t$ and spatial location $x$, respectively.  The interval $[g(t), h(t)]$ is the evolving region of virus infection. The parameters here are all  positive constants.
 The initial functions $u_{i}^{0}(x)$ $(i=1,2)$ satisfy 
\begin{equation}\label{1.2a}
\begin{cases}
u_{i}^{0}\in C([-h_0,h_0]),\ \ u_{i}^{0}(-h_0)=u_{i}^{0}(h_0)=0,\\
0<u_{i}^{0}(x)\leq e_i\ {\rm for}\ x\in (-h_0,h_0),\; i=1,2.
\end{cases}
\end{equation}

We can easily apply  Theorem \ref{lemma2.1} to obtain the following comparison results.  

\begin{corollary}\label{lemma3.7}  Assume   $(\mathbf{J})$  holds,   $T>0$, $ g,\,  h\in C([0,T])$ satisfy $ g(t)<  h(t)$, and $ D_T=\{(t,x): t\in (0,T],\;  g(t)<x< h(t)\}$. If  $H, \; V,\;  \tilde H,\; \tilde V \in C(\ol D_T)$ satisfy the following conditions:
\begin{itemize}
\item[\rm (i)] $\Phi_t\in C( \ol D_T)$ for $\Phi\in\{H,V, \td H,\td V\}$, 
\item[\rm (ii)]    $0\leq \Phi\leq e_1$ for $\Phi\in\{H,\td H\}$,\;  $0\leq \Psi\leq e_2$ for $\Psi\in\{V,\tilde V\}$,
\item[\rm (iii)] for $(t,x)\in  D_T$,
	\begin{equation}\label{upper}
	\begin{cases}
	\dd \tilde H_t\geq d_1 \int_{ g(t)}^{ h(t)}J_1(x-y)\tilde H(t,y)\rd y-d_1\tilde H+a_1(e_1-\tilde H) \tilde V-b_1 \tilde H, \\[3mm]
	\dd \tilde V_t\geq d_2 \int_{ g(t)}^{ h(t)}J_2(x-y)\tilde V(t,y)\rd y-d_2\tilde V+a_2(e_2-\tilde V)\tilde H-b_2\tilde V, 
		\end{cases}
	\end{equation}
	\item[\rm (iv)] for $(t,x)\in  D_T$, $(H, V)$ satisfies \eqref{upper} but with  the inequalities reversed, 
	\item[\rm (v)] at the boundary,
	\[\begin{cases}
	  H(t,g(t))\leq \tilde H(t,g(t)),\ V(t,g(t))\leq \tilde V(t,g(t))&{\rm for} \ t\in \Sigma_{min}^g,   \\
	  H(t,h(t))\leq \tilde H(t,h(t)),\ V(t,h(t))\leq \tilde V(t,h(t))&{\rm for} \ t\in \Sigma_{max}^h,
	\end{cases}\]
	\item[\rm (vi)] at the initial time,
	\[
	H(0,x)\leq \tilde H(0,x),\; V(0,x)\leq \tilde V(0,x)\ \ {\rm for} \ x\in [ g(0),  h(0)],
	\]
	\end{itemize}
	 then
	\begin{align*}
	  H(t,x)\leq \tilde H(t,x),\ V(t,x)\leq \tilde V(t,x)\ \ {\rm for} \ (t,x)\in  D_T.
	\end{align*}
\end{corollary}
\begin{proof}
Define 
\[
\phi_1:=\tilde H-H,\; \phi_2:=\tilde V-V,
\]
and
\[
c_{11}:=-(b_1+a_1V),\; c_{12}:=a_1(e_1-\tilde H),\; c_{21}:=a_1(e_2-\tilde V),\; c_{22}=-(b_2+a_2H).
\]
Then it is easily checked that $(\phi_1,\phi_2)$ satisfies \eqref{mp-upper} with $n=2$. Therefore $\phi_1\geq 0$ and $\phi_2\geq 0$ in $D_T$.
\end{proof}

\subsection{A comparison result for a scalar nonlocal free boundary problem}
Suppose the kernel function $J(x)$ satisfies the basic condition 
\begin{description}
\item[{\bf (J)}] $ J\in C(\R)\cap L^\infty(\R) \mbox{ is nonnegative},\  J(0)>0,~\displaystyle \int_{\mathbb{R}}J(x)dx=1$.
\end{description}
The function $f: \mathbb{R}^+\times\mathbb{R}\times\mathbb{R}^+
\rightarrow\mathbb{R}$  satisfies
\begin{description}
\item[(f1)] $f(t,x,0)\equiv 0$ and $f(t,x,u)$
is continuous in $(t,x,u)$ and locally Lipschitz in

 $u\in\R^+$, i.e., for any $L>0$,  there exists a constant $K=K(L)>0$ such that
\[\hspace{0.8cm}
\mbox{$\left|f(t,x,u_1)-f(t,x,u_2)\right|\le K|u_1-u_2|$ for $u_1, u_2\in [0, L]$,\  $(t,x)\in \R^+\times \R$.}
\]
\end{description}

\medskip

The nonlocal  free boundary problem to be considered has
 the following form:
\begin{equation}
\left\{
\begin{aligned}
&u_t=d\int_{g(t)}^{h(t)}J(x-y)u(t,y)dy-du(t,x)+f(t,x,u),
& &t>0,~x\in(g(t),h(t)),\\
&u(t,g(t))=u(t,h(t))=0,& &t>0,\\
&h'(t)=\mu\int_{g(t)}^{h(t)}\int_{h(t)}^{+\infty}
J(y-x)u(t,x)dydx,& &t>0,\\
&g'(t)=-\mu\int_{g(t)}^{h(t)}\int_{-\infty}^{g(t)}
J(y-x)u(t,x)dydx,& &t>0,\\
&u(0,x)=u_0(x),~h(0)=-g(0)=h_0,& &x\in[-h_0,h_0],
\end{aligned}
\right.
\label{101}
\end{equation}
where $x=g(t)$ and $x=h(t)$ are the moving boundaries
to be determined together with $u(t,x)$, which is always assumed to be identically 0 for $x\in \R\setminus [g(t), h(t)]$; $d$ and $\mu$
 are positive constants. The initial
function $u_0(x)$ satisfies
\begin{equation}
u_0(x)\in C([-h_0,h_0]),~u_0(-h_0)=u_0(h_0)=0
~\text{ and }~u_0(x)>0~\text{ in }~(-h_0,h_0),
\label{102}
\end{equation}
with  $[-h_0,h_0]$ representing the initial population range of the species.

\begin{theorem}\label{thm-CP}
$($Comparison principle$)$ Assume that ${\bf (J)}$ and  ${\bf (f1)} $ hold,  $u_0$ satisfies \eqref{102} and $(u,g,h)$ satisfies
\eqref{101}\footnote{Here we implicitly require $g, h\in C^1([0, T])$ and $u_t, u$ are continuous for $t\in [0, T]$, $x\in [g(t), h(t)]$.} for $0\leq t\leq T\in(0,+\infty)$. Suppose that $\overline h,\overline
g\in C^1([0,T])$ and
$\overline u_t(t,x), \overline u(t,x)$ are continuous for $t\in [0, T]$, $x\in[\overline g(t), \overline h(t)]$, and 
\begin{equation}\label{CP-upper}
\left\{
\begin{aligned}
&\overline u_t\ge d\int_{\overline g(t)}^{\overline h(t)}J(x-y)
\overline u(t,y)dy-d\overline u+f(t,x,\overline u), & &0<t\le T,~x\in(\overline g(t),\overline h(t)),\\
&\overline u(t,\overline g(t))\geq 0, \ \overline u(t,\overline h(t))\geq 0, & &0<t\le T,\\
&\overline h'(t)\ge\mu\int_{\overline g(t)}^{\overline h(t)}\int_{\overline h(t)}^{+\infty}J(y-x)\overline u(t,x)dydx,  & &0<t\le T,\\
&\overline g'(t)\le-\mu\int_{\overline g(t)}^{\overline h(t)}\int_{-\infty}^{\overline g(t)}J(y-x)\overline u(t,x)dydx,  & &0<t\le T,\\
&\overline u(0,x)\ge u_0(x),~\overline h(0)> h_0,~\overline g(0)<-h_0,  & & x\in[-h_0,h_0],\\
&\overline u(0,x)\geq 0, && x\in [\overline g(0), \overline h(0)].
\end{aligned}
\right.
\end{equation}
Then
\begin{equation}
u(t,x)<\overline u(t,x),~g(t)>\overline g(t)~\text{ and
}~h(t)<\overline h(t)~\text{ for }~0<t\leq T~\text{ and }~x
\in [g(t), h(t)].
\label{CP-compare}
\end{equation}
\end{theorem}

The triplet $(\overline u,\overline g,\overline h)$  above
is called an upper solution of \eqref{101}. We can define a lower
solution and obtain analogous results by reversing the inequalities
in (\ref{CP-upper}) and (\ref{CP-compare}).

\begin{proof} Due to ${\bf (f1)}$, we can write $f(t,x,\overline u(t,x))= c(t,x)\overline u(t,x)$ with $c\in L^\infty$. Hence we can apply Theorem \ref{lemma2.1} with $n=1$ to conclude that $\overline u>0$ for $0<t\leq T,~\overline g(t)<x<\overline h(t)$,  and thus  both $\overline h$ and $-\overline g$ are strictly increasing.

We claim that  $h(t)<\overline h(t)$ and $g(t)
>\overline g(t)$ for all $t\in(0,T]$.
 Clearly, these hold true for
small $t>0$. Suppose by way of contradiction that there exists
$t_1\in (0, T]$ such that
$$
h(t)<\bar h(t),\ g(t) >\overline g(t)\  \textrm{for}\  t\in(0,t_1)\  \textrm{and} \ [h(t_1) -\overline h(t_1)][  g(t_1)- \overline g(t_1)]=0.
$$
Without loss of generality, we may assume that
\[
 h(t_1) =\overline h(t_1) \mbox{ and }  g(t_1)\geq  \overline g(t_1).
 \]

We now compare $u$ and $\overline u$ over the region
$$
\Omega_{ t_1}:=\left\{(t,x)\in\mathbb{R}^2:0<t\leq t_1,
~g(t)<x<h(t)\right\}.
$$
Let  $w(t,x):=\overline u(t,x)-u(t,x)$. Then for  $(t,x)
\in\Omega_{t_1}$, we have
\begin{equation}
w_t\ge d\int_{g(t)}^{h(t)}J(x-y)w(t,y)dy
-dw(t,x)+C(t,x)w(t,x),
\label{304}
\end{equation}
for some  $L^\infty$ function $C(t, x)$. Moreover,
\[
w(t, g(t))> 0,\ w(t, h(t))> 0 \mbox{ for } t\in (0, t_1),\ w(0,x)\geq 0 \mbox{ for }x\in [-h_0, h_0].
\]
Therefore it follows from Theorem \ref{lemma2.1} that $w(t,x)\geq 0$ in $\Omega_{t_1}$. Moreover, for any $t_0\in (0, t_1)$,
$w(t_0, h(t_0))>0$ and so $w(t_0, x)\geq,\not\equiv 0$ in $[g(t_0), h(t_0)]$. So we can apply Theorem \ref{lemma2.1} over $t\in [t_0, t_1]$, $x\in [g(t), h(t)]$ to deduce $w(t, x)>0$
in this range. Since $t_0$ can be arbitrarily small we obtain 
\[
\mbox{$w(t,x)=\overline u(t,x)-u(t,x)>0$ for $t\in (0, t_1]$, $x\in [g(t), h(t)]$.}
\]

On the other hand, by the definition of $t_1$, we have   
\[
h(t_1)=\overline h(t_1),\ h'(t_1) \ge\overline h'(t_1).
\]
 This leads to the following contradiction:
\begin{align*}
0 &\ge  \overline h'(t_1)-h'(t_1)\\
  & \geq  \mu\int_{\overline g(t_1)
}^{\overline h(t_1)}\!\!\!\int_{\overline h(t_1)}^{+\infty}\!\!\!J(y-x)
 \overline u(t_1,x) dydx - \mu\int_{  g(t_1)
}^{  h(t_1)}\!\!\!\int_{ h(t_1)}^{+\infty}\!\!\!J(y-x)
  u(t_1,x) dydx\\
  &\geq \mu\int_{  g(t_1)
}^{  h(t_1)}\!\!\!\int_{ h(t_1)}^{+\infty}\!\!\!J(y-x)
 \big[\overline u(t_1,x)- u(t_1,x)\big] dydx>0. \ \mbox{\color{red} [strict inequality due to $J(0)>0$]}
\end{align*}
 The claim is thus proved, i.e., we always have $h(t)<\overline h(t)$ and $g(t)
>\overline g(t)$ for all $t\in(0,T]$. 

We may now use the comparison principle to obtain $\overline u(t,x)\geq u(t,x)$ for $t\in [0, T]$, $x\in [g(t), h(t)]$, and $\overline u(t,x)>u(t,x)$ for $t\in (t_0, T]$, $x\in [g(t), h(t)]$ for any $t_0\in (0, T)$.
\end{proof}

{\bf Remarks:} Theorem \ref{lemma2.1} is a simple variation of Lemma 3.1 in \cite{dn20}. Theorem \ref{thm-CP} is a simple variation of Theorem 3.1 in \cite{cdjfa}.

	\section{Existence and uniqueness} 

The following theorem is the main result to be proved here.

\begin{theorem}
Suppose that {\rm \bf (J)} and {\rm \bf (f1)-(f2)} hold. Then for any given $h_0>0$ and $u_0(x)$
satisfying \eqref{102}, problem \eqref{101} admits a unique
 solution $(u(t,x), g(t), h(t))$ defined for all $t>0$. Moreover, for any $T>0$, $g\in \mathbb G_{h_0, T},\;
 h\in\mathbb H_{h_0, T}$ and $u\in \mathbb{X}_{u_0, g, h}$.
\label{Thm22}
\end{theorem}

Here, and in what follows, for given $h_0, T>0$ we define
\begin{align*}
&\mathbb H_{h_0, T}:=\Big\{h\in C([0,T])~:~h(0)=h_0,
\; \inf_{0\le t_1<t_2\le T}\frac{h(t_2)-h(t_1)}{t_2-t_1}>0\Big\},\\
&\mathbb G_{h_0, T}:=\Big\{g\in C([0,T])~:-g\in\mathbb{H}_{h_0, T}\Big\},\\
&C_0([-h_0, h_0]):=\Big\{u\in C([-h_0, h_0]): u(-h_0)=u(h_0)=0\Big\}.
\end{align*}
\medskip

For $g\in \mathbb G_{h_0, T}$, $h\in\mathbb H_{h_0, T}$ and $u_0\in C_0([-h_0, h_0])$ nonnegative, we define\medskip

\begin{align*}
&\Omega=\Omega_{g, h}:=\left\{(t,x)\in\mathbb{R}^2: 0<t\leq T,~g(t)<x<h(t)\right\},\\
&\mathbb{X}=\mathbb X_{u_0,g,h}:=\Big\{\phi\in C(\overline\Omega_{g,h})~:~\phi\ge0~\text{in}
~\Omega_{g,h},~\phi(0,x)=u_0(x)~\text{for}~x\in [-h_0,h_0]~\\
& \hspace{5.6cm} \text{and} \;\; \phi(t,g(t))=
\phi(t,h(t))=0~\text{for }0\le t\le T\Big\}.
\end{align*}

\subsection{An auxiliary initial boundary value problem}

For any $T>0$ and $(g, h)\in\mathbb G_{h_0,T}\times \mathbb H_{h_0, T}$,
 we consider the following problem:
 \medskip
 
\begin{equation}
\left\{
\begin{aligned}
&v_t=d\int_{g(t)}^{h(t)}J(x-y)v(t,y)dy-dv+f(t,x,v),
& &0<t\leq T,~x\in(g(t),h(t)),\\
&v(t,h(t))=v(t,g(t))=0,& &0<t\leq T,\\
&v(0,x)=u_0(x),& &x\in[-h_0,h_0]
\end{aligned}
\right.
\label{201}
\end{equation}
\medskip

\begin{lemma}
Suppose that {\rm \bf (J)} and {\rm \bf (f1)-(f2)} hold, $h_0>0$ and $u_0(x)$
satisfies \eqref{102}. Then \eqref{201} admits a unique  solution, denoted by $V_{g,h}(t,x)$.
 Moreover $V_{g,h}$ satisfies
\begin{equation}
0<V_{g,h}(t,x)\le \max\left\{\max_{-h_0\le x\le h_0}u_0(x),
~K_0\right\}~\mbox{ for } 0<t\leq T,~x\in(g(t),h(t)),
\label{v-bound}
\end{equation}
where $K_0$ is defined in the assumption {\rm\bf (f2)}.
\label{Lemma202}
\end{lemma}

\noindent
{\bf Strategy of the proof of Theorem \ref{Thm22}:}
By Lemma \ref{Lemma202},
for any $T>0$ and $(h,g)\in\mathbb{G}_{h_0,T}\times\mathbb H_{h_0,T}$, we can find a
unique $V_{g,h}\in\mathbb{X}_{u_0,g,h}$ that solves (\ref{201}), and it has the property
$$
0<V_{g,h}(t,x)\le M_0:=\max\big\{\|u_0\|_\infty,~K_0\big\} \mbox{ for  } (t,x)\in
\Omega_{g,h}.
$$

\noindent
\underline{\it A nonlinear mapping}: Using  $V_{g,h}(t,x)$, we define a mapping $\tilde \Gamma$ by 
\[\tilde \Gamma (g,h)=
\left(\tilde g,\,\tilde h\right), \ \mbox{ where, for $0<t\leq T$,}\]
\begin{equation*}
\left\{
\begin{aligned}
&\tilde g(t)=-h_0-\mu\int_0^t\int_{g(\tau)}^{h(\tau)}\int_{-\infty}^{g(\tau)}J(y-x)V_{g,h}(\tau,x)dydxd\tau,\\
&\tilde h(t)=h_0+\mu\int_0^t\int_{g(\tau)}^{h(\tau)}\int_{h(\tau)}^{+\infty}J(y-x)V_{g,h}(\tau,x)dydxd\tau.
\end{aligned}
\right.
\label{2003}
\end{equation*}
 \smallskip

\noindent
\underline{\it Local existence}: We will show
that if $T$ is small enough, then $\tilde\Gamma$ maps  a suitable closed subset $\Sigma_T$ of $\mathbb{G}_{h_0,T}\times\mathbb{H}_{h_0,T}$ into itself,
and  is a {\color{red}contraction mapping}. This  implies that $\tilde \Gamma$ has a unique {\color{red}fixed point} $(g,h)$ in $\Sigma_T$,
which gives a solution $(V_{g,h}, g, h)$ of \eqref{101} defined for $t\in (0, T]$. 
\smallskip

\noindent
\underline{\it Global existence}:  We will then show that this unique solution
defined locally in time can be extended uniquely for all $t>0$.
\bigskip

\noindent
{\bf Proof of Lemma \ref{Lemma202}:} We break the proof into three steps.

\medskip

\noindent {\bf Step 1:} {\it A parametrized ODE problem.}

For given $x\in[g(T),h(T)]$, define
\begin{equation}\begin{aligned}
\tilde u_0(x)&:=
\begin{cases}
0 &\mbox{ if } x\not\in[-h_0,h_0],\\
u_0(x)&\mbox{ if }  x\in[-h_0,h_0].
\end{cases}
\\
t_x&:=
\begin{cases}
t_{x,g}& \mbox{ if  $x\in[g(T),-h_0)$ and $g(t_{x,g})=x$},\\
0 &\mbox{ if } x\in[-h_0,h_0],\\
t_{x,h} &\mbox{ if $x\in(h_0,h(T)]$ and $h(t_{x,h})=x$}.
\end{cases}
\end{aligned}
\label{definition-of-t_x}
\end{equation}
Clearly $t_x=T$ for $x=g(T)$ and $x=h(T)$,  $t_x<T$ for $x\in (g(T), h(T))$, and 
\[\mbox{$x\to t_x$ is continuous over $[g(T), h(T)]$.}
\]

 For any given $\phi\in
\mathbb{X}_{u_0, g,h}$, consider the following ODE initial value problem (with parameter $x$):
\begin{equation}\label{202}
\left\{
\begin{aligned}
 &v_t=d\int_{g(t)}^{h(t)}J(x-y){\color{red}\phi(t,y)}dy-dv(t,x)+\tilde f(t,x,v),& &t_x<t\le T,\\
&v(t_x,x)=\tilde u_0(x),& &x\in(g(T),h(T)),
\end{aligned}
\right.
\end{equation}
where 
\[
\tilde f(t,x,v):=\begin{cases} 0 & \mbox{ for $v<0$,}\\ 
f(t,x,v) & \mbox{ for $v\geq 0$.}
\end{cases}
\]
 Clearly $\tilde f$ also satisfies  {\rm \bf (f1)-(f2)}.
Denote
$$
F(t,x,v):=d\int_{g(t)}^{h(t)}J(x-y)\phi(t,y)dy-dv(t,x)
+\tilde f(t,x,v).
$$
Thanks to the assumption \textbf{(f1)}, for any $v_1,v_2
\in (-\infty, L]$, we have
$$
\Big|F(t,x,v_1)-F(t,x,v_2)\Big|\le\Big|\tilde f(t,x,v_1)-\tilde f(t,x,v_2)\Big|
+d\Big|v_1-v_2\Big|\le K_1\Big|v_1-v_2\Big|,
$$
where
\[
  L:=1+\max\Big\{\|\phi\|_{C(\overline\Omega_T)},  K_0\Big\},\; K_1:=d+K(L).
\]
In other words, the function $F(t,x,v)$ is Lipschitz continuous in $v$ for $v\in (-\infty, L]$
 with Lipschitz
constant $K_1$, uniformly for $t\in [0, T]$ and $x\in (g(T), h(T))$. Additionally, $F(t,x,v)$ is continuous
in all its variables in this range.
Hence it follows from the Fundamental Theorem of
ODEs  that, for every fixed $x\in (g(T), h(T))$, problem (\ref{202}) admits a
unique solution, denoted by $V_{\phi}(t,x)$ defined in some  interval $[t_x,T_x)$ of $t$.

We claim that $t\to V_\phi(t,x)$ can be uniquely extended to $[t_x, T]$. Clearly it suffices to show that if $V_\phi(t,x)$ is uniquely defined for $t\in [t_x, \tilde T]$ with $\tilde T\in (t_x, T]$, then
\begin{equation}\label{V-bd}
0\leq V_\phi(t,x)< L \mbox{ for } t\in (t_x, \tilde T].
\end{equation}

We first show that $V_\phi(t,x)<L$ for $ t\in (t_x, \tilde T]$. Arguing indirectly we assume that this inequality does not hold,
and hence, in view of $V_\phi(t_x,x)=\tilde u_0(x)\leq \|\phi\|_{C(\overline\Omega_T)} <L$, there exists some $t^*\in (t_x, \tilde T]$ such that
$V_\phi(t,x)<L$ for $t\in (t_x, t^*)$ and $V_\phi (t^*,x)=L$. It follows that $(V_\phi)_t(t^*,x)\geq 0$ and $\tilde f(t^*,x, V_\phi(t^*,x))\leq 0$
(due to $L>K_0$). We thus obtain from the differential equation satisfied by $V_\phi(t,x)$ that
\[
dL=dV_\phi(t^*,x)\leq d\int_{g(t^*)}^{h(t^*)}J(x-y)\phi(t^*,y)dy\leq d\|\phi\|_{C(\overline\Omega_T)}\leq d(L-1).
\]
It follows that $L\leq L-1$. This contradiction proves our claim.

We now prove the first inequality in \eqref{V-bd}. Since
\[
\tilde f(t,x,v)=\tilde f(t,x,v)-\tilde f(t,x,0)\geq -K(L) |v| \mbox{ for } v\in (-\infty, L],
\]
we have
\[
(V_\phi)_t\geq -K_1{\rm sgn}(V_\phi) V_\phi +d\int_{g(t)}^{h(t)}J(x-y)\phi(t,y)dy\geq -K_1 {\rm sgn}(V_\phi)V_\phi \mbox{ for } t\in [t_x, \tilde t].
\]
Since $V_\phi(t_x,x)=\tilde u_0(x)\geq 0$, the above inequality immediately gives $V_\phi(t,x)\geq 0$ for $t\in [t_x, \tilde T]$.
 We have thus proved \eqref{V-bd}, and therefore the solution $V_\phi(t,x)$ of  \eqref{202} is uniquely defined for $t\in [t_x, T]$.

\medskip

\noindent {\bf Step 2:} {\it A fixed point problem.}

Let us note that $V_{\phi}(0,x)=u_0(x)$ for $x\in [-h_0, h_0]$,
and $V_{\phi}(t,x)=0$ for $t\in [0, T)$ and $x\in\partial(g(t), h(t))=\big\{ g(t), h(t)\big\}$. Moreover, by the continuous dependence of the unique ODE solution on the initial value and on the parameters in the equation, we also see that $V_\phi(t,x)$ is continuous in $(t,x)\in\overline \Omega_T$, and hence $V_\phi\in \mathbb X_{u_0, g, h}$.  We now define
$\Gamma: \mathbb X_{u_0, g, h}\to \mathbb X_{u_0, g, h}$ by
\[
\Gamma \phi=V_{\phi}.
\]
and notice that $\phi$ solves \eqref{201} if it is a fixed point of $\Gamma$.

We want to show that $\Gamma$ is a contraction mapping if $T$ is replaced by a sufficiently small $s\in (0, T]$. 
For convenience of notation, we define
for any $s\in(0,T]$,
\[
\Omega_s:=\Big\{(t,x)\in\Omega_{g,h}: t\leq s\Big\},\; \mathbb X_s:=\Big\{\psi|_{\overline\Omega_s}:\;\psi\in\mathbb X_{u_0,g,h}\Big\}.
\]
We then
define the mapping $\Gamma_s: \mathbb X_s\to \mathbb X_s$ by
\[
\Gamma_s \psi=V_{ \psi}.
\]
Clearly, if $\Gamma_s \psi=\psi$ then $\psi(t,x)$ solves \eqref{201} for $t\in (0, s]$, and vice versa.

We show next that for sufficiently small $s>0$, $\Gamma_s$ has a unique fixed point in $\mathbb X_s$.
We will prove this conclusion by the contraction mapping theorem; namely we prove that for such $s$, {\color{red}$\Gamma_s$ is a contraction mapping on a closed subset of $\mathbb X_s$, and any fixed point of $\Gamma_s$ in $\mathbb X_s$ lies in this closed subset.}

Firstly we note that
 $\mathbb{X}_s$ is a complete metric space
with the metric
$$
d(\phi_1,\phi_2)=\|\phi_1-\phi_2\|_{C(\overline\Omega_s)}.
$$
Fix $M> \max\big\{4\|u_0\|_\infty, K_0\big\}$ and define
\[
\mathbb X_s^M:=\big\{ \phi\in \mathbb X_s: \; \|\phi\|_{C(\overline \Omega_s)}\leq M\big\}.
\]
Clearly $\mathbb X_s^M$ is a closed subset of $\mathbb X_s$. We show next that there exists $\delta>0$ small depending on $M$ such that
for every $s\in (0, \delta]$, $\Gamma_s$ maps $\mathbb X_s^M$ into itself, and is a contraction mapping.

Let $\phi\in \mathbb X_s^M$ and denote $v=\Gamma_s \phi$. Then $v$ solves  \eqref{202}  with $T$ replaced by $s$. It follows that \eqref{V-bd} holds
with {$\tilde T$} replaced by $s$ and $V_\phi$ replaced by $v$. We prove that for all small $s>0$,
\[
v(t,x)\leq M \mbox{ for } t\in [t_x, s],\; x\in (g(s), h(s)),
\]
which is equivalent to $\|v\|_{C(\overline \Omega_s)}\leq M$.

Let us observe that due to {\rm\bf (f1)-(f2)}, there exists $K_*>0$ such that
\[
f(t,x,u)\leq K_* u \mbox{ for all } u\in [0,\infty).
\]
Now from \eqref{202} we obtain, for $ t\in [t_x, s]$ and $ x\in (g(s), h(s))$,
\[
v_t\leq d\int_{g(t)}^{h(t)}J(x-y)\phi(t,y)dy+K_*v\leq d\|\phi\|_{C(\overline \Omega_s)}+K_*v.
\]
It follows that, for such $t$ and $x$,
\[
e^{-K_*t} v(t, x)- e^{-K_*t_x}  v(t_x,x)\leq d\int_{t_x}^t e^{-K_*\tau}d\tau \|\phi\|_{C(\overline \Omega_s)},
\]
and
\[
v(t,x)\leq \|u_0\|_\infty e^{K_*t}+d(t-t_x) e^{K_*t}  \|\phi\|_{C(\overline \Omega_s)}\leq \|u_0\|_\infty e^{K_*s}+ds e^{K_* s } M.
\]
If $\delta_1>0$ is small enough such that
\[
d\delta_1  e^{K_* \delta_1} \leq \frac 14,\; e^{K_*\delta_1}\leq 2,
\]
then for $s\in (0,\delta_1]$ we have
\[
v(t,x)\leq \frac 14(8 \|u_0\|_\infty+M)\leq M \mbox{ in } \Omega_s.
\]
Thus $v=\Gamma_s\phi\in \mathbb X_s^M$, as we wanted. Let us note from the above choice of $\delta_1$ that it only depends on
$d$ and $K_*$.

Next we show that by shrinking $\delta_1$ if necessary, $\Gamma_s$ is a contraction mapping on $\mathbb X_s^M$ when $s\in (0, \delta_1]$.
So let $\phi_1,\phi_2\in\mathbb{X}_s^M$, and denote
$V_i=\Gamma_s\phi_i$, $i=1,2$. Then $w=V_1-V_2$ satisfies
\begin{equation}
\left\{
\begin{aligned}
&w_t+c_1(t,x)w=d\int_{g(t)}^{h(t)}J(x-y)\left(\phi_1-\phi_2\right)(t,y)dy,& &t_x<t\le s,~x\in(g(t),h(t)),\\
&w(t_x,x)=0,& &x\in (g(t), h(t)),
\end{aligned}
\right.
\label{203}
\end{equation}
where
$$
c_1(t,x):=d-\frac{f(t,x,V_1)-f(t,x,V_2)}{V_1-V_2} \mbox{ and hence } \|c_1\|_\infty\leq K_1(M):=d+K(M).
$$
It follows that, for $t_x<t\le s$ and $x\in  (g(t),h(t))$,
\begin{align*}
w(t,x)=de^{-\int_{t_x}^tc_1(\tau,x)d\tau}
\int_{t_x}^te^{\int_{t_x}^\xi c_1(\tau,x)d\tau}\int_{g(\xi)}^{h(\xi)}J(x-y)\left(
\phi_1-\phi_2\right)(\xi,y)dyd\xi.
\end{align*}
We thus deduce, for such $t$ and $x$,
\begin{align*}
\Big|w(t,x)\Big|&\le de^{K_1(M)(t-t_x)}\|\phi_1-\phi_2\|_{C(
\overline\Omega_s)}\int_{t_x}^te^{K_1(M)(\xi-t_x)}d\xi\\
&\le de^{K_1(M)s}\|\phi_1-\phi_2\|_{C(\overline \Omega_s)}\cdot
(t-t_x)e^{K_1(M)(t-t_x)}\\
&\le s d\, e^{2K_1(M)s}\|\phi_1-\phi_2\|_{C(\overline\Omega_s)}.
\end{align*}
Hence
$$
\|\Gamma_s \phi_1-\Gamma_s\phi_2\|_{C(\overline\Omega_s)} =\|w\|_{C(\overline\Omega_s)}\le\frac 12\|\phi_1-\phi_2\|_{C(\overline
\Omega_s)}~\text{ for }~s\in (0, \delta],
$$
provided that $\delta\in (0,\delta_1]$ satisfies
\[
\delta d\, e^{2K_1(M)\delta}\leq \frac 12.
\]
For such $s$ we may now apply the Contraction Mapping Theorem to conclude that $\Gamma_s$ has a unique fixed point $V$ in $\mathbb X_s^M$. It follows that $v=V$ solves \eqref{201} for $0<t\leq s$.

If we can show that any solution $v$ of \eqref{201} must satisfy $0\leq v\leq M$ in $\Omega_s$, then $v$ would coincide with the unique fixed point $V$ of $\Gamma_s$ in $\mathbb X_s^M$, and uniqueness of the local solution to \eqref{201} is proved.

 We next prove such an estimate for $v$.
We note that $v\geq 0$ already follows from \eqref{V-bd}. So we only need to prove $v\leq M$. We actually prove the following stronger inequality
\begin{equation}\label{v-upper-bd}
v(t,x)\leq M_0:=\max\big\{\|u_0\|_\infty,\; K_0\big\}<M \mbox{ for } t\in [t_x, s],\; x\in (g(s), h(s)).
\end{equation}
It suffices to show that the above inequality holds with $M_0$ replaced by $M_0+\epsilon$ for any given $\epsilon>0$.
We argue by contradiction. Suppose this is not true. Then due to
$v(t_x, x)=\tilde u_0(x)\leq \|u_0\|_\infty<M_\epsilon:=M_0+\epsilon$, there exists some  $t^*\in (t_x, s]$ and $x^*\in (g(s), h(s))$  such that
\[
v(t^*, x^*)=M_\epsilon \mbox{ and }
0\leq v(t,x)<M_\epsilon \mbox{ for } t\in [t_x, t^*),\; x\in (g(s), h(s)).
\]
It follows that $v_t(t^*, x^*)\geq 0$ and $f(t^*, x^*, v(t^*, x^*))\leq 0$.
Hence from \eqref{201} we obtain
\[
0\leq v_t(t^*,x^*)\leq d\int_{g(t^*)}^{h(t^*)}J(x^*-y) v(t^*,y)dy-d v(t^*, x^*).
\]
Since $v(t^*,g(t^*))=v(t^*, h(t^*))=0$, for $y\in (g(t^*), h(t^*))$ but close to the boundary of this interval, $v(t^*,y)<M_\epsilon$.
It follows that
\[
dM_\epsilon=dv(t^*,x^*)\leq d\int_{g(t^*)}^{h(t^*)}J(x^*-y) v(t^*,y)dy<dM_\epsilon\int_{g(t^*)}^{h(t^*)}J(x^*-y)dy\leq dM_\epsilon.
\]
This contradiction proves \eqref{v-upper-bd}.
 Thus $v$ satisfies the wanted inequality and hence coincides with the unique fixed point
of $\Gamma_s$ in $\mathbb X_s^M$. We have now proved the fact that for every $s\in (0,\delta]$, $\Gamma_s$ has a unique fixed point
in $\mathbb X_s$, which is the unique solution to \eqref{201} with $T$ replaced by $s$.
\smallskip

\noindent
{\bf Step 3:} {\it Extension and completion of the proof.}

From Step 2 we know that \eqref{201} has a unique solution defined for $t\in [0, s]$ with $s\in (0, \delta]$. Applying Step 2
to \eqref{201} but with the initial time $t=0$ replaced by $t=s$ we see that the unique solution can be extended to a slightly bigger interval
of $t$. Moreover, by \eqref{v-upper-bd} and the definition of $\delta$ in Step 2, we see that the new extension can be done by
increasing $t$ by at least some $\tilde\delta>0$, with $\tilde\delta$ depends only on $M_0$ and $d$. Furthermore, from the above proof of
\eqref{v-upper-bd} we easily see that the extended solution $v$ satisfies \eqref{v-upper-bd} in the newly extended range of $t$.
Thus the extension by $\tilde \delta$ for $t$ can be repeated. Clearly by repeating this process finitely many times, the solution of
\eqref{201} will be uniquely extended to $t\in [t_x, T)$. As explained above, now
 \eqref{v-upper-bd} holds for $t\in [t_x, T)$, and hence to prove \eqref{v-bound}, it only remains to show $V_{g,h}(t,x)>0$ for $t\in (0, T)$ and $x\in (g(t), h(t))$. However, due to {\rm\bf (f1)-(f2)} and \eqref{v-upper-bd}, we may write $f(t,x,V_{g,h}(t,x))=c(t,x)V_{g,h}(t,x)$ with $c\in L^\infty( \Omega_s)$ for any $s\in (0, T)$. Thus we can use the maximum principle Theorem 2.1 to conclude.
\hfill $\Box$

\subsection{Proof of Theorem \ref{Thm22}}

By Lemma \ref{Lemma202},
for any $T>0$ and $(h,g)\in\mathbb{G}_{h_0,T}\times\mathbb H_{h_0,T}$, we can find a
unique $V_{g,h}\in\mathbb{X}_{u_0,g,h}$ that solves (\ref{201}), and it has the property
$$
0<V_{g,h}(t,x)\le M_0:=\max\big\{\|u_0\|_\infty,~K_0\big\} \mbox{ for  } (t,x)\in
\Omega_{g,h}.
$$

Using such a $V_{g,h}(t,x)$, we define the mapping $\tilde \Gamma$ by $\tilde \Gamma (g,h)=
\left(\tilde g,\,\tilde h\right)$, where, for $0<t\leq T$,
\begin{equation}
\left\{
\begin{aligned}
&\tilde h(t)=h_0+\mu\int_0^t\int_{g(\tau)}^{h(\tau)}\int_{h(\tau)
}^{+\infty}J(y-x)V_{g,h}(\tau,x)dydxd\tau,\\
&\tilde g(t)=-h_0-\mu\int_0^t\int_{g(\tau)}^{h(\tau)}\int_{-\infty
}^{g(\tau)}J(y-x)V_{g,h}(\tau,x)dydxd\tau.
\end{aligned}
\right.
\label{2003}
\end{equation}
 To stress the dependence on $T$, we will write
\[
\mathbb G_T=\mathbb G_{h_0,T},\; \mathbb H_T=\mathbb H_{h_0,T},\; \Omega_T=\Omega_{g,h},\; \mathbb X_T=\mathbb X_{u_0, g, h}.
\]

To prove this theorem, we will show
that if $T$ is small enough, then $\tilde\Gamma$ maps  a suitable closed subset $\Sigma_T$ of $\mathbb{G}_T\times\mathbb{H}_T$ into itself,
and  is a contraction mapping. This clearly implies that $\tilde \Gamma$ has a unique fixed point in $\Sigma_T$,
which gives a solution $(V_{g,h}, g, h)$ of \eqref{101} defined for $t\in (0, T]$. We will show that any solution $(u, g,h)$ of \eqref{101}  with  $(g,h)\in \mathbb{G}_T\times\mathbb{H}_T$ must satisfy $(g,h)\in \Sigma_T$, and hence $(g,h)$ must coincide with the unique fixed point of $\tilde\Gamma$ in $\Sigma_T$, which then implies that  $(u,g,h)=(V_{g,h}, g,h)$ is the unique solution of \eqref{101}.  

We will finally  show that this unique solution
defined locally in time can be extended uniquely for all $t>0$.
\medskip

This plan is carried out below in four steps.
\medskip

\noindent
{\bf Step 1:} {\it Properties of $(\tilde g, \tilde h)$ and a closed subset of $\mathbb{G}_T\times\mathbb{H}_T$.}

Let $(g,h)\in \mathbb{G}_T\times\mathbb{H}_T$.
The definitions of $\tilde h(t)$ and
$\tilde g(t)$ indicate that they belong to $C^1([0, T])$  and for $0<t\le T$,
\begin{equation}
\left\{
\begin{aligned}
&\tilde h'(t)= \mu \int_{g(t)}^{h(t)}\int_{h(t)
}^{+\infty}J(y-x)dyV_{g,h}(t,x)dx,\\
&\tilde g'(t)=-\mu\int_{g(t)}^{h(t)}\int_{-\infty
}^{g(t)}J(y-x)dyV_{g,h}(t,x)dx.
\end{aligned}
\right.
\label{2006}
\end{equation}
These identities already imply $\tilde\Gamma (g,h)=(\tilde g, \tilde h)\in  \mathbb{G}_T\times\mathbb{H}_T$, but in order to show $\tilde \Gamma$ is a contraction mapping, we need to
prove some further properties of $\tilde g$ and $\tilde h$, and then choose a suitable closed subset of  $\mathbb{G}_T\times\mathbb{H}_T$, which
is invariant under $\tilde\Gamma$, and on which $\tilde \Gamma$ is a contraction mapping.

Since  $v=V_{g,h}$ solves (\ref{201})  we  obtain by using {\bf (f1)-(f2)} and  \eqref{v-bound} that
\begin{equation}
\left\{
\begin{aligned}
&\left(V_{g,h}\right)_t(t,x)\ge-dV_{g,h}(t,x)-K(M_0)V_{g,h}(t,x), & &0<t\le T,~x\in(g(t),h(t)),\\
&V_{g,h}(t,h(t))=V_{g,h}(t,g(t))=0,& &0<t\le T,\\
&V_{g,h}(0,x)=u_0(x),& &x\in[-h_0,h_0].
\end{aligned}
\right.
\label{2007}
\end{equation}
It follows that
\begin{equation}
\label{V>cu_0}
V_{g,h}(t,x)\ge e^{-(d+K(M_0))t}u_0(x)\ge e^{-(d+K(M_0))T}u_0(x) \mbox{ for } x\in [-h_0, h_0],\; t\in (0, T].
\end{equation}
By \textbf{(J)}  there
exist constants $\epsilon_0\in (0, h_0/4)$ and $\delta_0>0$ such that
\begin{equation}
\label{J>delta_0}
\mbox{$J(z)
\ge\delta_0$ if $|z|\le\epsilon_0$.}
\end{equation}
Using \eqref{2006} we easily see
\[
[\tilde h(t)-\tilde g(t)]'\leq \mu M_0 [h(t)-g(t)] \mbox{ for } t\in [0, T].
\]
We now assume  that $(g, h)$ has the extra property that
\[\mbox{
    $h(T)-g(T)\leq 2h_0+\frac{\epsilon_0}{4}$.}
    \]
    Then
\[
\tilde h(t)-\tilde g(t)\leq 2h_0 +T\mu M_0(2h_0+\frac{\epsilon_0}{4})\leq 2h_0+\frac{\epsilon_0}{4} \mbox{ for } t\in [0, T],
\]
provided that $T>0$ is small enough, depending on $ (\mu, M_0, h_0, \epsilon_0)$.
We fix such a  $T$ and notice from the above extra assumption on $(g, h)$  that
\[
h(t)\in [h_0,  h_0+ \frac{\epsilon_0}{4}],\; g(t)\in [-h_0-\frac{\epsilon_0}{4}, -h_0]  \mbox{ for } t\in [0, T].
\]
Combining this with \eqref{V>cu_0} and \eqref{J>delta_0} we obtain, for such $T$ and $t\in (0, T]$,
\begin{align*}
\int_{g(t)}^{h(t)}\int_{h(t)}^{+\infty}J(y-x)V_{g,h}
(t,x)dydx&\ge\int_{h(t)-\frac{\epsilon_0}{2}}^{h(t)}\int_{
h(t)}^{h(t)+\frac{\epsilon_0}{2}}J(y-x)V_{g,h}(t,x)dydx\\
&\ge e^{-(d+K(M_0))T}\int_{h_0-\frac{\epsilon_0}{4}}^{h_0}\int_{
h_0+\frac{\epsilon_0}{4}}^{h_0+\frac{\epsilon_0}{2}}J(y-x)dyu_0(x)dx\\
&\ge\frac 14\epsilon_0\delta_0e^{-(d+K(M_0))T}\int_{h_0-\frac{
\epsilon_0}{4}}^{h_0}u_0(x)dx=:c_0>0,
\end{align*}
with $c_0$ depending only on $(J, u_0, f)$. Thus, for sufficiently small $T=T(\mu, M_0, h_0, \epsilon_0)>0$,
\begin{equation}
\label{tilde-h'}
\tilde h'(t)\geq \mu c_0 \mbox{ for } t\in [0, T].
\end{equation}
We can similarly obtain, for such $T$,
\begin{equation}
\label{tilde-g'}
\tilde g'(t)\leq -\mu \tilde c_0 \mbox{ for } t\in [0, T],
\end{equation}
for some positive constant $\tilde c_0$ depending only on $(J, u_0, f)$.

We now define, for $s\in (0, T_0]:=(0, T(\mu, M_0, h_0, \epsilon_0)]$,
\begin{align*}
\Sigma_s:=&\Big\{(g,h)\in\mathbb G_s\times \mathbb H_s: \sup_{0\leq t_1<t_2\leq s}\frac{g(t_2)-g(t_1)}{t_2-t_1}\leq -\mu \tilde c_0,\;
\inf_{0\leq t_1<t_2\leq s}\frac{h(t_2)-h(t_1)}{t_2-t_1}\geq \mu c_0,\\
& \hspace{7.5cm} h(t)-g(t)\leq 2h_0+\frac{\epsilon_0}{4} \mbox{ for } t\in [0, s]\Big\}.
\end{align*}
Our analysis above shows that
\[
\tilde \Gamma (\Sigma_s)\subset \Sigma_s \mbox{ for } s\in (0, T_0].
\]

\medskip

\noindent {\bf Step 2:} {\it $\tilde\Gamma$ is a
contraction mapping on $\Sigma_s$ for sufficiently small $s>0$.}
\medskip

Let $s\in (0, T_0]$, $(h_1,g_1), (h_2,g_2)\in\Sigma_s$,
 and note that $\Sigma_s$ is a complete metric space
under the metric
$$
d\left((h_1,g_1),(h_2,g_2)\right)=\|h_1-h_2\|_{C([0,s])}+
\|g_1-g_2\|_{C([0,s])}.
$$

For $i=1,2$, let us denote
\[
\mbox{
 $V_i(t,x):=V_{h_i,g_i}(t,x)$ and
$\tilde\Gamma\left(h_i,g_i\right):=\left(\tilde h_i,\tilde g_i
\right)$. }
\]
We also define
\begin{align*}
&H_{min}(t):=\min\left\{h_1(t),~h_2(t)\right\},~~~H_{max}(t):=\max\left\{
h_1(t),~h_2(t)\right\},\\
&G_{min}(t):=\min\left\{g_1(t),~g_2(t)\right\},~~~~G_{max}(t):=\max\left\{
g_1(t),~g_2(t)\right\},\\
&\Omega_{*s}=\Omega_{G_{min}, H_{max}}:=\Omega_{g_1, h_1}\cup \Omega_{g_2, h_2}.
\end{align*}
For $t\in [0,s]$, we have
\[
2h_0\leq H_{max}(t)-G_{min}(t)\leq 2h_0+\epsilon_0\leq 3h_0,
\]
and
\begin{align*}
~&\left|\tilde h_1(t)-\tilde h_2(t)\right|
\\
\le~&\mu\int_0^t\left|\int_{g_1(\tau)}^{h_1(\tau)}\int_{h_1(\tau)
}^{+\infty}J(y-x)V_1(\tau,x)dydxd\tau-\int_{g_2(\tau)}^{h_2(\tau)}
\int_{h_2(\tau)}^{+\infty}J(y-x)V_2(\tau,x)dydx\right|d\tau
\\
\le~&\mu\int_0^t\int_{g_1(\tau)}^{h_1(\tau)}\int_{h_1(\tau)
}^{+\infty}J(y-x)\Big|V_1(\tau,x)-V_2(\tau,x)\Big|dydxd\tau
\\
&~+\mu\int_0^t\left|\left(\int_{h_1(\tau)}^{h_2(\tau)}\int_{h_1(\tau)
}^{+\infty}+\int_{g_2(\tau)}^{g_1(\tau)}\int_{h_1(\tau)
}^{+\infty}+\int_{g_2(\tau)}^{h_2(\tau)}\int_{h_1(\tau)}^{h_2(\tau)}\right)
J(y-x)V_2(t,x)dydx\right|d\tau
\\
\le~&3h_0\mu\|V_1-V_2\|_{C(\overline\Omega_{*s})}s+\mu M_0\big(1+3h_0\|J\|_\infty\big)\|h_1-h_2\|_{C([0,s])}s
+\mu M_0\|g_1-g_2\|_{C([0,s])}s\\
\le ~ &C_0s\Big[\|V_1-V_2\|_{C(\overline\Omega_{*s})}+\|h_1-h_2\|_{C([0,s])}+\|g_1-g_2\|_{C([0,s])}\Big],
\end{align*}
where $C_0$ depends only on $(\mu, u_0, J, f)$. Let us recall that $V_i$ is always extended by 0 in $\big([0,\infty)\times \mathbb R\big)\setminus \Omega_{g_i, h_i}$ for $i=1,2$.

Similarly, we have, for $t\in [0, s]$,
\begin{align*}
\Big|\tilde g_1(t)-\tilde g_2(t)\Big|\le C_0s\Big[\|V_1-V_2\|_{C(\overline\Omega_s)}+\|h_1-h_2\|_{C([0,s])}+\|g_1-g_2\|_{C([0,s])}\Big].
\end{align*}
Therefore,
\begin{equation}
\begin{aligned}
&~\|\tilde h_1-\tilde h_2\|_{C([0,s])}+\|\tilde g_1-\tilde g_2
\|_{C([0,s])}\\
\le&~2C_0s\Big[\|V_1-V_2\|_{C(\overline\Omega_{*s})}+\|h_1-h_2\|_{C([0,s])}+\|g_1-g_2\|_{C([0,s])}\Big].
\end{aligned}
\label{20013}
\end{equation}

\medskip

Next, we  estimate $\|V_1-V_2\|_{C(\overline\Omega_{*s})}$. We denote $U=V_1-V_2$, and
for fixed $(t^*,x^*)\in\Omega_{*s}$,  we consider
three cases separately.

\smallskip
\noindent
\underline{Case 1.}  $x^*\in[-h_0, h_0]$.

It
follows from the equations satisfied by $V_1$ and $V_2$ that
$U(0,x^*)=0$ and for $0<t\leq s$,
\begin{equation}
U_t(t,x^*)+c_1(t,x^*)U(t,x^*)=A(t,x^*),
\label{2008}
\end{equation}
where
\begin{align*}
&c_1(t,x^*):=d-\frac{f(t,x^*,V_1(t,x^*))-f(t,x^*,V_2(t,x^*))}
{V_1(t,x^*)-V_2(t,x^*)} \mbox{ and so } \|c_1\|_\infty\leq d+K(M_0),\\
&A(t,x^*):=d\int_{g_1(t)}^{h_1(t)}J(x^*-y)V_1(t,y)dy
-d\int_{g_2(t)}^{h_2(t)}J(x^*-y)V_2(t,y)dy.
\end{align*}
Thus
\begin{align*}
U(t^*,x^*)
=e^{-\int_0^{t^*}c_1(\tau,x^*)d\tau}
\int_0^{t^*}e^{\int_0^tc_1(\tau,x^*)d\tau}A(t,x^*)dt.
\end{align*}
We have
\begin{align*}
\Big|A(t,x^*)\Big|&=d\left|\int_{g_1(t)}^{h_1(t)}J(x^*-y)
V_1(t,y)dy-\int_{g_2(t)}^{h_2(t)}J(x^*-y)V_2(t,y)dy\right|\\
&\le d\int_{g_1(t)}^{h_1(t)}J(x^*-y)\big|V_1(t,y)-V_2(t,y)\big|dy
+d\left|\left(\int_{g_2(t)}^{g_1(t)}+\int_{h_1(t)}^{h_2(t)}\right)J(x^*-y)V_2(t,y)dy\right|\\
&\le d\|U\|_{C(\overline \Omega_{*s})}+d\|J\|_\infty M_0\left[\|h_1-h_2\|_{C([0,s])}+\|g_1-g_2\|_{C([0,s])}\right].
\end{align*}
Thus for some $C_1>0$ depending only on
$(d,u_0, M_0, J)$, we have
\begin{equation}
\max_{t\in [0,s]}\Big|A(t,x^*)\Big|
\le C_1 \left(\|U\|_{C(\overline \Omega_{*s})}+\|h_1
-h_2\|_{C([0,s])}+\|g_1-g_2\|_{C([0,s])}\right).
\label{A}
\end{equation}
It follows that
\begin{equation}
\Big|U(t^*,x^*)\Big|
\le C_1s\, e^{2(d+K(M_0))s} \left(\|U\|_{C(\overline \Omega_{*s})}+\|h_1
-h_2\|_{C([0,s])}+\|g_1-g_2\|_{C([0,s])}\right).
\label{20010}
\end{equation}
\smallskip

 \noindent
  \underline{Case 2.}   $x^*\in(h_0, H_{min}(s))$.

  In this case there exist $t_1^*,\, t_2^*\in (0, t^*)$ such
that $x^*=h_1(t_1^*)=h_2(t_2^*)$. Without loss of generality, we may assume that $0<t_1^*
\leq t_2^*$.  Now  we use (\ref{2008}) for $t\in [t_2^*, t^*]$, and obtain
\begin{align*}
U(t^*,x^*)=e^{-\int_{t_2^*}^{t^*}c_1(\tau,x^*)d\tau}\left[U(
t_2^*,x^*)+\int_{t_2^*}^{t^*}e^{\int_{t_2^*}^tc_1(\tau,x^*)d\tau}
A(t,x^*)dt\right].
\end{align*}
It follows that
\begin{equation}
\label{U*}
\begin{aligned}
\Big|U(t^*,x^*)\Big|&\le e^{(d+K(M_0))t^*}\left[\Big|U(t_2^*,x^*)
\Big|+\int_{t_2^*}^{t^*}e^{(d+K(M_0))t}\Big|A(t,x^*)\Big|dt\right]\\
&\le e^{(d+K(M_0))s}\Big|U(t_2^*,x^*)\Big|+s e^{2(d+K(M_0))s}\max_{t\in [0, s]}|A(t,x^*)|.
\end{aligned}
\end{equation}
Since $V_1(t_1^*,x^*)=V_2(t_2^*, x^*)=0$, we have
\begin{align*}
U(t_2^*,x^*)=V_1(t_2^*,x^*)-V_1(t_1^*,x^*)=\int_{t_1^*}^{t_2^*}(V_1)_t(t,x^*)dt,
\end{align*}
and hence from the equation satisfied by $V_1$ we obtain
\begin{align*}
\Big|U(t_2^*,x^*)\Big|&\le\int_{t_1^*}^{t_2^*}\left|d\int_{g_1
(t)}^{h_1(t)}J(x^*-y)V_1(t,y)dy
-dV_1(t,x^*)+f(t,x^*,V_1(t,x^*))\right|dt\\
&\le C_2\Big(t_2^*-t_1^*\Big), \mbox{\;\; for some $C_2>0$ depending only on $(d, M_0, f)$}.
\end{align*}
If $t_1^*=t_2^*$ then clearly $U(t_2^*, x^*)=0$. If $t_1^*<t_2^*$, then
using $\frac{h_1(t_2^*)-h_1(t_1^*)}
{t_2^*-t_1^*}\ge \mu c_0$ we obtain
\[
t_2^*-t_1^*\le\Big|h_1(t_2^*)-h_1(t_1^*)\Big|(\mu c_0)^{-1}.
\]
Since
$$
0=h_1(t_1^*)-h_2(t_2^*)=h_1(t_1^*)-h_1(t_2^*)+h_1(t_2^*)-h_2(t_2^*),
$$
we have $h_1(t_2^*)-h_1(t_1^*)=h_1(t_2^*)-h_2(t_2^*)$, and thus
$$
t_2^*-t_1^*\le\Big|h_1(t_2^*)-h_1(t_1^*)\Big|(\mu c_0)^{-1}=
\Big|h_1(t_2^*)-h_2(t_2^*)\Big|(\mu c_0)^{-1}.
$$
Therefore there exists some positive constant $C_3=C_3(\mu c_0,C_2)$
such that
$$
\Big|U(t_2^*,x^*)\Big|\le C_3\|h_1-h_2\|_{C([0,s])}.
$$
Substituting this and \eqref{A} proved in Case 1 above to \eqref{U*}, we obtain
\begin{equation}
\begin{aligned}
&\Big|U(t^*,x^*)\Big|
\le\ e^{(d+K(M_0))s}C_3\|h_1-h_2\|_{C( [0,s])} \\
& \ \ \ \ \  + C_1 s e^{2(d+K(M_0))s} \left(\|U\|_{C(\overline \Omega_{*s})}+\|h_1
-h_2\|_{C([0,s])}+\|g_1-g_2\|_{C([0,s])}\right).
\end{aligned}
\label{2009}
\end{equation}

\smallskip

\noindent
\underline{Case 3.}  $x^*\in [H_{min}(s), H_{max}(s))$.

Without loss of generality we assume that $h_1(s)<h_2(s)$. Then $H_1(s)=h_1(s),\; H_2(s)=h_2(s)$ and
\[
h_1(t^*)\leq h_1(s)<x^*<H_2(t^*) =h_2(t^*),
\]
\[
\mbox{$V_1(t, x^*)=0$ for $t\in [t_2^*, t^*]$,\;  $0<h_2(t^*)-h_2(t_2^*)\leq h_2(t^*)-h_1(t^*)$}.
\]
We have
\begin{align*}
0<V_2(t^*, x^*)&=\int_{t_2^*}^{t^*} \left[d\int_{g_2(t)}^{h_2(t)}J(x^*-y)V_2(t, y)dy-dV_2(t, x^*)+f(t, x^*, V_2(t, x^*))\right]dt\\
&\leq (t^*-t_2^*)\big[d+K(M_0)\big]M_0\\
&\leq \big[h_2(t^*)-h_2(t_2^*)\big](\mu c_0)^{-1}\big[d+K(M_0)\big]M_0\\
&\leq (\mu c_0)^{-1}\big[d+K(M_0)\big]M_0 \big[h_2(t^*)-h_1(t^*)\big]\\
&\leq C_4\|h_1-h_2\|_{C([0,s])},
\end{align*}
with $C_4:= (\mu c_0)^{-1}\big[d+K(M_0)\big]M_0$.

 We thus obtain
\begin{equation}
|U(t^*, x^*)|=V_2(t^*,x^*)\leq C_4\|h_1-h_2\|_{C([0,s])}.
\label{20011}
\end{equation}

The inequalities (\ref{20010}),
(\ref{2009}) and (\ref{20011}) indicate that, there exists $C_5>0$ depending only on $(\mu c_0, d, u_0, J, f)$ such that,
whether we are in Cases 1, 2 or 3, we always have
\begin{equation}
|U(t^*, x^*)|\leq C_5 \left(\|U\|_{C(\overline \Omega_{*s})}s+\|h_1
-h_2\|_{C([0,s])}+\|g_1-g_2\|_{C([0,s])}\right).
\label{20012}
\end{equation}

Analogously, we can examine the cases $x^*\in (G_2(s), -h_0)$ and $x^*\in (G_1(s), G_2(s)]$ to obtain a constant $C_6>0$ depending only on
$(\mu \tilde c_0, d, u_0, J, f)$ such that \eqref{20012} holds with $C_5$ replaced by $C_6$. Setting $C^*:=\max\big\{C_5, C_6\big\}$, we thus obtain
\[
|U(t^*, x^*)|\leq C^* \left(\|U\|_{C(\overline \Omega_{*s})}s+\|h_1
-h_2\|_{C([0,s])}+\|g_1-g_2\|_{C([0,s])}\right) \mbox{ for all } (t^*, x^*)\in\Omega_{*s}.
\]
It follows that
\[
\|U\|_{C(\overline \Omega_{*s})}\leq  C^* \left(\|U\|_{C(\overline \Omega_{*s})}s+\|h_1
-h_2\|_{C([0,s])}+\|g_1-g_2\|_{C([0,s])}\right).
\]
Let us recall that the above inequality holds for all $s\in (0, T_0]$ with $T_0$ given near the end of Step 1. Set $T_1:=\min\Big\{T_0,\; \frac{1}{2C^*}\Big\}$.
Then we easily deduce
\[
\|U\|_{C(\overline \Omega_{*s})}\leq  2C^* \left(\|h_1
-h_2\|_{C([0,s])}+\|g_1-g_2\|_{C([0,s])}\right)\; \mbox{ for } s\in (0, T_1].
\]
Substituting this inequality into (\ref{20013}) we obtain, for $s\in (0, T_1]$,
\begin{align*}
&~\|\tilde h_1-\tilde h_2\|_{C([0,s])}+\|\tilde g_1-\tilde g_2
\|_{C([0,s])}\\
\le&~ 2C_0(2C^*+1)s\left[\|h_1-h_2\|_{C([0,s])}
+\|g_1-g_2\|_{C([0,s])}\right].
\end{align*}
Thus if we define
$T_2$  by $2C_0(2C^*+1)T_2=\frac 12$, and $T^*:=\min\big\{T_1, T_2\big\}$, then
\begin{align*}
\|\tilde h_1-\tilde h_2\|_{C([0,T^*])}+\|\tilde g_1-\tilde g_2
\|_{C([0,T^*])}\le\frac 12\left[\|h_1-h_2\|_{C([0,T^*])}
+\|g_1-g_2\|_{C([0,T^*])}\right],
\end{align*}
 i.e., $\tilde\Gamma$
is a contraction mapping on $\Sigma_{T^*}$.

\medskip

\noindent
{\bf Step 3:} {\it Local existence and uniqueness.}

By Step 2 and the Contraction Mapping Theorem we know that \eqref{101} has a solution $(u, g,h)$ for $t\in (0, T^*]$. If we can show that
$(g,h)\in \Sigma_{T^*}$  holds for any solution $(u,g,h)$ of \eqref{101} defined over $t\in (0, T^*]$, then it is the unique fixed point of $\tilde \Gamma$ in $\Sigma_{T^*}$ and the uniqueness of $(u,g,h)$ follows.

So let $(u,g,h)$ be an arbitrary solution of \eqref{101} defined for $t\in (0, T^*]$. Then
\[
\left\{
\begin{aligned}
& h'(t)= \mu \int_{g(t)}^{h(t)}\int_{h(t)
}^{+\infty}J(y-x)u (t,x)dydx,\\
& g'(t)=-\mu\int_{g(t)}^{h(t)}\int_{-\infty
}^{g(t)}J(y-x)u(t,x)dydx.
\end{aligned}
\right.
\]
By Lemma \ref{Lemma202}, we have
\[
0<u(t,x)\leq M_0 \mbox{ for } t\in [0, T^*], \; x\in (g(t), h(t)).
\]
It follows that
\[
[h(t)-g(t)]'=\mu\int_{g(t)}^{h(t)}\Big[1-\int_{g(t)}^{h(t)}J(y-x)dy\Big]u(t,x)dx\leq \mu M_0[h(t)-g(t)] \mbox{ for } t\in (0, T^*].
\]
We thus obtain
\begin{equation}
\label{h-g}
h(t)-g(t)\leq 2h_0 e^{\mu M_0 t} \mbox{ for } t\in (0, T^*].
\end{equation}
Therefore if we shrink $T^*$  if necessary so that
\[
2h_0e^{\mu M_0 T^*}\leq 2h_0+\frac{\epsilon_0}{4},
\]
then
\[
h(t)-g(t)\leq 2h_0+\frac{\epsilon_0}{4} \mbox{ for } t\in [0, T^*].
\]
Moreover, the proof of \eqref{tilde-h'} and \eqref{tilde-g'} gives
\[
h'(t)\geq \mu c_0,\; g'(t)\leq -\mu \tilde c_0 \mbox{ for } t\in (0, T^*].
\]
Thus indeed $(g,h)\in\Sigma_{T^*}$, as we wanted.
This proves the local existence and uniqueness of the  solution to \eqref{101}.

\medskip

\noindent
{\bf Step 4:} {\it Global existence and uniqueness.}\medskip

By Step 3, we see the \eqref{101} has a unique solution $(u,g,h)$  for some initial time interval $(0, T)$, and for
any $s\in (0, T)$, $u(s,x)>0$ for $x\in (g(s), h(s))$ and $u(s,\cdot)$ is continuous over $[g(s), h(s)]$.
This implies that we can treat $u(s,\cdot)$ as an initial function and use Step 3 to extend the solution from $t=s$ to some $T'\geq T$.
Suppose $(0, \hat T)$ is the maximal interval that the solution $(u,g,h)$ of \eqref{101} can be defined through this extension process.
We show that $\hat T=\infty$. Otherwise $\hat T\in (0, \infty)$ and we are going to derive a contradiction.\medskip

Firstly we notice that \eqref{h-g} now holds for $t\in (0, \hat T)$. Since $h(t)$ and $g(t)$ are monotone functions over $[0, \hat T)$,
we may define
\[
h(\hat T):=\lim_{t\to\hat T} h(t),\; g(\hat T):=\lim_{t\to\hat T} g(t) \; \mbox{ with } h({\hat T})-g(\hat T)\leq 2h_0e^{\mu M_0 \hat T}.
\]
The third and fourth equations in \eqref{101}, together with $0\leq u\leq M_0$ indicate that $h'$ and $g'$ belong to $L^\infty([0, \hat T))$ and
hence with $g(\hat T)$ and $h(\hat T)$ defined as above, $g, h\in C([0,\hat T])$. It also follows that the right-hand side of the first equation in \eqref{101}
belongs to $L^\infty(\Omega_{\hat T})$, where $\Omega_{\hat T}:=\big\{(t,x): t\in [0, \hat T],\; g(t)< x<h(t)\big\}$.
It follows that $u_t\in L^\infty(\Omega_{\hat T})$. Thus for each $x\in (g(\hat T), h(\hat T))$,
\[
u(\hat T,x):=\lim_{t\nearrow \hat T}u(t,x) \mbox{ exists},
\]
and $u(\cdot, x)$ is continuous at $t=\hat T$. We may now view $u(t,x)$ as the unique solution of the ODE problem in Step 1 of the proof of Lemma \ref{Lemma202}
(with $\phi=u$),
which is defined over $[t_x, \hat T]$. Since $t_x$, $J(x-y)$ and $f(t,x,u)$ are all continuous in $x$, by the continuous dependence of the ODE solution to
the initial function and the parameters in the equation, we see that {\color{red}$u(t,x)$ is continuous in $\Omega_{\hat T}$}.
By assumption, {\color{red}$u\in C(\overline \Omega_s)$ for any $s\in (0, \hat T)$.} To show this also holds with $s=\hat T$, it remains to show that
\[
\mbox{$u(t,x)\to 0$ as $(t,x)\to (\hat T, g(\hat T))$ and as $(t,x)\to (\hat T, h(\hat T))$ from $\Omega_{\hat T}$.}
\]

 We only prove the former as the other case can be shown similarly.
We note that as $x\searrow g(\hat T)$, we have $t_x\nearrow \hat T$, and so
\begin{align*}
|u(t,x)|&=\left|\int_{t_x}^t \left[ d\int_{g(t)}^{h(t)}J(x-y)u(\tau, y)dy-d u(\tau, x)+f(\tau, x, u(\tau,x))\right]d\tau\right|\\
&\leq (t-t_x)\big[2d+K(M_0)\big]M_0\\
&\to 0 \mbox{ as } \Omega_{\hat T} \ni (t,x)\to (\hat T, g(\hat T)).
\end{align*}

Thus we have shown that $u\in C(\overline \Omega_{\hat T})$ and $(u,g,h)$ satisfies \eqref{101} for $t\in (0, \hat T]$. By Lemma 2.2 we have
$u(\hat T, x)>0$ for $x\in (g(\hat T), h(\hat T))$. Thus we can regard $u(\hat T, \cdot)$ as an initial function and apply Step 3 to conclude that
the solution of \eqref{101} can be extended to some $(0, \tilde T)$ with $\tilde T>\hat T$. This contradicts the definition of $\hat T$. Therefore we must have $\hat T=\infty$.
\hfill $\Box$
\medskip

{\bf Remark:} The material in this section is taken from \cite{cdjfa} with some minor variations.

\section{Spreading-vanishing dichotomy and criteria}

We investigate the long-time dynamics of 
\begin{equation}
\left\{
\begin{aligned}
&u_t=d\int_{g(t)}^{h(t)}J(x-y)u(t,y)dy-du+f(u),
& &t>0,~x\in(g(t),h(t)),\\
&u(t,g(t))=u(t,h(t))=0,& &t>0,\\
&h'(t)=\mu\int_{g(t)}^{h(t)}\int_{h(t)}^{+\infty}
J(y-x)u(t,x)dydx,& &t>0,\\
&g'(t)=-\mu\int_{g(t)}^{h(t)}\int_{-\infty}^{g(t)}
J(y-x)u(t,x)dydx,& &t>0,\\
&u(0,x)=u_0(x),~h(0)=-g(0)=h_0,& &x\in[-h_0,h_0],
\end{aligned}
\right.
\label{101'}
\end{equation}
where $d, \mu, h_0$ are given positive constants. The initial
function $u_0(x)$ satisfies \eqref{102}. The kernel function
$J: \mathbb{R}\rightarrow\mathbb{R}$ satisfies the basic condition 
\begin{description}
\item[(J)] $J\in C(\R)\cap L^\infty(\R)$, $J\geq 0$, $J(0)>0,~\int_{\mathbb{R}}J(x)dx=1$.
\end{description}
The growth term $f: \mathbb{R}^+
\rightarrow\mathbb{R}$  satisfies the KPP condition
\begin{description}
\item[${\bf (f_{KPP})}$] $\begin{cases}f\in C^1, \ f(0)=f(1)=0, \ f'(0)>0>f'(1),\\
 \mbox{$f(u)/u$ is non-increasing in $(0,\infty)$.}\end{cases}$
\end{description}

We are going to prove the following two theorems from \cite{cdjfa}.

\begin{theorem}[Spreading-vanishing dichotomy]\label{thm2}
Suppose {\bf (J)} and ${\bf (f_{KPP})}$ hold,  $u_0$ satisfies \eqref{102} and {\color{red}$J$ is symmetric}: $J(x)=J(-x)$. Let $(u,g,h)$ be the unique solution of problem
\eqref{101'}. Then one of the following alternatives must happen for \eqref{101'}:
\begin{itemize}
\item[(i)] \underline{\rm Spreading:}\ \  $\begin{cases}  \lim_{t\to+\infty} (g(t), h(t))=\mathbb R,\\ \lim_{t
\rightarrow+\infty}u(t,x)=1\ \mbox{ locally uniformly  in
$\mathbb{R}$,}\end{cases}$
\item[(ii)] \underline{\rm Vanishing:}\ \  $\begin{cases} \lim_{t\to+\infty} (g(t), h(t))=(g_\infty, h_\infty) \mbox{ is a finite interval},\\
 \lim_{t\rightarrow+\infty}u(t,x)=0 \mbox{ uniformly for $x\in [g(t),h(t)]$.}\end{cases}$
\end{itemize}
\end{theorem}

\begin{theorem}[Spreading-vanishing criteria]\label{thm3}
Under the conditions of Theorem \ref{thm2}, if $d\in (0, f'(0)]$, then spreading always happens. If $d>f'(0)$, then there exists a unique $\ell^*>0$ such that
spreading always happens if $h_0\geq \ell^*/2$; and for $h_0\in (0, \ell^*/2)$, there exists a unique $\mu^*>0$ so that spreading happens exactly when $\mu>\mu^*$.
\end{theorem}

As we will see in the proof, $\ell^*$ depends only on $(f'(0),  d, J)$. On the other hand, $\mu^*$ depends also on $u_0$.

\medskip

\noindent{\bf Extension to weakly non-symmetric kernels}
\smallskip

It turns out that  the symmetry requirement of $J$ in Theorems \ref{thm2} and \ref{thm3} can be significantly relaxed in the above two theorems. For a non-symmetric $J$ satisfying ${\bf (J)}$, the following two quantities determined by $J$ and $f'(0)$ alone play an important role:
\begin{align*}
	\displaystyle	c_*^- = \sup_{\nu<0} \frac{\displaystyle d\int_{\mathbb{R}} J(x)e^{\nu x}\,dx - d + f'(0)}{\nu}, \quad
		c_*^+ = \inf_{\nu>0} \frac{\displaystyle d\int_{\mathbb{R}} J(x)e^{\nu x}\,dx - d + f'(0)}{\nu},
	\end{align*}
	
	 It can be shown that $c_*^-$ is achieved by some $\nu<0$ when it is finite, and a parallel conclusion holds for $c_*^+$.
	 It is easily checked that $c_*^-$ is finite if and only if $J$ satisfies additionally the following {\color{red}thin-tail} condition at $x=-\infty$,\smallskip

${\bf (J_{thin}^-):}$ There exists $\lambda>0$ such that $\displaystyle \int_{0}^{+\infty}J(-x)e^{\lambda x}\,dx<+\infty$.
\medskip
	
	\noindent Similarly,  $c_*^+$ is finite if and only if $J$ satisfies\smallskip
	
	${\bf (J_{thin}^+):}$ There exists $\lambda>0$ such that $\displaystyle \int_{0}^{+\infty}J(x)e^{\lambda x}\,dx<+\infty$.

\medskip

If we define
\begin{equation}\label{J-thin-not}\begin{cases}
c_*^-=-\infty \mbox{ when ${\bf (J_{thin}^-)}$ does not hold},\\[3mm]
 c_*^+=+\infty \mbox{ when ${\bf (J_{thin}^+)}$ does not hold},
 \end{cases}
\end{equation}
then  the propagation dynamics of the corresponding Cauchy problem of \eqref{101},
\begin{equation}
\label{cau3}
	\left\{
	\begin{array}{ll}
		U_t = \displaystyle d \int_{\mathbb{R}} J(x-y) U(t,y) \, dy - d U(t,x) + f(U), & t > 0, \; x \in \mathbb{R}, \\
		U(0, x) =U_0(x)
	\end{array}
	\right.
\end{equation}
has the properties described in the following result:

\noindent
{\bf Theorem A.}(\cite{du25}) {\it
Suppose that ${\bf(J)}$ and ${\bf (f_{KPP})}$ hold. Then for any initial function $U_0(x)$ which is continuous and nonnegative with non-empty compact support, the unique solution $U(t,x)$ of \eqref{cau3} satisfies
\begin{align*}
	 \lim_{t\to \infty} U(t,x)=\begin{cases} 1 \mbox{ uniformly for } x\in [a_1t, b_1t] \mbox{ provided that } [a_1,b_1]\subset (c_*^-, c_*^+),\\[2mm]
	 0 \mbox{ uniformly for $x\leq a_2t$ provided that $c_*^->-\infty$ and $a_2<c_*^-$},\\[3mm]
	 0 \mbox{ uniformly for $x\geq b_2t$ provided that $c_*^+<\infty$ and $b_2>c_*^+$}.
	 \end{cases}
	\end{align*}
}

Following \cite{aw}, the conclusions in Theorem A can be interpreted as indicating a leftward spreading speed of $c_*^-$ and rightward spreading speed of $c_*^+$ for \eqref{cau3}.
The following  result of Yagisita \cite{yagisita}  (see also Theorem 1.5 in \cite{coville2}) on traveling waves  provides further meanings for $c_*^-$ and $c_*^+$.  \smallskip

\noindent
{\bf Theorem B.} (\cite{yagisita}) {\it
Suppose that ${\bf(J)}$ and ${\bf (f_{KPP})}$ are satisfied. Then the following conclusions hold.
\begin{itemize}
	\item[{\rm (i)}] The \underline{rightward} traveling wave problem
	\begin{equation}
		\label{e6-0}
		\begin{cases}
			\displaystyle d\int_{\mathbb{R}} J(x-y)\phi(y)\,dy - d\phi(x) + c\phi'(x) + f(\phi(x)) = 0, & x \in \mathbb{R}, \\[3mm]
			\phi(-\infty) = 1, \quad \phi(+\infty) = 0
		\end{cases}
	\end{equation}
	has a solution pair $(c,\phi) \in \mathbb R\times L^{\infty}(\mathbb{R})$ with $\phi$ nonincreasing if and only if $c_*^+<\infty$. Moreover, in such a case, for every $c\geq c_*^+$,  \eqref{e6-0} has a solution $\phi \in C^1(\mathbb{R})$ that is strictly decreasing, and  \eqref{e6-0} has no such solution for $c<c_*^+$.

	\item[{\rm (ii)}]  The \underline{leftward} traveling wave problem
	\begin{equation}
		\label{e7-}
		\left\{
		\begin{array}{ll}
			\displaystyle d\int_{\mathbb{R}}J(x-y)\psi(y)\,dy-d\psi(x)-c\psi'(x)+f(\psi(x))=0, & x\in\mathbb{R}, \\
			\psi(-\infty)=0, \ \ \psi(+\infty)=1,
		\end{array}
		\right.
	\end{equation}
	has a solution pair $(c,\psi)\in \mathbb{R}\times L^{\infty}(\mathbb{R})$ with $\psi$ nondecreasing if and only if $c_*^->-\infty$. Moreover, in such a case, for each $c\geq -c_*^-$,  \eqref{e7-} has a solution $\psi \in C^1(\mathbb{R})$ that is strictly increasing, and  \eqref{e7-} has no such solution for $c<-c_*^-$.
	\end{itemize}
}

Problem \eqref{cau3} and its many variations have been extensively studied in the literature; see, for example, \cite{BCV, coville2, DJ, XLR} and the references therein as a small sample of these works. It can be shown as in \cite{du21} that \eqref{cau3} is the limiting problem of \eqref{101} when $\mu\to\infty$.
\medskip

\noindent
{\bf Definition:} For a kernel function $J$ satisfying {\bf (J)} we say it is {\bf weakly non-symmetric} if 
\begin{equation}\label{weak-}
-\infty\leq c_*^-<0<c_*^+\leq \infty.
\end{equation}

\begin{theorem}\label{thm4.3}
Theorems \ref{thm2} and \ref{thm3} remain valid if $J(x)$ is weakly non-symmetric.
\end{theorem}

{\bf Remark:} If $J(x)$ is not weakly non-symmetric, then fundamental differences arise in the long-time behaviour of \eqref{101}; such a case was considered in \cite{DFN1}.

\subsection{The associated problem over a fixed spatial interval}
For $c\in\mathbb R$ and $\Omega=(l_1,l_2)$  a bounded interval,
define
\begin{align*}
	\mathcal{L}^c_{\Omega}[\phi](x) := d\int_{\Omega}J(x-y)\phi(y)\,dy - d\phi(x) + c\phi'(x) + f'(0)\phi(x),\quad \phi \in C^1(\Omega)\cap C(\overline{\Omega}).
\end{align*}
It is known \cite{li, coville20} that
$$\lambda_p(\mathcal{L}^c_\Omega):=\inf\{\lambda\in \mathbb{R}:\ \mathcal{L}^c_\Omega[\phi]\leq \lambda\phi, \,\phi>0\,\text{in}\,\,\Omega
\,\,\text{for some}\,\,\phi\in C(\bar{\Omega})\}$$
is a principal eigenvalue of $\mathcal{L}^c_\Omega$, which corresponds to a positive eigenfunction. From the definition it is easily seen that
\[
\lambda_p(\mathcal{L}^c_{(l_1, l_2)})=\lambda_p(\mathcal{L}^c_{(0, l_2-l_1)}).
\]
Moreover, the following conclusions hold:
\begin{proposition}[\cite{coville20, du25}]\label{c-l}
	Suppose that the kernel $J$ satisfies ${\bf (J)}$ and $c\in\mathbb R$. Then $l\to \lambda_p(\mathcal{L}^c_{(-l,l)})$ is continuous and strictly increasing in $l\in (0,\infty)$, and
	$$	\lim\limits_{l\to \infty}\lambda_p(\mathcal{L}^c_{(-l,l)}) = \inf\limits_{\nu\in\mathbb{R}}\left[ d\int_{\mathbb{R}} J(x)e^{-\nu x}\,dx +  c\nu  \right] -d+f'(0).
	$$
	Moreover,
	\[
\lim\limits_{l\to \infty}\lambda_p(\mathcal{L}^c_{(-l,l)})>0 \mbox{ if and only if } c\in (c_*^-, c_*^+).
\]
\end{proposition}
\begin{proof}
The continuity and monotonicity property of $l\to \lambda_p(\mathcal{L}^c_{(-l,l)})$ were proved in \cite{coville20}, the formula for the limit $\lim\limits_{l\to\infty}\lambda_p(\mathcal{L}_{(-l,l)})$ is given in Theorem 1.2 of \cite{du25}, and the last conclusion is taken from Proposition 5.1 of \cite{du25}.
\end{proof}


Consider the problem
\begin{equation}\label{l-0}
			\left\{
			\begin{array}{ll}
				\displaystyle V_t=d\int_{-l}^lJ(x-y)V(t,y)\,dy-dV+f(V), & t>0,\  x\in (-l, l), \\
				V(0,x)=V_0(x),  &x\in [-l, l].
			\end{array}
			\right.
		\end{equation}
		By Theorem 1.3 of \cite{du25}, the following conclusion holds.
	\begin{proposition}[\cite{du25}]\label{prop-single}
		Suppose that  ${\bf (J)}$ and ${\bf (f_{KPP})}$ hold, and $V_0\in C([-l,l])$ is nonnegative and not identically $0$. Then \eqref{l-0} has a unique solution $V(t,x)$ and
		\[
		\lim_{t\to\infty} V(t,x)=\begin{cases} 0 &\mbox{ uniformly in $x\in[-l,l]$ if } \lambda_p(\mathcal L^0_{(-l,l)})\leq 0,\\
		V_l(x) &\mbox{ uniformly in $x\in[-l,l]$ if } \lambda_p(\mathcal L^0_{(-l,l)})> 0,
		\end{cases}
		\]
		where $V_l(x)$ is the unique positive stationary solution of \eqref{l-0}. Moreover, when $\lambda_p(\mathcal L^0_{\mathbb R})>0$ and hence
		$\lambda_p(\mathcal L^0_{(-l,l)})> 0$ for all large $l>0$, we have
		\[
		\lim_{l\to\infty} V_l(x)=1 \mbox{ uniformly for $x$ in any bounded interval of $\mathbb R$}.
		\]
		\end{proposition}
		
		\begin{lemma}\label{dicho2}
Assume ${\bf(J)}$ and ${\bf(f_{KPP})}$ hold and $J$ is weakly non-symmetric, i.e., \eqref{weak-} holds. Then there exists $l_*\geq 0$ such that
 $\lambda_p(\mathcal{L}^0_{(-l, l)})>0$ if and only if $l>l_*$; moreover,  $l_*=0$ when $f'(0)\geq d$, and $l_*>0$ when $f'(0)<d$.
\end{lemma}
\begin{proof} We first prove the following conclusion:
$$\lim\limits_{l\to 0}\lambda_p(\mathcal{L}^0_{(-l,l)})=f'(0)-d.$$
Since $\lambda_l := \lambda_p(\mathcal{L}^0_{(-l,l)})$ is a principal eigenvalue,   there exists  a strictly positive function $\phi_l \in C([-l,l])$ such that
$$
d \int_{-l}^l J(x-y) \phi_l(y)dy - d \phi_l(x)+f'(0) \phi_l(x)  = \lambda_l \phi_l\ \ \  \textrm{in}\ [-l, l].
$$
Therefore
\begin{eqnarray*}
\big| \lambda_l  -f'(0) +d \big|& =& \frac{d\displaystyle \int_{-l}^l \int_{-l}^l J(x-y) \phi_l(y) \phi_l (x)dy dx  }{ \displaystyle\int_{-l}^l \phi_l^2 (x)dx} \leq  \frac{d \displaystyle\|J\|_\infty \left( \int_{-l}^l  \phi_l (x)dx \right)^2 }{\displaystyle \int_{-l}^l \phi_l^2 (x)dx} \\
&\leq & \frac{d \|J\|_\infty 2l \displaystyle \int_{-l}^l  \phi_h^2 (x)dx  }{ \displaystyle\int_{-l}^l \phi_h^2 (x)dx}=2ld\|J\|_\infty \to 0\;\; \mbox{ as $l \rightarrow 0^+$.}
\end{eqnarray*}

By Proposition \ref{c-l}, $l\to\lambda_l$ is continuous and strictly increasing, and due to  \eqref{weak-}, $\lim_{l\to\infty} \lambda_l>0$.
Therefore,  
\[
d\in (0,f'(0)]\implies \lambda_l>\lim_{h\to 0}\lambda_h=f'(0)-d\geq 0 \mbox{ for every fixed } l>0,
\]
and $d>f'(0)$ implies the existence of a unique $l_*>0$ such that 
\[
\lambda_l<0 \mbox{ for } l\in (0, l_*),\ \lambda_{l_*}=0,\ \lambda_l>0 \mbox{ for } l>l_*.
\]
This completes the proof.
\end{proof}

\subsection{Proof of Theorem \ref{thm2}}
Throughout this subsection, we assume that ${\bf(J)}$, ${\bf(f_{KPP})}$ hold and $J$ is weakly non-symmetric, i.e., \eqref{weak-} holds.

\begin{lemma}\label{thm-vanishing}
If $h_\infty-g_\infty< + \infty$,
then $u(t,x)\rightarrow0$ uniformly in $[g(t),h(t)]$ as
$t\rightarrow+\infty$ and
$\lambda_p(\mathcal{L}^0_{(g_\infty, h_\infty)})\leq 0$.
\end{lemma}

\begin{proof}
We first prove that
$$
\lambda_p(\mathcal{L}^0_{(g_\infty, h_\infty)})\leq 0.
$$
Suppose that
$
\lambda_p(\mathcal{L}^0_{(g_\infty, h_\infty)}) > 0$.
Then
$\lambda_p(\mathcal{L}^0_{(g_\infty + \epsilon, h_\infty - \epsilon)}) > 0$ for small $\epsilon>0$, say $\epsilon\in (0,\epsilon_1)$. Moreover, for such $\epsilon$, there exists $ T_{\epsilon}>0$ such
that
$$
h(t)>h_\infty-\epsilon, \ \ g(t)<g_\infty+\epsilon\;\;\;  \mbox{ for $t> T_{\epsilon}$}.
$$
Consider the problem 
\begin{equation}\label{single-epsilon}
\left\{
\begin{aligned}
&w_t=d\int_{g_\infty + \epsilon}^{h_\infty - \epsilon}J(x-y)w(t,y)dy-dw +f(w),
&& t >  T_{\epsilon},~x\in [g_\infty + \epsilon, h_\infty - \epsilon],\\
&w( T_{\epsilon} ,x)=u(T_{\epsilon}, x), && x\in [g_\infty + \epsilon, h_\infty - \epsilon].
\end{aligned}
\right.
\end{equation}
Since $\lambda_p(\mathcal{L}^0_{(g_\infty + \epsilon, h_\infty - \epsilon)}) > 0$, Proposition \ref{prop-single} indicates that
the  solution   $ w_{\epsilon}(t,x)$ of   (\ref{single-epsilon})  converges to the unique steady state $W_{\epsilon}(x)$ of (\ref{single-epsilon}) uniformly in $[g_\infty + \epsilon, h_\infty - \epsilon]$ as $t \rightarrow+\infty$.

Moreover, by the maximum principle Theorem 2.1 and a simple  comparison argument we have
$$
u(t,x)\ge  w_{\epsilon}(t,x)~\text{ for }~t> T_{\epsilon}~\text{ and }~x\in[g_\infty+\epsilon,h_\infty-\epsilon].
$$
Thus, there exists $T_{1\epsilon}> T_{\epsilon}$ such that
$$
u(t,x)\ge {1\over 2} W_{\epsilon}(x)>0~\text{ for }~t> T_{1\epsilon}~\text{ and }~x\in[g_\infty+\epsilon,h_\infty-\epsilon].
$$

Note that since $J(0)>0$, there exist $\epsilon_0>0$ and $\delta_0>0$  such that $J(x)> \delta_0$ if $|x| < \epsilon_0$. Thus for $0<\epsilon <
\min\big\{\epsilon_1, \epsilon_0 /2\big\}$ and $t> T_{1\epsilon}$, we have
\begin{eqnarray*}
h'(t) &= & \mu\int_{g(t)}^{h(t)}\int_{h(t)}^{+\infty} J(y-x)u(t,x)dydx
      \geq  \mu\int_{g_\infty+\epsilon}^{h_\infty - \epsilon}\int_{h_\infty}^{+\infty} J(y-x)u(t,x)dydx\\
      &\geq & \mu\int_{h_\infty - \epsilon_0 /2}^{h_\infty - \epsilon}\int_{h_\infty}^{h_\infty + \epsilon_0 /2} \delta_0 {1\over 2} W_{\epsilon}(x) dydx >0.
\end{eqnarray*}
 This implies  $h_\infty=+\infty$, a contradiction  to the assumption that $h_\infty-g_\infty<+\infty$.
Therefore, we must have
$$
\lambda_p(\mathcal{L}^0_{(g_\infty, h_\infty)})\leq 0.
$$

We are now ready to show that $u(t,x)\rightarrow 0$ uniformly in $[g(t),h(t)]$ as
$t\rightarrow+\infty$.
Let $\bar u(t,x)$ denote the unique solution of
\begin{equation}
\label{single-infinity}
\left\{
\begin{aligned}
&\bar u_t=d\int_{g_\infty}^{h_\infty}J(x-y)\bar u(t,y)dy-d\bar u(t,x)+f(\bar u),&  & t>0,~x\in [g_\infty,h_\infty],\\
&\bar u(0,x)=\tilde u_0(x), & &  x\in  [g_\infty,h_\infty],
\end{aligned}
\right.
\end{equation}
where
$$
\tilde u_0(x)=u_0(x)~\text{ if }-h_0\le x\le h_0,~\text{ and }
\tilde u_0(x)=0~\text{ if }~x\not\in [-h_0, h_0].
$$
By the maximum principle Theorem 2.1, we have
$0\leq u(t,x)\le\bar u(t,x)$ for $t>0$ and $x\in[g(t),h(t)]$.
Since
$$
\lambda_p(\mathcal{L}^0_{(g_\infty, h_\infty)})\leq 0,
$$
Proposition \ref{prop-single} implies that   $\overline u(t,x)
\rightarrow0$ uniformly in $x\in [g_\infty,h_\infty]$ as $t
\rightarrow+\infty$. Hence $u(t,x)\rightarrow0$ uniformly in
$x\in[g(t),h(t)]$ as $t\rightarrow+\infty$. This completes
the proof.
\end{proof}

\begin{lemma}\label{lemma-same}
 $h_\infty< + \infty$ if and only if
$-g_\infty<  + \infty$.
\end{lemma}

\begin{proof}
Arguing indirectly, we assume, without loss of generality,  that $h_\infty=+\infty$
and $-g_\infty<+\infty$.     By Proposition \ref{c-l},  there exists  $h_1>0$ such that $\lambda_p(\mathcal{L}^0_{(0, h_1)}) > 0$. Moreover, for any $\epsilon>0$ small, there exists $ T_{\epsilon}>0$ such
that
$
h(t)>h_1, \ \ g(t)<g_\infty+\epsilon<0
$ for $t> T_{\epsilon}$.
In particular,
$$
\lambda_p(\mathcal{L}^0_{(g_\infty+\epsilon, h_1)}) >\lambda_p(\mathcal{L}^0_{(0, h_1)})  > 0.
$$
We now consider the problem
\begin{equation*}
\left\{
\begin{aligned}
&w_t=d\int_{g_\infty + \epsilon}^{h_1}J(x-y)w(t,y)dy-dw +f(w),
& & t >  T_{\epsilon},~x\in [g_\infty + \epsilon, h_1],\\
& w( T_{\epsilon} ,x)=u(T_{\epsilon}, x),  & & x\in [g_\infty + \epsilon, h_1].
\end{aligned}
\right.
\end{equation*}
Similar to  the proof  of Theorem \ref{thm-vanishing},   by choosing $\epsilon < \epsilon_0/2$, we have
$g'(t) <- c<0 $ for all large $t$. This is a contradiction to  $-g_\infty<+\infty$.
\end{proof}

\begin{lemma}\label{thm-spreading}
If $h_\infty
-g_\infty= + \infty$, then $\lim_{t\rightarrow+\infty}u(t,x)=1$ locally uniformly in $\mathbb{R}$.
\end{lemma}

\begin{proof}
Thanks to Lemma \ref{lemma-same}, $h_\infty-g_\infty=+\infty$ implies  $h_{\infty} = -g_{\infty}= +\infty$. Choose an increasing sequence $\{ t_n\}_{n\geq 1}$ satisfying
$$
\lim_{n\rightarrow +\infty} t_n =   +\infty,\; \lambda_p(\mathcal{L}^0_{(g(t_n), h(t_n))})>0 \mbox{ for all } n\geq 1.
$$

Denote $g_n = g(t_n)$, $h_n = h(t_n)$ and let $\underline u_n(t,x)$  be the unique solution of the following problem
\begin{equation}\label{single-tn}
\left\{
\begin{aligned}
& \underline u_t=d\int_{g_n}^{h_n}J(x-y)\underline u(t,y)dy
-d\underline u(t,x)+f(\underline u),&  & t>t_n,~x\in  [g_n,h_n],\\
& \underline u(t_n,x)= u(t_n,x),& &  x\in  [g_n,h_n].
\end{aligned}
\right.
\end{equation}
By the maximum principle Theorem 2.1  we have
\begin{equation}\label{pf-thm-tn>}
u(t,x)\ge\underline u_n(t,x)~\text{ in }~[t_n,+\infty)\times[g_n,h_n].
\end{equation}
Since
$
\lambda_p(\mathcal{L}^0_{[g_n, h_n]}) > 0$, by
 Proposition \ref{prop-single}, problem (\ref{single-tn}) admits a unique positive steady state $\underline u_n(x)$ and
\begin{equation}\label{pf-thm-limit}
\lim_{t\rightarrow+\infty} \underline u_n (t,x)=\underline u_n (x)~\text{ uniformly in } \ [g_n,h_n].
\end{equation}

By Proposition \ref{prop-single},
\[
\lim_{n\to\infty}  \underline u_n(x)  =1 \mbox{ locally uniformly in } x\in\mathbb R.
\]
It follows from this fact,  (\ref{pf-thm-tn>}) and (\ref{pf-thm-limit}) that
\begin{equation}\label{pf-thm-liminf}
\liminf_{t\rightarrow+\infty}u(t,x)\ge 1~\text{ locally
uniformly in }~\mathbb{R}.
\end{equation}

To complete the proof, it remains to prove that
\begin{equation}\label{pf-thm-limsup}
\limsup_{t\rightarrow+\infty}u(t,x)\le 1~\text{ locally
uniformly in }~\mathbb{R}.
\end{equation}
Let $\hat u(t)$ be the unique solution of the ODE problem
\[
\hat u'=f(\hat u),\; \hat u(0)=\|u_0\|_\infty.
\]
By the maximum principle we have
$u(t,x)\le\hat u(t)$ for $t>0$ and $x\in[g(t),h(t)]$. Since $\hat u(t)\to 1$ as $t\to\infty$, \eqref{pf-thm-limsup} follows immediately.
\end{proof}

Theorem \ref{thm2} clearly follows directly from  Lemmas \ref{thm-vanishing} and \ref{thm-spreading}.

\subsection{Proof of Theorem \ref{thm3}}
Next we look for criteria guaranteeing spreading or vanishing for \eqref{101}. From Lemma \ref{dicho2} we see that if
\begin{equation}\label{d-big}
d\in (0, f'(0)],
\end{equation}
then $\lambda_p(\mathcal{L}^0_{(\ell_1,\ell_2)}) > 0$ for any finite interval $(\ell_1,\ell_2)$. Combining this with Lemma \ref{thm-vanishing} and Theorem \ref{thm2}, we
immediately obtain the following conclusion:

 \begin{lemma}\label{f'(0)>d}
  When \eqref{d-big} holds,  spreading always happens for \eqref{101}.
\end{lemma}

We next consider the case
\begin{equation}\label{d-large}
d>f'(0).
\end{equation}
In this case, by Lemma \ref{dicho2}, there exists $\ell^*>0$ such that
\[
\lambda_p(\mathcal{L}_{I})=0 \mbox{ if } |I|=\ell^*,\; \lambda_p(\mathcal{L}_I)<0 \mbox{ if } |I|<\ell^*,\; \lambda_p(\mathcal{L}_I)>0 \mbox{ if } |I|>\ell^*,
\]
where $I$ stands for a finite open interval in $\mathbb R$, and $|I|$ denotes its length.

\begin{lemma}\label{thm-mu-small} Suppose that \eqref{d-large} holds and $\ell^*$ is defined above.
If $h_0\geq \ell^*/2$ then spreading always happens for \eqref{101}. If $h_0<\ell^*/2$,  then there
exists $\underline\mu>0$ such that vanishing happens for \eqref{101}  if
$0<\mu\le\underline\mu$.
\end{lemma}

\begin{proof} If $h_0\geq \ell^*/2$ and vanishing happens, then $(g_\infty, h_\infty)$ is a finite interval with length strictly bigger than $2h_0\geq \ell^*$.
Therefore $\lambda_p(\mathcal{L}_{(g_\infty, h_\infty)})>0$, contradicting the conclusion in Lemma \ref{thm-vanishing}. Thus when $h_0\geq \ell^*/2$,
spreading always happens for \eqref{101}.

We now consider the case $ h_0<\ell^*/2$. We fix   $h_1\in (h_0, \ell^*/2)$ and consider the
following problem
\begin{equation}\label{single-h1}
\left\{
\begin{aligned}
&w_t(t,x)=d\int_{-h_1}^{h_1}J(x-y)w(t,y)dy-dw +f(w),
& &t>0,~x\in [- h_1, h_1],\\
&w(0,x)=u_0 (x), & &  x\in [- h_0, h_0], \\
& w(0,x ) = 0,   & &  x\in [- h_1, -h_0) \cup ( h_0, h_1]
\end{aligned}
\right.
\end{equation}
and  denote its unique solution by $\hat w(t,x)$.
The choice of $h_1$ guarantees that
$$
\lambda_1 := \lambda_p(\mathcal{L}_{(-h_1, h_1)}) < 0.
$$
Let $\phi_1>0$ be the corresponding normalized eigenfunction of $\lambda_1$, namely $\|\phi_1\|_\infty=1$ and
\[
\mathcal{L}_{(-h_1, h_1)}[\,\phi_1](x)=\lambda_1\phi_1(x) \mbox{ for } x\in [-h_1, h_1].
\]

By ${\bf (f_{KPP})}$, 
\begin{eqnarray*}
\hat w_t(t,x) &= & d\int_{-h_1}^{h_1}J(x-y)\hat w(t,y)dy-d\hat w +f(\hat w)\\
        & \leq & d\int_{-h_1}^{h_1}J(x-y)\hat w(t,y)dy-d\hat w +f'(0) \hat w.
\end{eqnarray*}
On the other hand, for $C_1>0$ and $w_1 = C_1 e^{\lambda_1 t/4} \phi _{1}$  it is easy to check that
\begin{eqnarray*}
&&    d\int_{-h_1}^{h_1}J(x-y)w_1(t,y)dy  -  d w_1  + f'(0) w_1 - w_{1t}(t,x)\\
        & = & C_1 e^{\lambda_1 t/4 }\left \{  d\int_{-h_1}^{h_1}J(x-y)  \phi _1 (y)dy-d  \phi _1 +f'(0)  \phi _1  - {\lambda_1\over 4} \phi _1 \right\}\\
        & = &{3\lambda_1\over 4}C_1 e^{\lambda_1 t/4 } \phi _1<0.
\end{eqnarray*}
 Choose
$C_1>0$ large such that $C_1\phi _1> u_0$ in $[-h_1, h_1]$.  Then we can apply the maximum principle Theorem 2.1 to $w_1-\hat w$ to deduce
\begin{equation}\label{pf-thm-w-decay}
\hat w (t,x)\le w_1 (t,x) = C_1 e^{\lambda_1 t/4} \phi _1 \leq C_1 e^{\lambda_1 t/4} ~\text{ for   }
~t>0~\text{and}~x\in[-h_1, h_1].
\end{equation}

Now define
\begin{equation*}
\hat h(t)=h_0+ 2\mu  h_1  C_1 \int_0^t  e^{\lambda_1 s/4}ds~\text{ and }~\hat g(t)=-\hat
h(t)~\text{ for }~t\ge0,
\end{equation*}
We claim that  $(\hat w, \hat h, \hat g)$ is an upper solution of \eqref{101}.

Firstly, we compute that for any $t>0$,
$$
\hat h(t)=h_0 - 2 \mu h_1 C_1  \frac{4}{\lambda_1 }\left( 1-e^{\lambda_1 t/4} \right) <   h_0 - 2 \mu h_1  C_1  \frac{4}{\lambda_1 }  \leq   h_1
$$
provided that
$$
0< \mu\leq \underline{\mu}:=  \frac{-\lambda_1 (h_1- h_0) }{8h_1  C_1}.
$$
Similarly, $\hat  g(t) > -h_1 $ for any $t>0$. Thus by (\ref{single-h1}) we have
$$
\hat w_t(t,x) \geq d\int_{\hat  g(t)}^{\hat  h(t)}J(x-y)\hat w(t,y)dy-d \hat w +f(\hat w)
 \ \ \  \textrm{for}\  t>0,~x\in [\hat  g(t),\hat  h(t)].
$$
Secondly, due to (\ref{pf-thm-w-decay}),  it is easy to check that
\begin{eqnarray*}
\int_{\hat g(t)}^{\hat h(t)}\int_{\hat h(t)}^{+\infty}J(y-x)\hat w(t,x) dydx  <  2 h_1 C_1 e^{\lambda_1 t/4}.
\end{eqnarray*}
Thus
$$
\hat h' (t)= 2  \mu h_1 C_1    e^{\lambda_1 t/4}  >  \mu \int_{\hat g(t)}^{\hat h(t)}\int_{\hat h(t)}^{+\infty}J(y-x)\hat w(t,x) dydx.
$$
Similarly, one has
$$
\hat g' (t) <-  \mu \int_{\hat g(t)}^{\hat h(t)}\int_{-\infty}^{\hat g(t)}  J(y-x)\hat w(t,x) dydx.
$$

Now it is clear that $(\hat w, \hat h, \hat g)$ is an upper solution of (\ref{101}).
Hence, by the comparison principle Theorem 2.3, we have
$$
u(t,x ) \leq \hat w (t,x),\ g(t) \geq \hat g(t)\ \textrm{and}\  h(t)\leq \hat h(t) \ \ \textrm{for} \ t>0,\ x\in [g(t), h(t)].
$$
It follows that
$$
h_\infty-g_\infty\le\lim_{t\rightarrow+\infty}\left(\hat h(t) -\hat g(t)\right)\leq 2 h_1 < + \infty.
$$
This completes the proof.
\end{proof}

\begin{theorem}\label{thm-mu-large} Suppose that \eqref{d-large} holds and $h_0<\ell^*/2$.
 Then there exists  $\bar\mu > 0$ such
that spreading happens to \eqref{101} if $\mu > \bar\mu$.
\end{theorem}

\begin{proof}
Suppose that  for any $\mu >0$,   $h_\infty-g_\infty<+\infty$. We will derive a contradiction.

First of all, notice that by Lemma \ref{thm-vanishing}, we have $\lambda_p(\mathcal{L}_{(g_\infty, h_\infty)})\leq 0.$ This indicates that $   h_\infty-g_\infty \leq \ell^* $. To stress the dependence on $\mu$,  let $(u_{\mu}, g_{\mu}, h_{\mu})$ denote the solution of  (\ref{101}).
By the comparison principle Thorem 2.3, it is easily seen that $u_{\mu},  -g_{\mu},  h_{\mu}$  are increasing in $\mu>0$. Also denote
$$
h_{\mu, \infty} : = \lim_{t\rightarrow +\infty}  h_{\mu}(t), \ \ g_{\mu, \infty} : = \lim_{t\rightarrow +\infty}  g_{\mu}(t).
$$
Obviously, both $h_{\mu, \infty} $ and $- g_{\mu, \infty}$ are increasing in $\mu$. Denote
$$
H_{\infty} := \lim_{\mu \rightarrow +\infty}h_{\mu, \infty},\ \ G_{\infty} := \lim_{\mu \rightarrow +\infty}g_{\mu, \infty}.
$$

Recall   that since $J(0)>0$, there exist $\epsilon_0>0$ and $\delta_0>0$  such that $J(x)> \delta_0$ if $|x| < \epsilon_0$. Then there exist $\mu_1$, $t_1$ such that for $\mu\geq \mu_1$, $t\geq t_1$, we have $h_{\mu} (t)> H_{\infty}-\epsilon_0/4$.  It follows that
\begin{eqnarray*}
\mu &=& \left( \int_{t_1}^{+\infty} \int_{g_{\mu}(\tau)}^{h_{\mu}(\tau)}\int_{h_{\mu}(\tau)}^{+\infty}
J(y-x)u_{\mu}(\tau,x)dydxd\tau \right)^{-1} \left[h_{\mu, \infty} - h_{\mu}(t_1) \right]\\
  &\leq &\left( \int_{t_1}^{t_1+1} \int_{g_{\mu_1}(\tau)}^{h_{\mu_1}(\tau)}\int_{h_{\mu_1}(\tau)+\epsilon_0/4}^{+\infty}
J(y-x)u_{\mu_1}(\tau,x)dydxd\tau \right)^{-1} \ell^*\\
&\leq &\left(\delta_0 \int_{t_1}^{t_1+1} \int_{h_{\mu_1}(\tau) - \epsilon_0/2}^{h_{\mu_1}(\tau)}\int_{h_{\mu_1}(\tau)+\epsilon_0/4}^{h_{\mu_1}(\tau)+\epsilon_0/2}
 u_{\mu_1}(\tau,x)dydxd\tau \right)^{-1} \ell^*\\
   &=  &   \left({1\over 4} \delta_0\epsilon_0\int_{t_1}^{t_1+1} \int_{h_{\mu_1}(\tau) - \epsilon_0/2}^{h_{\mu_1}(\tau)}
  u_{\mu_1}(\tau,x) dxd\tau \right)^{-1} \ell^*< +\infty,
\end{eqnarray*}
which clearly is a contradiction.
\end{proof}

We can now deduce a sharp criteria in terms of $\mu$ for the spreading-vanishing dichotomy.

\begin{lemma}\label{thm-mu-critical} Suppose that \eqref{d-large} holds and $h_0<\ell^*/2$.
 Then there
exists $\mu^*\in (0, \infty)$ such that  vanishing   happens for \eqref{101} if $0<\mu\le\mu^*$
and  spreading  happens for \eqref{101} if $\mu>\mu^*$.
\end{lemma}

\begin{proof}
Define
$$
\Sigma=\left\{\mu:~\mu>0~\text{such that}~h_\infty -g_\infty  < + \infty\right\}.
$$

 By Lemmas \ref{thm-mu-small} and \ref{thm-mu-large} we see that
  $0< \sup \ \Sigma < +\infty$.
Again we let $(u_{\mu}, g_{\mu}, h_{\mu})$ denote the solution of  (\ref{101}), and set $h_{\mu,\infty}:=\lim_{t\rightarrow+\infty}h_{\mu}(t)$, $g_{\mu,\infty}:=\lim_{t\rightarrow+\infty}g_{\mu}(t)$, and
denote $ \mu^*  =\sup  \Sigma$.

As before $u_{\mu},  -g_{\mu},  h_{\mu}$  are increasing in $\mu>0$. This immediately gives that if $\mu_1 \in \Sigma$, then $\mu  \in \Sigma$ for any $\mu<\mu_1$ and if $\mu_1 \not\in \Sigma$, then $\mu  \not\in \Sigma$ for any $\mu > \mu_1$. Hence it follows that
\begin{equation}\label{pf-thm-mu-separate}
(0, \mu^*) \subseteq \Sigma,\ \ ( \mu^*, +\infty) \cap \Sigma=\emptyset.
\end{equation}

To complete the proof, it remains to show  that $\mu^* \in  \Sigma$. Suppose that $\mu^* \not\in  \Sigma$. Then $h_{\mu^*,\infty}=  -g_{\mu^*,\infty} =+\infty$. Thus there exists $T>0$ such that $-g_{\mu^*}(t)> \ell^*,  h_{\mu^*}(t) > \ell^*$ for $t\geq T$. Hence  there exists $\epsilon >0$ such that for 
$\mu \in ( \mu^*-\epsilon, \mu^*+\epsilon)$,
$-g_{ \mu}(T)> \ell^*/2  ,  h_{ \mu}(T) > \ell^*/2  $, which implies $\mu\not\in\Sigma$. This clearly contradicts (\ref{pf-thm-mu-separate}).
Therefore $ \mu^* \in  \Sigma$.
\end{proof}

	\section{Semi-wave solutions} 

We want to determine the spreading speed of the nonlocal free boundary problem
\begin{equation}
\left\{
\begin{aligned}
&u_t=d\int_{g(t)}^{h(t)}J(x-y)u(t,y)dy-du+f(u),
& &t>0,~x\in(g(t),h(t)),\\
&u(t,g(t))=u(t,h(t))=0,& &t>0,\\
&h'(t)=\mu\int_{g(t)}^{h(t)}\int_{h(t)}^{+\infty}
J(y-x)u(t,x)dydx,& &t>0,\\
&g'(t)=-\mu\int_{g(t)}^{h(t)}\int_{-\infty}^{g(t)}
J(y-x)u(t,x)dydx,& &t>0,\\
&u(0,x)=u_0(x),~h(0)=-g(0)=h_0,& &x\in[-h_0,h_0],
\end{aligned}
\right.
\label{e1}
\end{equation}
where $d, \mu, h_0$ are given positive constants. The initial
function $u_0(x)$ satisfies \eqref{102}.
 The kernel function
$J: \mathbb{R}\rightarrow\mathbb{R}$ satisfies the basic condition 
\begin{description}
\item[(J)] $J\in C(\R)\cap L^\infty(\R)$, $J\geq 0$, $J(0)>0,~\int_{\mathbb{R}}J(x)dx=1$.
\end{description}
The growth term $f: \mathbb{R}^+
\rightarrow\mathbb{R}$  satisfies the KPP condition
\begin{description}
\item[${\bf (f_{KPP})}$] $\begin{cases}f\in C^1, \ f(0)=f(1)=0, \ f'(0)>0>f'(1),\\
 \mbox{$f(u)/u$ is non-increasing in $(0,\infty)$.}\end{cases}$
\end{description}

 For a non-symmetric $J$ satisfying ${\bf (J)}$, the following two quantities determined by $J$ and $f'(0)$ alone play an important role:
\begin{align*}
	\displaystyle	c_*^- = \sup_{\nu<0} \frac{\displaystyle d\int_{\mathbb{R}} J(x)e^{\nu x}\,dx - d + f'(0)}{\nu}, \quad
		c_*^+ = \inf_{\nu>0} \frac{\displaystyle d\int_{\mathbb{R}} J(x)e^{\nu x}\,dx - d + f'(0)}{\nu},
	\end{align*}
	
	 It can be shown that $c_*^-$ is achieved by some $\nu<0$ when it is finite, and a parallel conclusion holds for $c_*^+$.
	 It is easily checked that $c_*^-$ is finite if and only if $J$ satisfies additionally the following {\color{red}thin-tail} condition at $x=-\infty$,\smallskip

${\bf (J_{thin}^-):}$ There exists $\lambda>0$ such that $\displaystyle \int_{0}^{+\infty}J(-x)e^{\lambda x}\,dx<+\infty$.
\medskip
	
	\noindent Similarly,  $c_*^+$ is finite if and only if $J$ satisfies\smallskip
	
	${\bf (J_{thin}^+):}$ There exists $\lambda>0$ such that $\displaystyle \int_{0}^{+\infty}J(x)e^{\lambda x}\,dx<+\infty$.

\medskip

If we define
\begin{equation}\label{J-thin-not}\begin{cases}
c_*^-=-\infty \mbox{ when ${\bf (J_{thin}^-)}$ does not hold},\\[3mm]
 c_*^+=+\infty \mbox{ when ${\bf (J_{thin}^+)}$ does not hold},
 \end{cases}
\end{equation}
then  the propagation dynamics of the corresponding Cauchy problem of \eqref{e1},
\begin{equation}
\label{cau}
	\left\{
	\begin{array}{ll}
		U_t = \displaystyle d \int_{\mathbb{R}} J(x-y) U(t,y) \, dy - d U(t,x) + f(U), & t > 0, \; x \in \mathbb{R}, \\
		U(0, x) =U_0(x)
	\end{array}
	\right.
\end{equation}
has the properties described in the following result:

\noindent
{\bf Theorem A.}(\cite{du25}) {\it
Suppose that ${\bf(J)}$ and ${\bf (f_{KPP})}$ hold. Then for any initial function $U_0(x)$ which is continuous and nonnegative with non-empty compact support, the unique solution $U(t,x)$ of \eqref{cau} satisfies
\begin{align*}
	 \lim_{t\to \infty} U(t,x)=\begin{cases} 1 \mbox{ uniformly for } x\in [a_1t, b_1t] \mbox{ provided that } [a_1,b_1]\subset (c_*^-, c_*^+),\\[2mm]
	 0 \mbox{ uniformly for $x\leq a_2t$ provided that $c_*^->-\infty$ and $a_2<c_*^-$},\\[3mm]
	 0 \mbox{ uniformly for $x\geq b_2t$ provided that $c_*^+<\infty$ and $b_2>c_*^+$}.
	 \end{cases}
	\end{align*}
}

Following \cite{aw}, the conclusions in Theorem A can be interpreted as indicating a leftward spreading speed of $c_*^-$ and rightward spreading speed of $c_*^+$ for \eqref{cau}.
The following  result of Yagisita \cite{yagisita}  (see also Theorem 1.5 in \cite{coville2}) on traveling waves  provides further meanings for $c_*^-$ and $c_*^+$.  \smallskip

\noindent
{\bf Theorem B.} (\cite{yagisita}) {\it
Suppose that ${\bf(J)}$ and ${\bf (f_{KPP})}$ are satisfied. Then the following conclusions hold.
\begin{itemize}
	\item[{\rm (i)}] The \underline{rightward} traveling wave problem
	\begin{equation}
		\label{e6}
		\begin{cases}
			\displaystyle d\int_{\mathbb{R}} J(x-y)\phi(y)\,dy - d\phi(x) + c\phi'(x) + f(\phi(x)) = 0, & x \in \mathbb{R}, \\[3mm]
			\phi(-\infty) = 1, \quad \phi(+\infty) = 0
		\end{cases}
	\end{equation}
	has a solution pair $(c,\phi) \in \mathbb R\times L^{\infty}(\mathbb{R})$ with $\phi$ nonincreasing if and only if $c_*^+<\infty$. Moreover, in such a case, for every $c\geq c_*^+$,  \eqref{e6} has a solution $\phi \in C^1(\mathbb{R})$ that is strictly decreasing, and  \eqref{e6} has no such solution for $c<c_*^+$.

	\item[{\rm (ii)}]  The \underline{leftward} traveling wave problem
	\begin{equation}
		\label{e7}
		\left\{
		\begin{array}{ll}
			\displaystyle d\int_{\mathbb{R}}J(x-y)\psi(y)\,dy-d\psi(x)-c\psi'(x)+f(\psi(x))=0, & x\in\mathbb{R}, \\
			\psi(-\infty)=0, \ \ \psi(+\infty)=1,
		\end{array}
		\right.
	\end{equation}
	has a solution pair $(c,\psi)\in \mathbb{R}\times L^{\infty}(\mathbb{R})$ with $\psi$ nondecreasing if and only if $c_*^->-\infty$. Moreover, in such a case, for each $c\geq -c_*^-$,  \eqref{e7} has a solution $\psi \in C^1(\mathbb{R})$ that is strictly increasing, and  \eqref{e7} has no such solution for $c<-c_*^-$.
	\end{itemize}
}
\medskip

{\bf Remark:} Problem \eqref{cau}  is the limiting problem of \eqref{e1} when $\mu\to\infty$.
\medskip

The propagation dynamics of \eqref{e1}  depends crucially on the associated semi-wave solutions, which are pairs $(c,\phi)\in (0,+\infty)\times C^1((-\infty,0])$ and $(\tilde c,\psi) \in (0,+\infty)\times C^1([0,\infty))$, determined by the following equations, respectively:

{\small
\begin{equation}
	\label{e2}
	\left\{
	\begin{array}{ll}
		\displaystyle d\int_{-\infty}^{0}J(x-y)\phi(y)\,dy-d\phi(x){\color{red}+c\, }\phi'(x)+f(\phi(x))=0, & {\color{red}-\infty<x<0,} \smallskip\\
		\phi(-\infty)=1,\ \phi(0)=0,
	\end{array}
	\right.
\end{equation}\vspace{-0.5cm}
\begin{equation}
	\label{e3}
	\displaystyle c=\mu \int_{-\infty}^{0}\int_{0}^{+\infty}J(y-x)\phi(x)\,dydx,
\end{equation}\vspace{-0.5cm}
\begin{equation}
	\label{e4}
	\left\{
	\begin{array}{ll}
		\displaystyle d\int_{0}^{+\infty}J(x-y)\psi(y)\,dy-d\psi(x){\color{red}-\tilde c\, }\psi'(x)+f(\psi(x))=0, & {\color{red}0<x<+\infty,} \medskip\\
		\psi(0)=0,\ \psi(+\infty)=1.
	\end{array}
	\right.
\end{equation}\vspace{-0.5cm}
\begin{equation}
	\label{e5}
	\tilde c=\mu \int_{0}^{+\infty}\int_{-\infty}^{0}J(y-x)\psi(x)\,dydx.
\end{equation}
}

Note that for $\Phi(t,x):=\phi(x-ct)$ and $\Psi(t,x)=\psi(x+\tilde ct)$, \eqref{e2} and \eqref{e4} imply
{\small \[\begin{cases}
\displaystyle \Phi_t=d\int_{-\infty}^{ct}J(x-y)\Phi(t,y)dy-d\Phi+f(\Phi),\ \Phi(t, ct)=0 \ &\mbox{ for } {\color{red}x<ct,}\ t>0,\\[3mm]
\displaystyle\Psi_t=d\int_{-\tilde ct}^\infty J(x-y)\Psi(t,y)dy-d\Psi +f(\Psi),\ \Psi(t, -\tilde c t)=0\ & \mbox{ for } {\color{red}x>-\tilde ct,}\ t>0.
\end{cases}
\]}

We will call $\phi^c$ a {\bf rightward semi-wave} of \eqref{cau} with speed $c$ if $(c, \phi^c)$ solves \eqref{e2}, and call $\psi^{\tilde c}$ a {\bf leftward semi-wave} of \eqref{cau} with speed $\tilde c$ if $(\tilde c, \psi^{\tilde c})$ solves \eqref{e4}.
 \smallskip

 Whether such semi-wave solutions can satisfy additionally \eqref{e3} and \eqref{e5} depends on the following extra properties of $J(x)$, apart from ${\bf (J)}$,
\begin{itemize}
\item [${\bf (J_1^+):}$] $\displaystyle \int_{-\infty}^{0}\int_{0}^{+\infty}J(y-x)\,dydx<+\infty$, i.e., $\displaystyle\int_{0}^{+\infty}xJ(x)\,dx<+\infty$,
\item [${\bf (J_1^-):}$] $\displaystyle \int_{0}^{+\infty}\int_{-\infty}^{0}J(y-x)\,dydx<+\infty$, i.e., $\displaystyle\int_{0}^{+\infty}xJ(-x)\,dx<+\infty$.
\end{itemize}

We are going to prove the following result.

\begin{theorem}\label{th1.2}
	Suppose that ${\bf(J)}$ and ${\bf(f_{KPP})}$ are satisfied. Then the following conclusions hold:
	
	\begin{itemize}
			\item[{\bf ($a^+$)}] Problem \eqref{e2} admits a nonnegative solution $\phi \in C^1((-\infty, 0])$ with $c>0$ if and only if $c_*^+ \in (0, +\infty]$ and $c <c_*^+$. Moreover, in such a case, \eqref{e2}  has a unique solution $\phi=\phi^c$, it is $C^1$ and  $(\phi^c)'(x) < 0$ for $x \in (-\infty, 0]$.
\item[{\bf ($b^+$)}] Suppose $c_*^+ \in (0, +\infty]$ and $\phi^c$ is the unique solution of \eqref{e2} with $c\in (0, c_*^+)$. Then there exists a unique $c_0\in (0, c_*^+)$ such that  $(c,\phi)=(c_0, \phi^{c_0})$ solves \eqref{e3}  if and only if ${\bf(J_1^+)}$ holds.
	
\item[{\bf ($a^-$)}] Problem \eqref{e4} admits a nonnegative solution $\psi \in C^1([0, +\infty))$ with $\tilde c>0$ if and only if $c_*^- \in [-\infty, 0)$ and $\tilde c < -c_*^-$. Moreover, in such a case, \eqref{e4} has a  unique solution $\psi=\psi^{\tilde c}$, it is $C^1$ and  $(\psi^{\tilde c})'(x) > 0$ for $x \in [0, +\infty)$.
			
			\item[{\bf ($b^-$)}] Suppose $c_*^- \in [-\infty, 0)$, and $\psi^{\tilde c}$ is the unique solution of \eqref{e4} with $\tilde c\in (0, -c_*^-)$. Then there exists a unique $\tilde c_0\in (0, -c_*^-)$ such that  $(c,\psi)=(\tilde c_0, \psi^{\tilde c_0})$ solves \eqref{e5}  if and only if ${\bf(J_1^-)}$ holds.
	\end{itemize}
	
\end{theorem}
 Note that the existence of a solution to \eqref{e2} requires \(c_*^+>0\). Similarly, the existence of a solution to \eqref{e4} requires \(c_*^-<0\).  \smallskip

 The unique speed $c=c_0$ in $(b^+)$ will determine the asymptotic speed of $h(t)$, and the corresponding $\phi^{c_0}$ will be called the rightward semi-wave of \eqref{e1}. Similarly,  the unique speed $\tilde c=\tilde c_0$ in $(b^-)$ will determine the asymptotic speed of $g(t)$, and the corresponding $\psi^{\tilde c}$ will be called the leftward semi-wave of \eqref{e1}.

\subsection{A maximum principle and its first application}

\begin{lemma}[Lemma 2.5 in \cite{du21}]\label{max}
	Assume that ${\bf(J)}$ holds and $w\in C(\mathbb{R})\cap C^1(\mathbb{R}\setminus \{0\})$ satisfies
	\begin{equation*}
		\left\{
		\begin{array}{ll}
			\displaystyle d\int_{-\infty}^{0}J(x-y)w(y)\,dy-dw(x)+a(x)w'(x)+b(x)w(x)\leq 0, & x<0, \\
			w(x)\geq 0,& x\geq 0,
		\end{array}
		\right.
	\end{equation*}
	with $a,b\in L^{\infty}_{loc}(\mathbb{R})$. If $w(x)\geq 0$ and $w(x)\not\equiv 0$ in $(-\infty, 0)$, then $w(x)>0$ for $x<0$.
\end{lemma}

\begin{proof}
Suppose that there exists $x_0<0$ such that $w(x_0) =0$. Then $w'(x_0)=0$ and it follows from the differential-integral inequality satisfied by $w$ that  at $x=x_0$,
$$
\displaystyle  d \int_{-\infty}^0 J (x_0-y) w(t,y)dy     \leq 0,
$$
which indicates that $w(y )=0$ when $y$ is close to $ x_0$, due to $J(0)>0$. This implies that $w(x) \equiv 0$ when $x<0$, since $\{ x<0 \ |\ w(x)=0\} $ is now both open and closed in $(-\infty, 0)$.
\end{proof}

This maximum principle will be used frequently. A first application is the following result.

\begin{lemma}\label{l2.3}
	Suppose that ${\bf (J)}$ and ${\bf (f_{KPP})}$ are satisfied.
	\begin{itemize}
		\item[{\rm (i)}] Assume $\phi=\phi^c$ is a nonnegative solution of \eqref{e2} with speed $c>0$. Then $\phi(x)>0$ for $x<0$ and $\phi'(x)<0$ for $x\le 0$.
		\item[{\rm (ii)}] Assume $\psi=\psi^{\tilde c}$ is a nonnegative solution of \eqref{e4} with speed $\tilde c>0$. Then $\psi(x)>0$ for $x>0$ and  $\psi'(x)>0$ for $x\ge 0$.
	\end{itemize}

\end{lemma}
\begin{proof} We only prove (i), since the proof of (ii) is similar.
	Since $\phi\geq 0$ and $\phi(-\infty)=1$, by Lemma \ref{max}  we have $\phi(x)>0$ for $x<0$.
	
	For fixed $\delta>0$, define $K=\{k\geq 1: k\phi(x-\delta)\geq \phi(x)\,\,\text{for}\, x\leq 0\}$. It follows from $\phi(-\infty)=1$, $\phi(x)>0$ for $x<0$  and $\phi(0)=0$ that $K\neq\emptyset$. Thus $k_*=\inf K\geq 1$ is well-defined. 
		
	We claim that $k_*=1$ (which implies that $\phi(x)$ is decreasing in $(-\infty, 0]$ due to the arbtrariness of $\delta>0$). Otherwise, suppose that $k_*>1$. Then $w(x):=k_*\phi(x-\delta)-\phi(x)\geq 0$, and since $w(0)=k_*\phi(-\delta)>0$ and $\lim\limits_{x\to-\infty}w(x)=k_*-1>0$, there is $x_0\in (-\infty,0)$ such that $w(x_0)=0$. From the equation satisfied by $\phi(x-\delta)$ and ${\bf (f_{KPP})}$, we have, for $x<0$,
	\begin{align*}
		0=&\ d\int_{-\infty}^{\delta}J(x-y)k_*\phi(y-\delta)\,dy-dk_*\phi(x-\delta)+ck_*\phi'(x-\delta)+k_*f(\phi(x-\delta))\\
		\geq &\ d\int_{-\infty}^{\delta}J(x-y)k_*\phi(y-\delta)\,dy-dk_*\phi(x-\delta)+ck_*\phi'(x-\delta)+f(k_*\phi(x-\delta))\\
		\geq &\ d\int_{-\infty}^{0}J(x-y)k_*\phi(y-\delta)\,dy-dk_*\phi(x-\delta)+ck_*\phi'(x-\delta)+f(k_*\phi(x-\delta)),
	\end{align*}
	and it follows that
	$$d\int_{-\infty}^{0}J(x-y)w(y)\,dy-dw(x)+cw'(x)+b(x)w(x)\leq 0,$$
	where
	$$b(x)=\left\{
	\begin{array}{ll}
		\frac{f(k_*\phi(x-\delta))-f(\phi(x))}{k_*\phi(x-\delta)-\phi(x)}, & \text{if}\, k_*\phi(x-\delta)-\phi(x)\neq 0, \\
		0, &\text{otherwise}.
	\end{array}
	\right.$$
	Note that $w(x)\not\equiv 0$ since $k_*>1$. Then by Lemma \ref{max}, $w(x)>0$ for $x<0$, a contradiction with $w(x_0)=0$. We have thus proved the claim $k^*=1$. So $\phi(x)$ is decreasing in $x\in (-\infty,0]$.
	
	It remains to show $\phi'(x)<0$ for $x\leq 0$.
	 From \eqref{e2}, we get
	\begin{equation}\label{derivative}
	\begin{aligned}	c\phi'(x)&=d\phi(x)-d\int_{-\infty}^{0}J(x-y)\phi(y)\,dy-f(\phi(x))\\
	&=d\phi(x)-d\int^{\infty}_{x}J(z)\phi(x-z)\,dz-f(\phi(x)) \mbox{ for } x<0.
	\end{aligned}
	\end{equation}
	Taking the derivative with respect to $x$ on both sides by $\phi\in C^1$, we have $$c\phi''(x)=d\phi'(x)-d\int_{-\infty}^{0}J(x-y)\phi'(y)\,dy-f'(\phi(x))\phi'(x),$$
	where we have used
	\[\begin{aligned}
	\frac{d}{dx}\int_x^\infty J(z)\phi(x-z)dz
	&=-J(x)\phi(0)+\int_x^{\infty} J(z)\phi'(x-z)dz\\
	&=\int_{-\infty}^{0}J(x-y)\phi'(y)\,dy.
	\end{aligned}
	\]
	Thus $w(x):=-\phi'(x)\geq 0$ satisfies
	$$d\int_{-\infty}^{0}J(x-y)w(y)\,dy-dw(x)+cw'(x)+f'(\phi(x))w(x)=0.$$
	Since $\phi(-\infty)=1$ and $\phi(0)=0$, we have $\phi'(x)\not\equiv 0$, that is, $w(x)\not\equiv 0$. By ${\bf (f_{KPP})}$ and Lemma \ref{max}, $w(x)=-\phi'(x)> 0$ for $x<0$. If $w(0)=0$, that is, $\phi'(0)=0$, it follows from \eqref{derivative} and $\phi(0)=0$ that
	$$0=-d\int_{-\infty}^{0}J(-y)\phi(y)\,dy<0,$$
	a contradiction.
	The proof is complete.
\end{proof}

\subsection{A perturbed semi-wave problem}

For  $\delta>0, \ c> {0}$, we consider the auxiliary problem 
\begin{equation}\label{2.1}
\begin{cases}
\dd d\int_{-\yy}^{\yy}{J} (x-y) \phi(y) {\rm d}y-d\phi+c\phi'(x)+f(\phi(x))=0,&
-\yy<x<0,\\
\phi(-\yy)=1,\ \ \phi(x)=\delta,&0\leq x<\yy.
\end{cases}
\end{equation}
If ${\delta}={0}$ then \eqref{2.1} is reduced to the semi-wave problem \eqref{e2}; therefore \eqref{2.1} can be viewed as a perturbed semi-wave problem. As we will see below, the semi-wave solutions and traveling wave solutions of \eqref{cau} can be obtained as the limit of the solution of \eqref{2.1} when $\delta\to{0}$, subject to suitable translations in $x$.

Define 
\begin{align}\label{2.7b}
\td\sigma(v):=f(v)+cM v-d v\ \ \ \ {\rm for}\ v\geq 0,
\end{align}
where $M>0$ is a constant.   Then the first equation in \eqref{2.1} is equivalent to
\begin{align}\label{2.2}
-c(e^{-Mx}\phi)'=e^{Mx}\lf[d\int_{-\yy}^{\yy} {J}(x-y)\phi(y) {\rm d}y+\td\sigma(\phi(x)) \rr].
\end{align}
Since $f$ is $C^1$,  we could choose  $M$ large enough such that $\td\sigma(v)$ is increasing for $v\in [0, 2]$, namely 
\[
\td\sigma(v)\geq \td\sigma(u) \mbox{ if } u, v\in [0, 2] \mbox{ and } v\geq u.
\]

\begin{lemma}\label{lemma2.2}
Suppose ${\bf (J)}$ and ${\bf (f_{KPP})}$ hold. Let  $\delta\in (0,1)$. Then   the problem \eqref{2.1} has a solution $\phi(x)$ which is nonincreasing in $x$, and can be obtained by an iteration process to be specified in the proof. 
\end{lemma}
\begin{proof}
Let 
\begin{align*}
\Omega:=\{\Gamma\in C(\R): 0\leq \Gamma(x)\leq 1 \mbox{ for all } x\in\R\}. 
\end{align*}
 Define an operator $P: \Omega\to  C(\R)$ by 
\begin{equation*}
P[\Gamma](x)=
\begin{cases}
\dd e^{Mx}\delta+\frac{e^{Mx}}{c}\int_{x}^{0}e^{-M\xi}\lf[d \int_{-\yy}^{\yy} {J}(\xi-y)\Gamma(y) {\rm d}y+\td\sigma(\Gamma(\xi)) \rr]{\rm d}\xi ,&x<0,\\
\delta,&x\geq  0.
\end{cases}
\end{equation*}
Using \eqref{2.2} we easily see that \eqref{2.1} is equivalent to 
\begin{equation}\label{fpt}
\begin{cases}
\phi(x)=P[\phi](x) \mbox{ for } x\in \R,\\
\phi(-\yy)=1.
\end{cases}
\end{equation}
We next solve \eqref{fpt} in three steps.

{\bf Step 1} We  show that $P$ has a fixed point in $\Omega$. 

Firstly we prove that $P[\delta](x)\geq \delta$ with  $\delta$ regarded as a constant function.
By the definition of $P$, we have $P[\delta](x)=\delta$ for $x\geq 0$. For $x<0$,   
\begin{align*}
P[\delta](x)=\ &e^{Mx}\delta+\frac{e^{Mx}}{c}\int_{x}^{0}e^{-M\xi}\lf[d\delta+\td\sigma(\delta) \rr]{\rm d}\xi\\
=\ &e^{Mx}\delta+\frac{e^{Mx}}{c}\int_{x}^{0}e^{-M\xi}\lf[cM\delta+ f(\delta) \rr]{\rm d}\xi\\
>\  & e^{Mx}\delta+\frac{e^{Mx}}{c}\int_{x}^{0}e^{-M\xi}cM\delta {\rm d}\xi\\
=\ & e^{Mx}\delta- e^{Mx}\delta+\delta=\delta 
\end{align*}
since $f(\delta)> {0}$  by ${\bf (f_{KPP})}$.

Secondly we show $P[1](x)\leq 1$.
Since $\delta>0$ is small,  $P[1](x)=\delta<1$ for $x\geq 0$.  For $x<0$, we have
\begin{align*}
P[1](x)=\ &e^{Mx}\delta+\frac{e^{Mx}}{c}\int_{x}^{0}e^{-M\xi}\lf[d+\td\sigma(1) \rr]{\rm d}\xi\\
=\ & e^{Mx}\delta+\frac{e^{Mx}}{c}\int_{x}^{0}e^{-M\xi}cM{\rm d}\xi
=e^{Mx}\delta- e^{Mx}+1<1. 
\end{align*}

Next we define inductively 
\begin{align*}
\Gamma_0(x):=\delta,\ \Gamma_{n+1}(x):=P[\Gamma_n](x)=P^n[\Gamma_0](x)\ \mbox{ for} \ \ n=0,1,2,\cdots,\ x\in \R.
\end{align*}
Then 
\begin{align*}
\Gamma_0\leq\Gamma_n\leq \Gamma_{n+1}\leq 1
\end{align*}
due to the monotonicity of $P$ which is a simple consequence of the fact that  $\td \sigma(v)$ is increasing in $v\in [0, 1]$.

Define 
\begin{align*}
\widehat \Gamma(x):=\lim_{n\to\yy }\Gamma_n(x)\in [0, 1].
\end{align*}
It is clear that  $\widehat \Gamma(x)=\delta$ for $x\geq 0$. Making use of the Lebesgue dominated convergence theorem and $\Gamma_{n+1}(x)=P[\Gamma_n](x)$, for $x<0$ we deduce 
\begin{align*}
\widehat \Gamma(x)=P[\widehat \Gamma](x),
\end{align*} 
which also implies that  $\widehat \Gamma'(x)$ exists and is continuous for $x<0$. Hence  $\widehat \Gamma$  is a fixed point of $P$ in $\Omega$.

{\bf Step 2}.  We show that $\widehat \Gamma'(x)\leq {0}$ for $x< 0$.

It suffices to prove that $\Gamma_n'(x)\leq {0}$ for $x<0$ and each $n=0,1,2,\cdots$, since this would imply each $\Gamma_n$ is nonincreasing and hence $\widehat\Gamma(x)$ is nonincreasing for $x<0$.  

It is clear that $\Gamma_0(x)=\delta$ is nonincreasing.  Assume $\Gamma_n'(x)\leq {0}$ for $x< 0$. We show that  $\Gamma_{n+1}'(x)\leq {0}$ for $x<0$. 

  By the definition,  for $x<0$,
 \begin{align*}
 \Gamma_{n+1}(x)=e^{Mx}\delta+\frac{e^{Mx}}{c}\int_{x}^{0}e^{-M\xi}g_n(\xi){\rm d}\xi, 
 \end{align*}
 where
 \begin{align*}
g_n(\xi)= g(\xi; \Gamma_{n}):=&\ d\int_{-\yy}^{\yy} {J}(\xi-y)\Gamma_{n}(y) {\rm d}y+\td\sigma(\Gamma_{n}(\xi))\\
 =&\  d\int_{-\yy}^{\yy} {J}(-y)\Gamma_{n}(y+\xi) {\rm d}y+\td\sigma(\Gamma_{n}(\xi)).
 \end{align*}
Let us note that $\Gamma_{n}'(z)\leq {0}$ for $z\neq 0$. It follows that  
$g_n(\xi)$ is differentiable for all $\xi\in \R$, and  $g_n'(\xi)\leq {0}$ for $\xi\in \R$. Moreover, 
\begin{align*}
 g_n(0)=g(0;\Gamma_{n})
\geq  g(0;\Gamma_{0})=d \delta+\td\sigma(\delta)=cM\delta+f(\delta)\geq cM\delta,
\end{align*}
since $\Gamma_{n}\geq \Gamma_{0}=\delta$, $F(\delta)> {0}$, and  $g(0;\Gamma_{n})$ is nondecreasing with respect to $\Gamma_{n}$. Therefore, for $x<0$,
\begin{align*}
 (\Gamma_{n+1})'(x)=&\ \delta Me^{Mx}+M\frac{e^{Mx}}{c}\int_{x}^{0}e^{-M\xi}g_n(\xi){\rm d}\xi-\frac{1}{c}g_n(x)\\
 =&\ \delta Me^{Mx}+M\frac{e^{Mx}}{c}\lf[\frac{-e^{-M\xi}}{M}g_n(\xi)\big |_{x}^{0}+ \int_{x}^{0}e^{-M\xi}g_n'(\xi){\rm d}\xi\rr]-\frac{1}{c}g_n(x)\\
 \leq&\ \delta Me^{Mx} +M\frac{e^{Mx}}{c}\lf[\frac{-g_n(0)}{M}+\frac{e^{-Mx}}{M}g_n(x)\rr]-\frac{1}{c}g_n(x)\\
 =&\ \delta Me^{Mx} -\frac{g_n(0)e^{Mx}}{c}\leq\delta Me^{Mx} -\delta Me^{Mx}=0.
\end{align*} 
By the principle of mathematical induction, we have  $\Gamma_n'(x)\leq {0}$ for $x< 0$ and all $n\geq 1$.

{\bf Step 3}.  We verify $\widehat \Gamma(-\yy)=1$. 

By step 2,  $\lim_{x\to-\yy}\widehat \Gamma(x)=K$ exists, and  ${0}\leq K\leq 1$.   We claim that 
\begin{align}\label{2.4}
\lim_{x\to-\yy}\int_{-\yy}^{\yy}{J}(x-y) \widehat\Gamma(y){\rm d} y=K.
\end{align}
Indeed, since $\widehat \Gamma$ is nonincreasing and $\lim_{x\to-\yy}\widehat \Gamma(x)=K$, we have
\begin{align*}
&\int_{-\yy}^{\yy}{J}(x-y)\widehat\Gamma(y){\rm d} y=\int_{-\yy}^{\yy}{J}(-y)\widehat\Gamma(y+x){\rm d} y\\
\geq & \int_{-\yy}^{-x/2}{J}(-y)\widehat\Gamma(y+x){\rm d} y\geq \widehat\Gamma(x/2)\int_{-\yy}^{-x/2}{J}(-y){\rm d} y \to K
\end{align*}
as $x\to -\yy$,  and on the other hand
\begin{align*}
\int_{-\yy}^{\yy}{J}(x-y)\widehat\Gamma(y){\rm d} y=&\int_{-\yy}^{\yy}{J}(-y)\widehat\Gamma(y+x){\rm d} y
\leq \int_{-\yy}^{\yy}{J}(-y) K{\rm d} y= K.
\end{align*}
Hence \eqref{2.4} holds. 

If $K\not=1$, then
by ${\bf({f_{KPP}})}$, we have  $f(K)\neq {0}$.
Note that  $\widehat\Gamma$ satisfies
\begin{align*}
d\int_{-\yy}^{\yy}{J}(x-y) \widehat\Gamma(y){\rm d} y-d\widehat\Gamma+c\widehat\Gamma'(x)+f(\widehat\Gamma(x))={0},\ \ \ \ 
-\yy<x<0.
\end{align*}
Letting $x\to -\yy$ and making use of \eqref{2.4}, we deduce 
\begin{align*}
\lim_{x\to-\yy}c\widehat\Gamma'(x)=-\lim_{x\to-\yy} f(\widehat\Gamma(x))=-f(K)\neq { 0},
\end{align*}
which contradicts  the fact that $\hat \Gamma$ is nonincreasing and  bounded. Thus, $\widehat \Gamma(-\yy)=1$. 

Combining Steps 1-3, we see that \eqref{2.1} admits a nonincreasing solution $\widehat \Gamma$, which is the limit of $\Gamma_n$ obtained from an iteration process.
\end{proof}

The following result describes the monotonic dependence on $c$ and $\delta$ of the solution $\phi$ to \eqref{2.1} obtained in the above lemma. To stress these dependences, we will write $\phi=\phi_\delta^c$.

\begin{lemma}\label{lemma2.3}
Suppose ${\bf (J)}$ and ${\bf (f_{KPP})}$ hold.  Let $\phi_\epsilon^c$ be the solution of \eqref{2.1}  obtained through the iteration process  in Lemma \ref{lemma2.2}, with $c>0$ and $\delta=\epsilon$. Then  
\begin{equation}\label{2.7}
\begin{cases}
\phi_{\epsilon_1}^c\leq\phi_{\epsilon_2}^c\ \  &{\rm if\ } 0< \epsilon_1\leq \epsilon_2\ll 1,\\
\phi_{\epsilon}^{c_1}\geq \phi_{\epsilon}^{ c_2}\ \  &{\rm if\ } 0<c_1\leq  c_2.
\end{cases}
\end{equation}
\end{lemma}
\begin{proof}
To verify the first inequality in \eqref{2.7} for fixed $c>0$, we adopt the definition of  $P$ and $\phi_n$  in Lemma \ref{lemma2.2}, but
in order to distinguish them between  $\delta=\epsilon_1$ and $\delta=\epsilon_2$,   we write  $P=P_{i}$
and  $\phi_{n}=\phi_{i,n}$ for $\delta=\epsilon_i$, $i=1,2$. Thus we have
\begin{align*}
\phi_{\epsilon_i}^c(x)=\lim_{n\to \yy} \phi_{i, n}(x).
\end{align*}
Since  $P[\phi](x)$ is nondecreasing with respect to $\delta$ and $\phi$, respectively, we have  
\begin{align*}
\phi_{1, n+1}(x)=P_{1}[\phi_{1,n}](x)\leq P_{1}[\phi_{2, n}](x)\leq P_{2}[\phi_{2, n}](x)=\phi_{2, n+1}(x)
\end{align*}
provided that 
\begin{align*}
\phi_{1, n}(x)\leq\phi_{2, n}(x).
\end{align*}
Since $\phi_{1, 0}(x)\equiv \epsilon_1\leq \epsilon_2 \equiv \phi_{2,0}(x)$, the above conclusion combined with the induction method gives $\phi_{1, n}(x)\leq\phi_{2, n}(x)$ for all $n=0,1,2,\cdots$, which implies $\phi_{\epsilon_1}^c(x)\leq \phi_{\epsilon_2}^c(x)$, as desired.

We now show the second inequality in \eqref{2.7} for fixed $\delta=\epsilon$.  To stress the reliance on $c_i$, we use the notions $P^{i}$ and $\phi_{n}^i$, respectively, for $P$ and $\phi$ when $c=c_i$, $i=1,2$.   
From Lemma \ref{lemma2.2},  we have for $i=1,2$,
\begin{align*}
\phi_\epsilon^{c_i}(x)=\lim_{n\to\yy} \phi_{n}^i(x)=\lim_{n\to\yy}  P^{i}[\phi_{n}^i](x).
\end{align*}
Due to $c_1\leq c_2$ and \eqref{2.1}, we have
\begin{align*}
&d\int_{-\yy}^{\yy} {J}(x-y) \phi_\epsilon^{c_1}(y)dy-d\phi_\epsilon^{c_1}+c_2(\phi_\epsilon^{c_1})'(x)+f(\phi_\epsilon^{c_1}(x))\\
\leq  &\ d\int_{-\yy}^{\yy} {J}(x-y)\phi_\epsilon^{ c_1}(y)dy-d\phi_\epsilon^{ c_1}+c_1(\phi_\epsilon^{ c_1})'(x)+f(\phi_\epsilon^{ c_1}(x))=0,
\end{align*}
which implies that 
\begin{align*}
\phi_\epsilon^{c_1}(x)\geq P^{2}[\phi_\epsilon^{c_1}](x).
\end{align*}
Since $P[\phi](x)$ is increasing with respect to $\phi$, it follows that
\begin{align*}
\phi_\epsilon^{c_1}(x)\geq P^{2}[\phi_\epsilon^{c_1}](x)\geq P^{2}[\phi_{n}^2](x)=\phi_{n+1}^2(x)
\end{align*} 
provided that
\begin{align*}
\phi_\epsilon^{c_1}(x)\geq \phi_{n}^2(x).
\end{align*}
Recall that $\phi_\epsilon^{c_1}(x)\geq  \delta\equiv \phi_{0}^2(x)$. By induction, we obtain that $\phi_\epsilon^{c_1}(x)\geq \phi_{n}^2(x)$ for all $n=0,1,2,\cdots$, and so $\phi_\epsilon^{ c_1}(x)\geq  \phi_\epsilon^{c_2}(x)$.
\end{proof}

\subsection{A dichotomy between semi-waves and traveling waves}

\begin{theorem}\label{lemma2.4}
Suppose $\mathbf{(J)}$ and $\mathbf{(f_{KPP})}$   hold. Then for each $c>0$, \eqref{cau}  has either a  monotone  semi-wave solution with speed $c$  or a monotone traveling wave solution with speed $c$, but not both. Moreover, one of the following holds:
\begin{itemize}
	\item[{\rm (i)}]  For  every $c>0$, \eqref{cau} has  a  monotone semi-wave  solution with speed $c$. 
		\item[{\rm (ii)}] For  every $c>0$, \eqref{cau} has  a  monotone traveling wave  solution with speed $c$.
		\item[{\rm (iii)}] There exists $C_*\in (0,\yy)$ such that  \eqref{cau} has   a  monotone semi-wave  solution with speed $c$ for  every $c\in(0, C_*)$, and has a monotone traveling wave solution with speed $c$ for every $c\geq  C_*$. 
\end{itemize}
\end{theorem}

The following result will be needed to  prove Theorem \ref{lemma2.4}.

\begin{lemma}\label{lem2.9}
Suppose $\mathbf{(J)}$ and $\mathbf{(f_{KPP})}$  hold. Then for each $c>0$, \eqref{cau}  has either a  monotone  semi-wave solution with speed $c$  or a monotone traveling wave solution with speed $c$, but not both.
\end{lemma}

\begin{proof}
	Let $\phi_{n}^c$ be the solution of \eqref{2.1} defined in Lemma \ref{lemma2.2} with ${\delta}=\epsilon_n$,
	$\epsilon_n\searrow 0$ as $n\to\infty$.  Then 
	\begin{align*}
	 x_{n}^c:=\max \lf\{ x:\phi_{n}^ c(x)=1/2\rr\}
	\end{align*}
	is well defined, and 
	\begin{align*}
	\phi_{n}^c(x_{n}^c)=1/2,\ \ \ \phi_{n}^c(x)<1/2\ \ \ \ {\rm for}\ x>x_{n}^c.
	\end{align*}
	Moreover, making use of Lemma \ref{lemma2.3}, we have
	\begin{equation}\label{2.8}
	\begin{cases}
	 0>x_{n}^{c}\geq   x_{m}^{c}\ \ & {\rm if}\ n\leq m,\\
	 0>x_{n}^{c_1}\geq  x_{n}^{c_2}\ \ & {\rm if}\ c_1\leq c_2.
	 \end{cases}
	\end{equation}
	Define
	\begin{align*}
	\wtd \phi_{n}^{ c}(x):= \phi_{n}^ {c}(x+x_{n}^{c}), \ \ \ \ x\in \R.
	\end{align*}
	Then $\wtd \phi_{n}^{ c}$ satisfies, for $x<-x_{n}^{c}$,
	\begin{align}\label{2.9}
		d \int_{-\yy}^{\yy}{J} (x-y) \wtd \phi_{n}^{ c}(y) {\rm d}y-d\wtd \phi_{n}^{ c}(x)+c(\wtd \phi_{n}^{ c})'(x)+f(\wtd \phi_{n}^{ c}(x))=0,
	\end{align}
	and for $x\geq -x_{n}^{c}$,
	$
	\wtd \phi_{n}^{ c}(x)=\epsilon_n$. Moreover,
	\[
	\td \phi_{n}^{ c}(0)=1/2.
	\]
	   Since $x_{n}^{c}$ is nonincreasing in $n$, 
	 \[
	 x^c:=-\lim_{n\to\infty} x_n^c\in (0,\infty]
	 \]
	  always exists, and there are  two possible cases
	\begin{itemize}
		\item Case 1. $x^c=\infty$
		\item Case 2.  $x^c\in (0,\yy)$.
	\end{itemize}
	
	Clearly, for fixed $c>0$, $\wtd \phi_{n}^{ c}(x)$ and, by the equation subsequently $(\wtd \phi_{n}^{ c})'(x)$ (for $x\neq -x_{n}^{c}$), are uniformly bounded in $n$. Then by the Arzela-Ascoli theorem and a standard argument involving a diagonal process of choosing subsequences, we see that
	$\{\wtd \phi_{n}^{ c}\}_{n\geq 1}$
has	a subsequence, still denoted by itself for simplicity of notation, which converges to 
	some $\wtd \phi^{c}\in C(\R)$  locally uniformly in $\R$. Moreover, $\wtd \phi^{c}(x)$ is nonincreasing in $x$ with $\td \phi^{c}(0)=1/2$ . 
	
	If Case 1 happens, we easily see that  $\wtd \phi^{c}$ satisfies 
	\begin{align}\label{2.10}
	d \int_{-\yy}^{\yy}{J} (x-y) \wtd \phi^{c}(y) {\rm d}y-d\wtd \phi^{c}(x)+c(\wtd \phi^{ c})'(x)+f(\wtd \phi^{c}(x))=0 \mbox{ for } x\in \R.
	\end{align}
In fact, from \eqref{2.9}, for $x\in\R$ and all large $n$ satisfying $x<-x_n^c$, we have
\begin{align*}
c\wtd \phi_{n}^{ c}(x)-c\wtd \phi_{n}^{ c}(0)=&-d \int_{0}^{x} \lf[\int_{-\yy}^{\yy}{J} (\xi -y) \wtd \phi_{n}^{ c}(y) {\rm d}y-d\wtd \phi_{n}^{ c}(\xi)+f(\wtd \phi_{n}^{ c}(\xi))\rr] {\rm d}\xi.
\end{align*}
It then follows from the dominated convergence theorem that, for $x\in\R$, 
\begin{align*}
c\wtd \phi^{c}(x)-c\wtd \phi^{c}(0)=&-d \int_{0}^{x} \lf[\int_{-\yy}^{\yy}{J} (\xi -y) \wtd \phi^{ c}(y) {\rm d}y-d\wtd \phi^{c}(\xi)+f(\wtd \phi^{c}(\xi))\rr] {\rm d}\xi,
\end{align*}
and  \eqref{2.10} thus follows by differentiating this equation.  Due to the monotonicity and boundedness of $\wtd \phi^{c}(x)$,  the  arguments in step 3 of the proof of Lemma \ref{lemma2.2} can be repeated to give 
\begin{align*}
	\lim_{x\to -\yy}\lf[d \int_{-\yy}^{\yy}{J} (x-y) \wtd \phi^{c}(y) {\rm d}y-d\wtd \phi^{c}(x)\rr]=0,
\end{align*}
and so 
\begin{align*}
	\lim_{x\to -\yy} [c (\wtd \phi_{ c})'(x)+f(\wtd \phi^{c}(x))]=0.
\end{align*}
Denote $K:=\lim_{x\to-\yy}\wtd \phi^{c}(x)\in \R_+$. Then we must have 
\begin{align*}
f(K)=\lim_{x\to -\yy} f(\wtd \phi^{c}(x))=-\lim_{x\to -\yy} c( \wtd \phi_{ c})'(x).
\end{align*}
This is possible only if $f(K)=0$. By  $\mathbf{(f_{KPP})}$ either $K={0}$ or $K=1$.  Since $\wtd \phi^{c}(x)$ is nonincreasing in $x$ with $\td \phi^{c}(0)=1/2>0$, we have $K>{0}$ and hence we must have $K=1$.  An analogous analysis can be applied to show  $\lim_{x\to\yy}\wtd \phi^{c}(x)={0}$. Therefore, $\wtd \phi^{c}(x)$ is a monotone traveling  wave of \eqref{cau} with speed $c$.

	If Case 2 happens, analogously for fixed  $x<x^c$,
\begin{align*}
c\wtd \phi^{c}(x)-c\wtd \phi^{c}(0)=&-d \int_{0}^{x} \lf[\int_{-\yy}^{\yy}{J} (\xi -y) \wtd \phi^{ c}(y) {\rm d}y-d\wtd \phi^{c}(\xi)+f(\wtd \phi^{c}(\xi))\rr] {\rm d}\xi,
\end{align*}
and $\wtd \phi^{c}(x)=0$ for $x\geq x^c$,  which yields 
\begin{equation*}
\begin{cases}
\dd	d \int_{-\yy}^{x^c}{J} (x-y) \wtd \phi^{c}(y) {\rm d}y-d\wtd \phi^{c}(x)+c(\wtd \phi^{ c})'(x)+f(\wtd \phi^{c}(x))={ 0} \mbox{ for} \ \ x<x^c,\\
	\wtd \phi^{c}(x^c)={ 0}.
\end{cases}
\end{equation*}
Let $\phi^{c}(x):=\wtd\phi^{c}(x+x^c)$ for $x\leq 0$, then $\phi^{c}(x)$ satisfies
\begin{equation*}
\begin{cases}
\dd d \int_{-\yy}^{0}{J} (x-y)\phi^{c}(y) {\rm d}y-d\phi^{c}(x)+c(\phi^{ c})'(x)+f(\phi^{c}(x))={ 0} \mbox{ for } \ x<0,\\
\wtd \phi^{c}(0)={0}.
\end{cases}
\end{equation*}
Moreover, as in Case 1, we can show $\lim_{x\to-\yy}\phi^{c}(x)=1$. Therefore, $ \phi^{c}(x)$ is a monotone semi-wave solution of \eqref{cau} with speed $c$.

We have thus proved that for any $c>0$,  \eqref{cau} has either a monotone traveling wave solution with speed $c$ or a monotone semi-wave solution with speed $c$. 
We show next that for any given $c>0$, \eqref{cau} cannot have both.

  Suppose, on the contrary, there is $c_0>0$ such that  \eqref{cau} admits a monotone traveling wave solution $\psi$ with speed $c_0$ and also a monotone semi-wave solution $\phi$ with speed $c_0$. We are going to drive a contradiction. 
  
  Let $\wtd \phi(x):=k \phi(x)$ for some fixed $k\in (0,1)$. Then by $\mathbf{(f_{KPP})}$, $\wtd \phi$ satisfies
	\begin{equation*}
	\begin{cases}
\dd	d \int_{-\yy}^{\yy}{J} (x-y) \wtd \phi(y) {\rm d}y-d \wtd \phi(x)+c \wtd \phi(x)+f(k \phi(x))\geq { 0},&x<0,\\
	\wtd \phi(-\yy)=k,\ 	\wtd \phi(x)={ 0},& x\geq 0.
	\end{cases}
	\end{equation*}
For $\beta\in \R$, define 
\begin{align*}
\psi^\beta(x):=\psi (x+\beta),\ \ \ w^\beta(x):=\psi^\beta(x)-\wtd \phi(x),  \ \ \ x\in \R.
\end{align*}	
For fixed $x\leq 0$,
\[
w^\beta(x)\geq \psi(\beta)-k \phi(x)\geq \psi(\beta)-k\to 1-k>0 \mbox{ as } \beta\to-\infty.
\]
Therefore there exists $\bar\beta\ll -1$ independent of $x$ such that
\[
w^\beta(x)>0 \mbox{ for } x\leq 0,\;\beta\leq\bar\beta.
\]
On the other hand, 
\[
w^\beta(-1)=\psi(\beta-1)-k\phi(-1)\to-k\phi(-1)<0 \mbox{ as } \beta\to\infty.
\]
Therefore we can find $\beta^*\in \R$ such that
\[
h(\beta):=\inf_{x\leq 0}w^\beta(x)> 0 \mbox{ for }  \beta< \beta^*,\; h(\beta^*)=0.
\]
Clearly $w^{\beta^*}(-\infty)=1-k>0$ and $w^{\beta^*}(0)=\psi^{\beta^*}(0)>0$. Therefore due to the continuity of $w^{\beta^*}(x)$ there exists $x^0\in (-\infty, 0)$ such that
$w^{\beta^*}(x^0)=0$. We can thus conclude that
\[
w^\beta(x)\geq 0 \mbox{ for } x\leq 0,\ \beta\leq \beta^*,\; \mbox{ and } w^{\beta^*}(x^0)=0.
\]
In particular,
\begin{equation}\label{2.11}
w^{\beta^*}(x)\geq {0} \mbox{ for } x\leq 0,\; \ \ w^{\beta^*}(x^0)=0.
\end{equation}
 
 By the definition of $\psi$ and $\phi$, we see that $w^{\beta^*}$ satisfies
	\begin{equation*}
\begin{cases}
\dd d \int_{-\yy}^{\yy}{J} (x-y)w^{\beta^*}(y) {\rm d}y-dw^{\beta^*}(x)
+c {w^{\beta^*}}'(x)\\
\ \ \ \ \ \ \ \ \ \ \ \ \ \ \ \ \ \ \ \ \ \ \ \ \ \ \ \ \ \ \ \ \ \ \ \ \ \ \ \ +f( \psi^{\beta^*}(x))-f(k \phi(x))\leq 0,&x<{0},\\
w^{\beta^*}(-\infty)=1-k>0,\; w^{\beta^*}(x)\geq 0,& x\in \R.
\end{cases}
\end{equation*}
We have
\begin{align*}
f( \psi^{\beta^*}(x))-f(k \phi(x))= C(x)w^{\beta^*}(x)
\end{align*}
with 
\[
C(x):=\int_0^1f'(k\phi(x)+tw^{\beta^*}(x))dt.
\]
 This allows us to use Lemma \ref{max} to conclude that $w^{\beta^*}(x)> {0}$ for $x<0$, which contradicts the second part of \eqref{2.11}. 
This completes the proof.
\end{proof}

\medskip

\noindent
{\bf Proof of Theorem \ref{lemma2.4}:}

From \eqref{2.8} we see that $x^c$ is nondecreasing in $c$ and hence there are three possible cases:
\begin{itemize}
	\item[{\rm (1)}]  For any $c>0$, $x^c<\yy$.
	\item[{\rm (2)}] For any $c>0$, $x^c=\yy$.
	\item[{\rm (3)}] There is $C_*>0$ such that  $x^c<\yy$  for any $c\in (0, C_*)$, and $x^c=\yy$  for any $c>C_*$.
	
\end{itemize}
 
 From the proof of Lemma \ref{lem2.9}, we know that in case (1), \eqref{cau} has a monotone semi-wave with speed $c$ for any $c>0$; in case (2), it has a monotone  traveling wave with speed $c$ for for every $c>0$; in case (3), for each $c\in (0, C_*)$ there is a monotone semi-wave solution with speed $c$, and 
  for each $c>C_*$, there is a traveling wave with speed $c$. Therefore to complete the proof  it suffices to show that in case (3) ,  \eqref{cau} has a monotone traveling wave solution with speed $c=C_*$.

Let $\psi^c$ be a monotone traveling wave solution  of \eqref{cau} with speed $c>C_*$. By a suitable translation we may assume $\psi^c(0)=1/2$. Since $\psi^c$ is uniformly bounded,  by the equation satisfied by $\psi^c$ we see that $(\psi^c)'$ is also uniformly bounded in $c$ for $c>C_*$.  Then by the Arzela-Ascoli theorem and a standard argument involving a diagonal process of choosing subsequences, 
for any sequence $c_n\searrow C_*$, $\{\psi^{c_n}\}_{n=1}^{\yy}$ has a subsequence, still denoted by itself, which converges to some
 $\psi\in C(\R)$  locally uniformly in $\R$ as $n\to\infty$. Similar to the proof of Lemma \ref{lem2.9}, we can check at once that $\psi$ satisfies
\begin{equation*}
\begin{cases}
\dd d\int_{-\yy}^{\yy}{J} (x-y)\psi(y) {\rm d}y-d\psi(x)+C_*\psi'(x)+f(\psi(x))={ 0},&
x\in \R,\\
\psi(0)=1/2.
\end{cases}
\end{equation*}
Making use of  the monotonicity of $\psi(x)$ inherited from $\psi_n^c(x)$, we can use the method in Step 3 of the proof of Lemma \ref{lemma2.2} to show that 
\begin{align*}
\psi(-\yy)=1,\ \ \psi(\yy)={ 0},
\end{align*}
which implies that $\psi$ is a monotone traveling wave solution of \eqref{cau} with speed $c=C_*$.
 The proof  is now completed.
\qed

\medskip

{\bf Remark:} In view of Theorem B, we see that case (1) of Theorem \ref{lemma2.4} happens if and only if $c_*^+=\infty$; case (2) happens if and only if $c_*^+\leq 0$; and case (3) happens if and only if $c_*^+\in (0, \infty)$, and in such a case, $C^*=c_*^+$.

\subsection{Uniqueness and strict monotonicity of semi-wave solutions to \eqref{cau}}
\begin{theorem}\label{lemma2.12}
Suppose that $\mathbf{(J)}$ and $\mathbf{(f_{KPP})}$ hold. Then for any $c>0$, \eqref{cau}  has  at most one monotone semi-wave solution $\phi=\phi^c$ with speed $c$, and when exists, $\phi^c(x)$ is strictly decreasing in $x$ for $x\in(-\infty, 0]$. Moreover, if $\phi^{c_1}$ and $\phi^{c_2}$ both exist and $0<c_1<c_2$, then $\phi^{c_1}(x)> \phi^{c_2}(x)$ for fixed $x<0$.
\end{theorem}
\begin{proof}  Assume that $\phi_1$ and $\phi_2$ are  monotone semi-wave solutions of  \eqref{cau}
with speed $c>0$. We want to show that $\phi_{1}\equiv \phi_{2}$. 

{\bf Claim 1}.   $\phi_k'(0^-)<0$ for $k=1,2$.

From the equation satisfied by $\phi_k$, we deduce, for $k=1,2$, 
\begin{equation}\label{2.18}
\begin{aligned}
&{\phi_k}'(0^-)=\lim_{x\to 0^-}\frac{{\phi_k}(x)}{x} \\
=&\lim_{x\to 0^-}\frac{1}{cx}\int_{0}^{x}\lf[-d\int_{-\yy}^{0} {J} (z-y) \phi_k(y) {\rm d}y+d\phi_k(z)-f(\phi_k(z))\rr]\rd z\\
=&-\frac{1}{c}d\int_{-\yy}^{0}{J} (-y) \phi_k(y) {\rm d}y<0.
\end{aligned}
\end{equation}

With the help of Claim 1, we are ready to define 
\begin{align*}
\rho^*:= \inf\{\rho\geq 1: \rho\phi_1(x)\geq \phi_2(x)\ {\rm for}\ x\leq 0\}.
\end{align*}
Since  $\phi_k(-\yy)=1$ for $k=1,2$,  $\frac{\phi_2(x)}{\phi_1(x)}$ is uniformly  bounded for  $x$ in a small left neighbourhood of 0 by Claim 1,  we see that
$\rho^*\in [1,\yy)$ is well-defined, and $\rho^*\phi_1(x)\geq \phi_2(x)\ {\rm for}\ x\leq 0$. 

{\bf Claim 2:} $\rho^*=1$. 

Otherwise $\rho^*>1$
  and from the definition of $\rho^*$ we can find  a sequence $x_n\in (-\infty, 0)$ such that
  \begin{equation}\label{=p}
  \lim_{n\to\infty}\frac{\phi_2(x_n)}{\phi_1(x_n)}=\rho^*>1.
  \end{equation}
  From $\phi_k(-\yy)=1$ for $k=1,2$ we see that $\{x_n\}$ must be a bounded sequence, and hence by passing to a subsequence, we may assume that $x_n\to x_*\in (-\infty, 0]$ as $n\to\infty$. Define 
\begin{align*}
V(x):=\rho^*\phi_1(x)-\phi_2(x).
\end{align*}
Clearly $V(x)\geq 0$ for $x\leq 0$. Our discussion below is organised according to the following two possibilities:
\begin{itemize}
	\item Case 1. $V(x)> 0$ for all $x< 0$.
	\item Case 2.  There exists $x_0<0$ such that $V(x_0)=0$.
\end{itemize}
In Case 1, 
 from \eqref{2.18} we obtain
\[
V'(0^-)=-\frac{1}{c}d\int_{-\yy}^{0}{J} (0-y) V(y) {\rm d}y<0.
\]

Let us examine the sequence $\{x_n\}$ in \eqref{=p}. We have $x_n\to x_*\in (-\infty, 0]$. If $x_*<0$ then we deduce
$V(x_*)=0$ which is a contradiction to $V(x)> 0$ for $x<0$. Therefore we must have $x_*=0$ and so $x_n\to 0$ as $n\to\infty$. It then follows that
\[
 \lim_{n\to\infty}\frac{\phi_2(x_n)}{\phi_1(x_n)}=\frac{\phi_2'(0^-)}{\phi_1'(0^-)}<\rho^*,
\]
due to $V'(0^-)<0$ and  $(\phi_k)'(0^-)<0$ for $k=1,2$. Thus we always arrive at a contradiction to \eqref{=p} in Case 1.

In Case 2, from the assumption $\mathbf{(f_{KPP})}$, we see that
\[
W(x):=\phi_1(x)-(\rho^*)^{-1}\phi_2(x)=(\rho^*)^{-1}V(x)
\]
 satisfies, for $x\leq 0$,
\begin{align*}
0=&\ d\int_{-\yy}^{0}{J} (x-y) W(y) {\rm d}y-d W(x)+cW'(x)+f(\phi_1(x))-(\rho^*)^{-1}f(\phi_2(x))\\
\geq  &\ d\int_{-\yy}^{0}{J} (x-y) W(y) {\rm d}y-d W(x)+cW'(x)+f(\phi_1(x))-f((\rho^*)^{-1}\phi_2(x))\\
= &\ d \int_{-\yy}^{0}{J} (x-y) W(y) {\rm d}y-dW(x)+cW'(x)+ b(x)W(x),
\end{align*}
where $b(x)$ is a bounded function.  In view of   $W(x)\geq {0}$ for $x\leq 0$, and $W(-\infty)>0$,  we can apply Lemma \ref{max} to  conclude that
\[
W(x)>0 \mbox{ for } x<0.
\]
This is  a contradiction to $W(x_0)=(\rho^*)^{-1}V(x_0)=0$.

 We have thus proved $\rho^*=1$, and so
  $\phi_1(x)\geq \phi_2(x)\ {\rm for}\ x\leq 0$. By swapping $\phi_1(x)$ with $ \phi_2(x)$ we also have $\phi_2(x)\geq \phi_1(x)\ {\rm for}\ x\leq 0$. This completes our proof for uniqueness of the semi-wave solution.

\medskip

Next we prove the strict monotonicity properties stated in the theorem. Let $\phi^c$ be a monotone semi-wave solution of \eqref{cau} with speed $c>0$. The strict monotonicity of $\phi^c(x)$ with respect to $x\leq 0$ clearly follows directly from Lemma \ref{l2.3}.
We show next that for fixed $x<0$,  $\phi^{c}(x)$ is strictly decreasing with respect to $c>0$, namely, $\phi^{c_1}(x)>\phi^{c_2}(x)$ for $c_2>c_1>0$. Denote $W(x):=\phi^{c_1}(x)-\phi^{c_2}(x)$. By Lemma \ref{lemma2.3} and the proof of Lemma \ref{lem2.9} without shifting $\phi_{n}^{c}$,  we see that  $W(x)\geq 0$ for $x\leq 0$. By $\mathbf{(f_{KPP})}$,  
\begin{align*}
f(\phi^{c_1}(x))-f(\phi^{c_2}(x))=E(x)W(x)
\end{align*}
where  $E(x)$ is a bounded function. This, combined with $c_1(\phi^{c_1})'(x)-c_2(\phi^{c_2})'(x)> c_1 W'(x)$, allows us to apply Lemma \ref{max} to conclude that $W(x)>0$ for $x<0$.
\end{proof}

\subsection{Semi-wave solution with the desired speed}

\begin{theorem}\label{thm2.12}
Suppose that $\mathbf{(J),\; (f_{KPP})}$ hold, $c_*^+\in (0, \infty]$ and $\phi^c(x)$ is the unique monotone semi-wave solution of \eqref{cau} with speed $c\in (0, c^+_*)$. Then
\begin{equation}\label{c-C*}
\lim_{c\nearrow c^+_*}\phi^c(x)=0 \mbox{ locally uniformly in } (-\infty, 0].
\end{equation}
Moreover,  \eqref{e2} and \eqref{e3} have a solution pair $(c,\phi)$  with  $\phi(x)$ monotone
		if and only if $\mathbf{(J^+_1)}$ holds. 
		And when $\mathbf{(J^+_1)}$ holds,  there exists a unique $c_0\in (0, c^+_*)$ such that $(c,\phi)=(c_0,\phi^{c_0})$ solves
		 \eqref{e2} and \eqref{e3}. 
\end{theorem}

\begin{proof}  We first prove \eqref{c-C*}.
Since $\phi^c(x)$ is decreasing  with respect to $c$, $\phi(x):=\lim_{c\nearrow c^+_*}\phi^c(x)$ is well-defined, and $\phi(x)\in[0, 1]$ for $x\leq 0$. Moreover, by the uniform boundedness of $(\phi^c)'(x)$ obtained from the equation it satisfies, the convergence of $\phi^c(x)$ to $\phi(x)$ is locally uniform in $(-\infty, 0]$. If $c^+_*=\infty$, then from
\[
\phi^c(x)=\frac{1}{c}\int_{0}^{x}\lf[-d\int_{-\yy}^{0}{J} (z-y) \phi^c(y) {\rm d}y+d\phi^c(z)-f(\phi^c(z))\rr]\rd z
\]
we immediately obtain $\phi(x)\equiv 0$. If $c^+_*<\infty$ then
$\phi$ satisfies
	\begin{equation*}
	\begin{cases}
	\dd d\int_{-\yy}^{0}{J} (x-y) \phi(y) {\rm d}y-d \phi(x)+c^+_*\phi'(x)+f(\phi(x))={0},&
	x<  0,\\
	\phi(0)={0}.
	\end{cases}
	\end{equation*}
	 Note that $\phi(x)$ is nonincreasing since $\phi^c(x)$ is.   As in  Step 3 of the proof of Lemma \ref{lemma2.2},  we can show that $\phi(-\yy)=1$ or ${0}$. By Theorem \ref{lemma2.4},  the Cauchy problem \eqref{cau} admits no monotone semi-wave solution for $c=c^+_*$, and hence necessarily $\phi(-\yy)={0}$. Thus we also have $\phi\equiv {0}$, and \eqref{c-C*} is proved.
	 
	 Next we show that if $\mathbf{(J^+_1)}$ holds, then \eqref{e2}-\eqref{e3} have a unique solution pair $(c_0,\phi^{c_0})$.  It suffices to prove that 
	 \begin{align*}
	 P(c):=c-M(c), \mbox{ with } M(c):= \mu\int_{-\yy}^{0}\int_{0}^{\yy}J(y-x)\phi^{c}(x) {\rm d}y{\rm d}x,
	 \end{align*}
	 has a unique root in $(0, c^+_*)$.
 Let us observe that when $\mathbf{(J^+_1)}$ holds, $M(c)$ is well-defined and strictly decreasing in $c$ by  Theorem \ref{lemma2.12}. Indeed,
 an elementary calculation yields
 \[
 \int_{-\yy}^{0}\int_{0}^{\yy}J(y-x){\rm d}y{\rm d}x=\int_0^\infty\int_0^\infty J(x+y)dydx=\int_0^\infty J(y)ydy,
 \]
 which implies that $M(c)$ is well-defined.
 
 Using the uniqueness of $\phi^c$, we can apply a similar convergence argument as used above to prove \eqref{c-C*} to show that $\phi^{c_n}\to \phi^{c}$ as $c_n\to c\in(0, c^+_*)$, which yields the continuity of $\phi^c(x)$ in $c\in (0,c^+_*)$ uniformly for $x$ over any bounded interval of $(-\infty, 0]$. Note that we can easily see that $\phi(x):=\lim_{c_n\to c}\phi^{c_n}(x)$ satisfies $\phi(-\infty)=1$ by comparing $\phi^{c_n}$ to some $\phi^{\hat c}$ with $\hat c\in (c, c^+_*)$ and using the monotonicity of $\phi^c$ in $c$.
 
Hence $P(c)$ is increasing and continuous in $c$. For  $c\in(0,c^+_*/2)$ close to 0, we  have $P(c)\leq  c-M(c^+_*/2)<0$, and for all $c$ close to $c^+_*$, $M(c)$ is small and hence $P(c)>0$. Thus there is a unique $c_0\in (0,c^+_*)$ such that $P(c_0)=0$.  

Finally we verify that $\mathbf{(J^+_1)}$ holds if   \eqref{e2}-\eqref{e3} have a solution pair $(c_0, \phi^{c_0})$.  Since
\begin{align*}
c_0=\mu\int_{-\yy}^{0}\int_{0}^{\yy}J(y-x)\phi^{c_0}(x) {\rm d}y{\rm d}x,
\end{align*}
 we have 
\begin{align*}
\int_{-\yy}^{0}\int_{0}^{\yy}J(y-x)\phi^{c_0}(x) {\rm d}y{\rm d}x<\yy.
\end{align*}
By Theorem \ref{lemma2.12}, $\phi^{c_0}(x)$ is decreasing in $x$. Hence, 
\begin{align*}
\int_{-\yy}^{0}\int_{0}^{\yy}J(y-x)\phi^{c_0}(x) {\rm d}y{\rm d}x\geq \phi^{c_0}(-1) \int_{-\yy}^{-1}\int_{0}^{\yy}J(y-x) {\rm d}y{\rm d}x,
\end{align*}
and so
\begin{align*}
\int_{-\yy}^{0}\int_{0}^{\yy}J(y-x) {\rm d}y{\rm d}x=&\int_{-1}^{0}\int_{0}^{\yy}J(y-x) {\rm d}y{\rm d}x+\int_{-\yy}^{-1}\int_{0}^{\yy}J(y-x) {\rm d}y{\rm d}x\\
\leq &\ 1+\int_{-\yy}^{-1}\int_{0}^{\yy}J(y-x) {\rm d}y{\rm d}x<\yy.
\end{align*}
Therefore,  $\mathbf{(J^+_1)}$ holds.
\end{proof}

\medskip

Theorem \ref{th1.2} parts $(a^+)$ and $(b^+)$ clearly follow directly from Theorems \ref{lemma2.4}, \ref{lemma2.12} and \ref{thm2.12}. 
The proof of parts $(a^-)$ and $(b^-)$ is parallel; these conclusions also follow from $(a^+)$ and $(b^+)$ by considering \eqref{e1} with $J(x)$ replaced by $J(-x)$.
\medskip

{\bf Remarks:} In the symmetric case $J(x)=J(-x)$, Theorem \ref{th1.2} was first proved in \cite{du21}. These results have been extended to rather general cooperative systems in \cite{dn22}, and much of our arguments here follow \cite{dn22} instead of \cite{du21}.

\section{Spreading speed}

We are going to determine the spreading speed of the nonlocal free boundary problem \eqref{e1}.
 For a non-symmetric $J$ satisfying ${\bf (J)}$, the following two quantities determined by $J$ and $f'(0)$ alone play an important role:
\begin{align*}
	\displaystyle	c_*^- = \sup_{\nu<0} \frac{\displaystyle d\int_{\mathbb{R}} J(x)e^{\nu x}\,dx - d + f'(0)}{\nu}, \quad
		c_*^+ = \inf_{\nu>0} \frac{\displaystyle d\int_{\mathbb{R}} J(x)e^{\nu x}\,dx - d + f'(0)}{\nu},
	\end{align*}
	
	 It can be shown that $c_*^-$ is achieved by some $\nu<0$ when it is finite, and a parallel conclusion holds for $c_*^+$.
	 It is easily checked that $c_*^-$ is finite if and only if $J$ satisfies additionally the following {\color{red}thin-tail} condition at $x=-\infty$,\smallskip

${\bf (J_{thin}^-):}$ There exists $\lambda>0$ such that $\displaystyle \int_{0}^{+\infty}J(-x)e^{\lambda x}\,dx<+\infty$.
\medskip
	
	\noindent Similarly,  $c_*^+$ is finite if and only if $J$ satisfies\smallskip
	
	${\bf (J_{thin}^+):}$ There exists $\lambda>0$ such that $\displaystyle \int_{0}^{+\infty}J(x)e^{\lambda x}\,dx<+\infty$.

\medskip

We define
\begin{equation}\label{J-thin-not}\begin{cases}
c_*^-=-\infty \mbox{ when ${\bf (J_{thin}^-)}$ does not hold},\\[3mm]
 c_*^+=+\infty \mbox{ when ${\bf (J_{thin}^+)}$ does not hold}.
 \end{cases}
\end{equation}

We say $J(x)$ is {\bf weakly non-symmetric} if
\begin{equation}\label{weak}
-\infty\leq c_*^-<0<c_*^+\leq \infty.
\end{equation}

\begin{theorem}[Spreading speed] \label{th2}
	Suppose that ${\bf(J)}$ and ${\bf(f_{KPP})}$ are satisfied, and \eqref{weak} holds. Let $(u, g, h)$ be the unique solution of \eqref{e1}, and assume that spreading occurs. Then the following conclusions are valid:
	\begin{itemize}
		\item[\rm (i)] The spreading speed of the right front $h(t)$ is given by
		\[
		\lim_{t \to \infty} \frac{h(t)}{t} = \begin{cases} c_0 & \mbox{if ${\bf(J_1^+)}$ holds},\\[2mm]
		+\infty &\mbox{if ${\bf(J_1^+)}$ does not hold},\end{cases}
		\]	
		 where $(c_0, \phi^{c_0})$ is the solution of \eqref{e2}-\eqref{e3}.
		\item[\rm (ii)] The spreading speed of the left front $g(t)$ is given by
		\[
		\lim_{t \to \infty} \frac{g(t)}{t} = \begin{cases}-\tilde c_0 & \mbox{if ${\bf(J_1^-)}$ holds},\\[2mm]
		-\infty &\mbox{if ${\bf(J_1^-)}$ does not hold},\end{cases}
		\]	
		 where $(\tilde c_0, \psi^{\tilde c_0})$ is the solution of \eqref{e4}-\eqref{e5}.
	\item[\rm (iii)] Define $c_0 = \infty$ if ${\bf(J_1^+)}$ does not hold and $\tilde c_0 = \infty$ if ${\bf(J_1^-)}$ does not hold. Then for any constants $a$ and $b$  satisfying $-\tilde c_0<a<b<c_0$, we have
		\begin{align*}
			\lim_{t\to\infty}\sup_{[at,bt]} |u(t,x)-1|=0.
		\end{align*}
	\end{itemize}
\end{theorem}

\subsection{Comparison principles revisited}

The following variations of the comparison principle in Section 1 will be used to prove Theorem \ref{th2}. Their proofs use similar techniques.

\begin{lemma}[Comparison principle 2]\label{cp3}
  Assume that conditions ${\bf (J)}$ and ${\bf (f)}$ hold, $u_0$ satisfies \eqref{102} and $(u,g,h)$ is the unique positive solution of problem \eqref{e1}. For $T\in(0,+\infty)$, suppose that $\bar{g}\in C([0,T])$, $\bar{u}(t,x)$ and $\bar{u}_t(t,x)$ are continuous for $t\in [0,T],\ x\in[\bar{g}(t),h(t)]$ and satisfy
  $\bar g(t)<h(t)$ and
\begin{equation*}
\left\{
  \begin{array}{ll}
 \displaystyle   \bar{u}_t\geq d\int_{\bar{g}(t)}^{h(t)}J(x-y)\bar{u}(t,y)\,dy-d\bar{u}+f(\bar{u}), & 0<t\leq T, x\in(\bar{g}(t),h(t)), \\
    \bar{u}(t,x)\geq 0,\  & 0<t\leq T, \ x\in\{\bar g(t), h(t)\},\\
 \displaystyle   \bar{g}'(t)\leq-\mu\int_{\bar{g}(t)}^{h(t)}\int_{-\infty}^{\bar{g}(t)}J(y-x)\bar{u}(t,x)\,dy dx, & 0<t\leq T, \\
 \displaystyle   \bar{u}(0,x)\geq u(0,x),\ \bar{g}(0)\leq g_0, & x\in[g_0,h_0].
  \end{array}
\right.
\end{equation*}
Then
$$u(t,x)\leq \bar{u}(t,x)\,\text{and}\,\, g(t)\geq \bar{g}(t)\,\,\text{for}\,\, 0<t\leq T\,\text{and}\,\,x\in [g(t), h(t)].$$
\end{lemma}

\begin{lemma}[Comparison principle 3]\label{cp4}
  Assume that conditions ${\bf (J)}$ and ${\bf (f)}$ hold, $u_0$ satisfies \eqref{102} and $(u,g,h)$ is the unique positive solution of problem \eqref{e1}. For $T\in(0,+\infty)$, suppose that $\underline{g},\underline{h}\in C([0,T])$, {\color{red}$g(t)\leq \underline{g}(t)<\underline h(t)$}, $\underline{u}(t,x)$ and $\underline{u}_t(t,x)$ are continuous for $t\in[0,T], \ x\in[\underline{g}(t),\underline{h}(t)]$ and satisfy
\begin{equation*}
\left\{
  \begin{array}{ll}
\displaystyle    \underline{u}_t\leq d\int_{\underline{g}(t)}^{\underline{h}(t)}J(x-y)\underline{u}(t,y)\,dy-d\underline{u}+f(\underline{u}),\ \underline u\geq 0, & 0<t\leq T,\ x\in(\underline{g}(t),\underline{h}(t)), \\
    \underline{u}(t,\underline{g}(t))=\underline{u}(t,\underline{h}(t))= 0, & 0<t\leq T, \\
 \displaystyle   \underline{h}'(t)\leq \mu\int_{\underline{g}(t)}^{\underline{h}(t)}\int_{\underline h(t)}^{+\infty}J(y-x)\underline{u}(t,x)\,dy dx, & 0<t\leq T, \\
    u(0,x)\geq \underline{u}(0,x), \ h(0)\geq \underline{h}(0), & x\in[\underline{g}(0),\underline{h}(0)].
  \end{array}
\right.
\end{equation*}
Then
$$u(t,x)\geq \underline{u}(t,x)\,\text{and}\,\, h(t)\geq \underline{h}(t)\,\,\text{for}\,\, 0<t\leq T\,\text{and}\,\,x\in[\underline{g}(t),\underline{h}(t)].$$
\end{lemma}

 {\bf Remark:} In the above lemmas  the assumption $u_t$ being continuous can be relaxed.
If, for each $(t,x)$, both the one-sided partial derivatives $u_t(t+0,x)$ and $u_t(t-0,x)$ exist, and the differential inequalities hold when $u_t$ is replaced by both one-sided partial derivatives, then the conclusions remain valid (see Remark 2.4 in \cite{dn22} for symmetric kernels, but the observation there still holds when the symmetry requirement for the kernel functions is dropped).

\subsection{Bounds from above}

\begin{lemma}\label{l4.2} Suppose that ${\bf(J)}$ and ${\bf(f_{KPP})}$ are satisfied, and \eqref{weak} holds. Let $(u, g,h)$ be the unique solution of \eqref{e1}.
If ${\bf (J_1^+)}$ is satisfied, then $\limsup\limits_{t\to\infty}\frac{h(t)}{t}\leq c_0$. If ${\bf (J_1^-)}$ is satisfied, then $\limsup\limits_{t\to\infty}\frac{-g(t)}{t}\leq \tilde c_0$.
\end{lemma}
\begin{proof}

Since the proofs for the estimates of $h(t)/t$ and $g(t)/t$ are similar, we only present  the proof for $g(t)/t$.

For any $\epsilon>0$, define $\delta:=2\epsilon \tilde c_0$ and
\begin{align*}
&\bar{g}(t):=-(\tilde c_0+\delta)t-L,\quad t\geq 0,\\
&\bar{u}(t,x):=(1+\epsilon)\psi(x-\bar{g}(t)),\quad x\in [\bar g,+\infty),\ t\geq 0
\end{align*}
 where  $(\tilde c_0,\psi)$ satisfies \eqref{e4}-\eqref{e5}, and  $L>0$ is a constant to be determined.

A simple comparison to the ODE problem $v'=f(v)$ with $v(0)=\|u_0\|_{\infty}$
shows that $u(t,x)\leq v(t)$ and hence
$\limsup_{t\to\infty}u(t,x)\leq 1$ uniformly in $x\in [g(t),h(t)]$. Hence there exists $T>0$ large such that
\[
\mbox{$u(T+t,x)\leq 1+{\epsilon}/{2}$ for $t\geq 0$ and $x\in [g(T+t),h(T+t)].$}
\]
Since $\psi(\infty)=1$, we can choose $L>0$ large such that $-\bar{g}(0)=L>-2g(T)$ and $\psi(\frac{L}{2})>\frac{1+\frac{\epsilon}{2}}{1+\epsilon}$.
Hence
\begin{align*}
\bar{u}(0,x)=(1+\epsilon)\psi(x+L)\geq (1+\epsilon)\psi(\frac{L}{2})>1+\frac{\epsilon}{2}\geq u(T,x)\ \mbox{ for } x\in [g(T),h(T)].
\end{align*}

Moreover, for $t\geq 0$  we have
\begin{align*}
\mu\int_{\bar{g}(t)}^{h(t+T)}\int_{-\infty}^{\bar{g}(t)}J(y-x)\bar{u}(t,x)\,dydx &\leq \mu\int_{\bar{g}(t)}^{+\infty}\int_{-\infty}^{\bar{g}(t)}J(y-x)\bar{u}(t,x)\,dydx\\
&= \mu(1+\epsilon)\int_{0}^{+\infty}\int_{-\infty}^{0}J(y-x)\psi(x)\,dydx\\
&= (1+\epsilon)\tilde c_0<\tilde c_0+\delta=-\bar{g}'(t).
 \end{align*}
 Using the equation satisfied by $\psi$, we deduce, for $t\geq 0$ and $x\in [\bar{g}(t),h(t+T)]$,
\begin{align*}
 \bar{u}_t &=(1+\epsilon)(\tilde c_0+\delta)\psi'(x-\bar{g}(t))>(1+\epsilon)\tilde c_0\psi'(x-\bar{g}(t))\\
&=(1+\epsilon)\left[d\int_{0}^{+\infty}J(x-\bar{g}(t)-y)\psi(y)\,dy-d\psi(x-\bar{g}(t))+f(\psi(x-\bar{g}(t)))\right]\\
&=d\int_{\bar{g}(t)}^{+\infty}J(x-y)\bar{u}(t,y)\,dy-d\bar{u}(t,x)+(1+\epsilon)f(\psi(x-\bar{g}(t)))\\
&\geq d\int_{\bar{g}(t)}^{h(t+T)}J(x-y)\bar{u}(t,y)\,dy-d\bar{u}(t,x)+f(\bar{u}(t,x)),
 \end{align*}
where the last inequality follows from ${\bf (f_{KPP})}$.

We may now use the comparison principle (Lemma \ref{cp3}) to conclude that $\bar{g}(t)\leq g(t+T)$ and $u(t+T,x)\leq\bar{u}(t,x)$ for $t>0$, $x\in [g(t+T),h(t+T)]$. Hence
\begin{align*}
\limsup\limits_{t\to\infty}\frac{-g(t)}{t}\leq \limsup\limits_{t\to\infty}\frac{-\bar{g}(t-T)}{t}=\tilde c_0+\delta=\tilde c_0(1+2\epsilon).
\end{align*}
 Letting $\epsilon\to 0$, we obtain the desired conclusion.
\end{proof}

\subsection{Bounds from below for compact kernels}

We first treat the case of compactly supported kernels. For the general case we will use compactly supported kernels to approximate a general kernel.

\begin{lemma}\label{l4.3}
Assume that ${\bf (f_{KPP})}$ holds, $J$ satisfies ${\bf (J)}$ and has compact support, and so ${\bf (J_1^+)}$ and ${\bf (J_1^-)}$ are satisfied automatically; then $\liminf\limits_{t\to\infty}\frac{-g(t)}{t}\geq \tilde c_0$, $\liminf\limits_{t\to\infty}\frac{h(t)}{t}\geq c_0$.
\end{lemma}
\begin{proof}
We follow the approach used to prove Lemma 3.2 of \cite{du21}.  Let $(c_0,\phi)$ and $(\tilde c_0,\psi)$ be the unique solution pair for \eqref{e2}-\eqref{e3} and \eqref{e4}-\eqref{e5}, respectively. Since $f'(1)<0$, there is a small $\delta_0>0$ such that $f'(u)<0$ for $u\in [1-\delta_0,1]$.
For $\epsilon\in (0,\delta_0]$, define
\begin{align*}
&\underline{h}(t):=(1-2\epsilon)c_0t+L,\ \underline{g}(t):=-(1-2\epsilon)\tilde c_0t-L,\\
&\underline{u}(t,x):=(1-\epsilon)\left[\phi (x-\underline{h}(t))+\psi(x-\underline{g}(t))-1\right].
\end{align*}


We claim that
$$\underline{h}'(t)\leq\mu\int_{\underline{g}(t)}^{\underline{h}(t)}\int_{\underline{h}(t)}^{+\infty}J(y-x)\underline{u}(t,x)\,dydx\,\,\text{for}\, t>0.$$
In fact, by  \eqref{e3}, we have
\begin{align*}
&\mu\int_{\underline{g}(t)}^{\underline{h}(t)}\int_{\underline{h}(t)}^{+\infty}J(y-x)\underline{u}(t,x)\,dydx\\
=&\  \mu(1-\epsilon)\int_{\underline{g}(t)-\underline{h}(t)}^{0}\int_{0}^{+\infty}J(y-x)\left[\phi (x)+\psi(x+\underline{h}(t)-\underline{g}(t))-1\right]\,dydx\\
=&\ (1-\epsilon)c_0-\mu(1-\epsilon)\int_{-\infty}^{\underline{g}(t)-\underline{h}(t)}\int_{0}^{+\infty}J(y-x)\phi (x)\,dydx\\
&-\mu(1-\epsilon)\int_{\underline{g}(t)-\underline{h}(t)}^{0}\int_{0}^{+\infty}J(y-x)\left[1-\psi(x+\underline{h}(t)-\underline{g}(t))\right]\,dydx.
 \end{align*}
 By ${\bf (J_1^+)}$, for all large $L>0$, we have
\begin{align*}
0\leq &\ \mu(1-\epsilon)\int_{-\infty}^{\underline{g}(t)-\underline{h}(t)}\int_{0}^{+\infty}J(y-x)\phi (x)\,dydx\\
\leq &\ \mu(1-\epsilon)\int_{-\infty}^{-2L}\int_{0}^{+\infty}J(y-x)\,dydx\\
<&\ \frac{1}{4}\epsilon c_0.
 \end{align*}
Since $\psi(x)$ is increasing, $\psi(\infty)=1$ and ${\bf (J_1^+)}$ holds, we deduce for all large $L$ and all $t\geq 0$,
\begin{align*}
0\leq &\  \mu(1-\epsilon)\int_{\underline{g}(t)-\underline{h}(t)}^{0}\int_{0}^{+\infty}J(y-x)\left[1-\psi(x+\underline{h}(t)-\underline{g}(t))\right]\,dydx\\
\leq &\  \mu(1-\epsilon)\int_{\underline{g}(t)-\underline{h}(t)}^{\frac{1}{2}(\underline{g}(t)-\underline{h}(t))}\int_{0}^{+\infty}J(y-x)\,dydx\\
& + \mu(1-\epsilon)\int_{\frac{1}{2}(\underline{g}(t)-\underline{h}(t))}^{0}\int_{0}^{+\infty}J(y-x)\left[1-\psi(x+\underline{h}(t)-\underline{g}(t))\right]\,dydx\\
 <&\ \frac{1}{4}\epsilon c_0+\mu(1-\epsilon)[1-\psi(L)]\int_{\frac{1}{2}(\underline{g}(t)-\underline{h}(t))}^{0}\int_{0}^{+\infty}J(y-x)\,dydx\\
<&\  \frac{1}{2}\epsilon c_0.
\end{align*}
Therefore, for  all large $L>0$,
\[
\mu\int_{\underline{g}(t)}^{\underline{h}(t)}\int_{\underline{h}(t)}^{+\infty}J(y-x)\underline{u}(t,x)\,dydx\geq (1-\epsilon)c_0-\epsilon c_0=h'(t)
\mbox{ for } t>0.
\]

 Similarly, we can show, for all large $L>0$,
 \[
 \mbox{$\underline{g}'(t)\geq \displaystyle -\mu\int_{\underline{g}(t)}^{\underline{h}(t)}\int_{-\infty}^{\underline{g}(t)}J(y-x)\underline{u}(t,x)\,dy dx$ for $t>0$.}
 \]

In the following, we verify
$$\underline{u}_t\leq d\int_{\underline{g}(t)}^{\underline{h}(t)}J(x-y)\underline{u}(t,y)\,dy-d\underline{u}(t,x)+f(\underline{u})\,\,\text{for}\,\,  t>0,\ x\in(g(t),h(t)).$$
Let us  extend $f(u)$ by defining $f(u)=f'(0)u$ for $u<0$. Since $f(1)=0$ and $f'(u)<0$ for $u\in [1-\epsilon,1]$,  we can choose $\tilde{\epsilon}>0$ small enough such that
\begin{equation}\label{e4.1}
2(1-\epsilon)f(1-\frac{\tilde{\epsilon}}{2})<f(1-\epsilon)\,\,\text{and}\,\, f'(u)<0\,\,\text{for}\, u\in [(1-\epsilon)(1-\tilde{\epsilon}),1]
\end{equation}
Fix sufficiently large $M>0$ such that $\phi (-M)>1-\frac{\tilde{\epsilon}}{2}$ and $\psi(M)>1-\frac{\tilde{\epsilon}}{2}$;
then
\begin{equation}\label{e4.2}
 \phi (x-\underline{h}(t)),\,\psi(x-\underline{g}(t))\in (1-\frac{\tilde{\epsilon}}{2},1)\,\,\text{for}\, x\in [\underline{g}(t)+M,\underline{h}(t)-M].
\end{equation}
It follows from the properties of $\phi$ and $\psi$ that
$$\inf_{x\in [-M,0]}|\phi'(x)|\geq \epsilon_0>0\,\, \text{and}\,\, \inf_{x\in [0,M]}\psi'(x)\geq \epsilon_0>0,$$
and so
\begin{equation}\label{e4.3}\begin{cases}
\psi'(x-\underline{g}(t))\geq \epsilon_0\,\, &\text{for}\ x\in [\underline{g}(t), \underline{g}(t)+M];\\[2mm]
 \phi'(x-\underline{h}(t))\leq -\epsilon_0\,\,&\text{for}\ x\in [\underline{h}(t)-M,\underline{h}(t)].
\end{cases}
\end{equation}

By  \eqref{e2}, we have
\begin{align*}
\underline{u}_t=&-(1-\epsilon)(1-2\epsilon)c_0 \phi'(x-\underline{h}(t))+(1-\epsilon)(1-2\epsilon)\tilde c_0\psi'(x-\underline{g}(t))\\
=&\ (1-\epsilon)2\epsilon c_0\phi'(x-\underline{h}(t))-(1-\epsilon)2\epsilon \tilde c_0\psi'(x-\underline{g}(t))\\
&+(1-\epsilon)\left[d\int_{-\infty}^{0}J(x-\underline{h}(t)-y)\phi(y)\,dy-d\phi(x-\underline{h}(t))+f(\phi(x-\underline{h}(t)))\right]\\
&+(1-\epsilon)\left[d\int_{0}^{+\infty}J(x-\underline{g}(t)-y)\psi(y)\,dy-d\psi(x-\underline{g}(t))+f(\psi(x-\underline{g}(t)))\right]\\
=&\ (1-\epsilon)2\epsilon\left[ c_0 \phi '(x-\underline{h}(t))-\tilde c_0\psi'(x-\underline{g}(t))\right]\\
&+d\int_{\underline{g}(t)}^{\underline{h}(t)}J(x-y)\underline{u}(t,y)\,dy-d\underline{u}(t,x)\\
&+(1-\epsilon)d\left[\int_{-\infty}^{\underline{g}(t)}J(x-y)[\phi (y-\underline{h}(t))-1]\,dy+\int_{\underline{h}(t)}^{+\infty}J(x-y)[\psi(y-\underline{g}(t))-1]\,dy\right]\\
&+(1-\epsilon)\left[f(\phi(x-\underline{h}(t)))+f(\psi(x-\underline{g}(t)))\right]\\
\leq &\ (1-\epsilon)2\epsilon \left[c_0 \phi'(x-\underline{h}(t))-\tilde c_0 \psi'(x-\underline{g}(t))\right]\\
&+d\int_{\underline{g}(t)}^{\underline{h}(t)}J(x-y)\underline{u}(t,y)\,dy-d\underline{u}(t,x)+(1-\epsilon)\left[f(\phi (x-\underline{h}(t)))+f(\psi(x-\underline{g}(t)))\right]\\
=&\  d\int_{\underline{g}(t)}^{\underline{h}(t)}J(x-y)\underline{u}(t,y)\,dy-d\underline{u}(t,x)+f(\underline{u}(t,x))+\delta(t,x),
\end{align*}
where
\begin{align*}
\delta(t,x):=& (1-\epsilon)2\epsilon \left[ c_0 \phi '(x-\underline{h}(t))-\tilde c_0\psi'(x-\underline{g}(t))\right]\\
&+(1-\epsilon)\left[f(\phi(x-\underline{h}(t)))+f(\psi(x-\underline{g}(t)))\right]-f(\underline{u}(t,x)).
\end{align*}
To verify the desired inequality, it suffices to show that $\delta(t,x)\leq 0$ for $x\in [\underline{g}(t),\underline{h}(t)]$, $t\geq 0$.

Define
\[
\mbox{$M_0:=\max\limits_{u\in [0,1]}|f'(u)|$, $\hat{\epsilon}:= 2\epsilon \min\{c_0, \tilde c_0\}\frac{\epsilon_0}{2M_0}$.}
\]

For $x\in [\underline{h}(t)-M,\underline{h}(t)]$, choose large $L>0$ such that
$$0>\psi(x-\underline{g}(t))-1\geq \psi(\underline{h}(t)-\underline{g}(t)-M)-1\geq \psi(2L-M)-1\geq -\hat{\epsilon}.$$
Then
$$f(\underline{u}(t,x))\geq f((1-\epsilon)\phi(x-\underline{h}(t)))-M_0(1-\epsilon)\hat{\epsilon},$$
$$f(\psi(x-\underline{g}(t)))=f(\psi(x-\underline{g}(t)))-f(1)\leq M_0\hat{\epsilon}.$$
It now follows from \eqref{e4.3} and ${\bf (f_{KPP})}$ that
\begin{align*}
\delta(t,x)\leq & -(1-\epsilon)2\epsilon c_0\epsilon_0+(1-\epsilon)\left[f(\phi(x-\underline{h}(t)))+M_0\hat{\epsilon}\right]\\
&-f((1-\epsilon)\phi(x-\underline{h}(t)))+M_0(1-\epsilon)\hat{\epsilon}\\
\leq &-(1-\epsilon)2\epsilon c_0\epsilon_0+2M_0(1-\epsilon)\hat{\epsilon}\leq 0.
\end{align*}

For $x\in [\underline{g}(t),\underline{g}(t)+M]$, choose large $L>0$ such that
$$0>\phi(x-\underline{h}(t))-1\geq \phi(\underline{g}(t)-\underline{h}(t)+M)-1\geq \phi(-2L+M)-1\geq -\hat{\epsilon}.$$
Then
$$f(\underline{u}(t,x))\geq f((1-\epsilon)\psi(x-\underline{g}(t)))-M_0(1-\epsilon)\hat{\epsilon},$$
$$f(\phi(x-\underline{h}(t)))=f(\phi(x-\underline{h}(t)))-f(1)\leq M_0\hat{\epsilon}.$$
Therefore by \eqref{e4.3} and ${\bf (f_{KPP})}$ we have
\begin{align*}
\delta(t,x)\leq & -(1-\epsilon)2\epsilon \tilde c_0\epsilon_0+(1-\epsilon)\left[f(\psi(x-\underline{g}(t)))+M_0\hat{\epsilon}\right]\\
&-f((1-\epsilon)\psi(x-\underline{g}(t)))+M_0(1-\epsilon)\hat{\epsilon}\\
\leq &-(1-\epsilon)2\epsilon \tilde c_0\epsilon_0+2M_0(1-\epsilon)\hat{\epsilon}\leq 0.
\end{align*}

For $x\in [\underline{g}(t)+M,\underline{h}(t)-M]$ and $t\geq 0$, by \eqref{e4.2},
$$\underline{u}(t,x)\in [(1-\epsilon)(1-\tilde{\epsilon}),1-\epsilon].$$
Using \eqref{e4.1} and \eqref{e4.2}, we have, for such $x$ and $t\geq 0$,
\[
\delta(t,x)<(1-\epsilon)[f(1-\frac{\tilde{\epsilon}}{2})+f(1-\frac{\tilde{\epsilon}}{2})]-f(1-\epsilon)<0.
\]

Since spreading happens and $\lim\limits_{t\to\infty}u(t,x)=1$ locally uniformly in $x\in\mathbb{R}$, there exists $T>0$ large enough such that
$$g(T)<-L=-\underline{g}(0),\  h(T)>L=\underline{h}(0),\  u(T,x)>1-\epsilon>\underline{u}(0,x)\,\,\text{for}\, x\in [-L,L].$$
By the comparison principle (Theorem 2.3 in Lecture 2), we obtain
$$g(t+T)\leq \underline{g}(t), h(t+T)\geq \underline{h}(t), u(t+T,x)\geq\underline{u}(t,x)\,\,\text{for}\,t>0,\ x\in[\underline{g}(t),\underline{h}(t)].$$
Hence,
\begin{align*}
&\liminf_{t\to\infty}\frac{-g(t)}{t}\geq \lim_{t\to\infty}\frac{-\underline{g}(t-T)}{t}=(1-2\epsilon)\tilde c_0,\\
&\liminf_{t\to\infty}\frac{h(t)}{t}\geq \lim_{t\to\infty}\frac{\underline{h}(t-T)}{t}= (1-2\epsilon)c_0.
\end{align*}
Letting $\epsilon\to 0$, we obtain the desired conclusions.  The proof is complete.
\end{proof}

\subsection{Convergence of semi-wave speeds}
Let \( J \) satisfy condition \({\bf (J)}\).
Assume a sequence of continuous kernel functions with compact support, denoted by \(\{J_n\}\), satisfies, for every \(n \geq 1\) and \(x \in \mathbb{R}\),
		\begin{equation}\label{Jn}
		0\leq J_n(x) \leq J_{n+1}(x) \leq J(x), \ J_n(0)>0,\quad \lim_{n \to \infty} J_n = J \quad \text{in  } L^\infty_{loc}(\mathbb{R}).
		\end{equation}
Then for each $J_n$, we may consider the semi-wave problems  \eqref{e2}-\eqref{e3} and \eqref{e4}-\eqref{e5} with $J$ replaced by $J_n$.
We note that $J_n$ satisfies ${\bf (J)}$ except that we only have $0<\int_{\mathbb R} J_n(x)dx\leq 1$.

It is easy to show that  as $n\to\infty$,
\begin{align}\label{c*(n)}
\begin{cases}	\displaystyle	c_*^-(n): = \sup_{\nu<0} \frac{d\int_{\mathbb{R}} J_n(x)e^{\nu x}\,dx - d + f'(0)}{\nu}\to c_*^-, \\
		\displaystyle c_*^+(n): = \inf_{\nu>0} \frac{d\int_{\mathbb{R}} J_n(x)e^{\nu x}\,dx - d + f'(0)}{\nu} \to c_*^+.
		\end{cases}
	\end{align}
Therefore, when $c_*^+>0$ we have $c_*^+(n)>0$ for all large $n$. Moreover, checking the proof of Theorem 4.1 in Section 4, it is easily seen  that for every such $n$, \eqref{e2}-\eqref{e3}
with $J$ replaced by $J_n$ has a unique pair of semi-wave solution $(c_n, \phi_n^{c_n})$ with $c_n\in (0, c_*^+(n))$.

Similarly, when $c_*^-<0$ then for every large $n$, \eqref{e4}-\eqref{e5}
with $J$ replaced by $J_n$ has a unique pair of semi-wave solution $(\tilde c_n, \psi_n^{\tilde c_n})$ with $\tilde c_n\in (0, -c_*^-(n))$.
\medskip

We have the following result on $\{c_n\}$ and $\{\tilde c_n\}$.

\begin{lemma}\label{uniqu}
	Let \( J \)  and $\{J_n\}$ be given as above.
	
	{\rm (i)} If $c_*^+>0$ and $(c_n, \phi_n^{c_n})$ is the semi-wave solution of \eqref{e2}-\eqref{e3} with $J$ replaced by $J_n$, which exists for every large $n$, and let $(c_0, \phi^{c_0})$ be the unique semi-wave solution of \eqref{e2}-\eqref{e3} when ${\bf (J_1^+)}$ holds, then $c_n\leq c_{n+1}$ and
	\[
	\lim_{n\to\infty} c_n =\begin{cases} c_0 \mbox{ if ${\bf (J_1^+)}$ holds,}\\[2mm]
	\infty \mbox{ if ${\bf (J_1^+)}$ does not hold.}
	\end{cases}\]
	
	{\rm (ii)} Similarly, if $c_*^-<0$ and $(\tilde c_n, \psi_n^{\tilde c_n})$ is the semi-wave solution of \eqref{e4}-\eqref{e5} with $J$ replaced by $J_n$, which exists for every large $n$, and let $(\tilde c_0, \psi^{\tilde c_0})$ be the unique semi-wave solution of \eqref{e4}-\eqref{e5} when ${\bf (J_1^-)}$ holds, then $\tilde c_n\leq \tilde c_{n+1}$ and
	\[
	\lim_{n\to\infty} \tilde c_n =\begin{cases} \tilde c_0 \mbox{ if ${\bf (J_1^-)}$ holds,}\\[2mm]
	\infty \mbox{ if ${\bf (J_1^-)}$ does not hold.}
	\end{cases}\]
\end{lemma}

\begin{proof}
	 These follow from \cite{duni22sima}, where only symmetric kernels are considered, but the proof of the conclusions used here does need the symmetry of the kernel functions. 
	 	
	For part (i), by Lemma 4.2 in \cite{duni22sima} we have $c_n\leq c_{n+1}$. The conclusion on $\lim_{n\to\infty} c_n$ follows from 
	 Proposition 4.1 in \cite{duni22sima}. The conclusions in part (ii) follow from part (i) by considering the kernel function $J(-x)$.
\end{proof}

\subsection{Completion of the proof of Theorem \ref{th2}}

	Choose a sequence of continuous kernel functions \( J_n(x) \) with compact support such that \eqref{Jn} holds.
		By \eqref{c*(n)} we know that $-\infty<c_*^-(n)<0<c_*^+(n)<\infty$ for all large $n$. By passing to a subsequence we may assume that this holds for all $n\geq 1$. We now consider \eqref{e1} with
	 \( J \) replaced by \( J_n \). If we define
	 \[
	 \tilde J_n:=J_n/\|J_n\|_{L^1(\mathbb R)},\  d_n:=d \|J_n\|_{L^1(\mathbb R)},\  f_n(u):=f(u)+(d_n-d)u, \mu_n:=\mu \|J_n\|_{L^1(\mathbb R)},
	 \]
	 then clearly \eqref{e1} with $J$ replaced by $J_n$ is the same as \eqref{e1} with $(J, f, d,\mu)$ replaced by $(\tilde J_n, f_n, d_n, \mu_n)$.
	 It is easily checked that $\tilde J_n$ satisfies ${\bf (J)}$ and $f_n$ satisfies ${\bf (f_{KPP})}$ except that $f(1)=0$ is replaced by  $f_n(1_n)=0$ for some unique $1_n$ close to 1 with $1_n\to 1$ as $n\to\infty$. Therefore for each $n\geq 1$ this new problem
	 has a unique solution \( (u_n, g_n, h_n) \). Since by assumption spreading happens to \eqref{e1}, it is easy to show that for all large $n$, spreading also happens to \( (u_n, g_n, h_n) \).
	
	 Analogously the corresponding semi-wave problems \eqref{e2}-\eqref{e3} and \eqref{e4}-\eqref{e5} with \( J \) replaced by \( J_n \) have unique solution pairs  \( (c_0^n, \phi^{c_0^n}) \) and \( (\tilde c_0^n, \psi^{\tilde c_0^n}) \), respectively. Moreover, we can apply Lemma \ref{l4.3} to conclude that
	 \begin{align*}
		\liminf_{t \to \infty} \frac{h_n(t)}{t} \geq c_0^n, \quad \liminf_{t \to \infty} \frac{-g_n(t)}{t} \geq \tilde c_0^n.
	\end{align*}
	Furthermore, by the comparison principle, we have $h(t) \geq h_{n+1}(t) \geq h_n(t)$,  $g(t) \leq g_{n+1}(t) \leq g_n(t)$ and
 \begin{equation}\label{ine}
 u_n(t,x)\leq u_{n+1}(t,x)\leq u(t,x)
 \end{equation}
 for all  $t > 0$ and   $n \geq 1$.	Hence, for every large \( n \),
	\begin{align*}
		\liminf_{t \to \infty} \frac{h(t)}{t} \geq c_0^n, \quad \liminf_{t \to \infty} \frac{-g(t)}{t} \geq \tilde c_0^n.
	\end{align*}
	 We may now apply Lemma \ref{uniqu} to  conclude that
	\begin{align*}
		\lim_{n \to \infty} c_0^n = \begin{cases} \infty \quad & \text{if } {\bf (J_1^+)} \text{ does not hold}, \\
		 c_0 \quad &\text{if } {\bf (J_1^+)} \text{ holds}.
		 \end{cases}
	\end{align*}
	It follows that
	\begin{align*}
		\liminf_{t \to \infty} \frac{h(t)}{t} \geq \begin{cases} \infty \quad & \text{if } {\bf (J_1^+)} \text{ does not hold}, \\
		 c_0 \quad &\text{if } {\bf (J_1^+)} \text{ holds}.
		 \end{cases}
	\end{align*}
	Combining this with  Lemma \ref{l4.2}, we obtain
	\begin{align*}
		\lim_{t \to \infty} \frac{h(t)}{t} = \begin{cases} \infty \quad & \text{if } {\bf (J_1^+)} \text{ does not hold}, \\
		 c_0 \quad &\text{if } {\bf (J_1^+)} \text{ holds}.
		 \end{cases}
	\end{align*}
	
	Similarly we can show
	\[
		\lim_{t \to \infty} \frac{g(t)}{t} = \begin{cases}-\tilde c_0 & \mbox{if ${\bf(J_1^-)}$ holds},\\
		-\infty &\mbox{if ${\bf(J_1^-)}$ does not hold}.\end{cases}
		\]

Finally we consider the limit of the density function $u(t,x)$ as $t\to\infty$. If both ${\bf (J_1^+)}$ and ${\bf (J_1^-)}$ hold, then by the definition of $\underline{u}$ in the proof of Lemma \ref{l4.3},
for any small $\epsilon>0$,
\[\mbox{$\liminf\limits_{t\to\infty}\min\limits_{(at,bt)}\underline{u}(t,x)\geq 1-\epsilon$, provided that $-\tilde c_0<a<b<c_0$.}
\]
Consequently, $\liminf\limits_{t\to\infty}\min\limits_{(at,bt)}u(t,x)\geq 1-\epsilon$.
Since $\epsilon$ can be arbitrarily small and $\limsup\limits_{t\to\infty}u(t,x)\leq 1$ uniformly in $x\in [g(t), h(t)]$, it follows that
$$\lim\limits_{t\to\infty}\max\limits_{(at,bt)}|u(t,x)-1|=0.$$

If neither ${\bf (J_1^+)}$ nor ${\bf (J_1^-)}$ holds, then we choose a sequence of compactly supported kernels $\{J_n\}$ as at the beginning of this proof,
so that the above conclusions for $(u,g,h)$ applies to the corresponding solution $(u_n, g_n, h_n)$ for each large $n$, and therefore for every small $\delta>0$
\[
 \lim\limits_{t\to\infty}\max\limits_{((-\tilde c_0^n+\delta)t,(c_0^n-\delta)t)}|u_n(t,x)-1_n|=0.
\]
Since $c_0^n\to\infty$, $\tilde c_0^n\to\infty$  as $n\to\infty$ and $u(t,x)\geq u_n(t,x)$ for $t>0,\ x\in [g_n(t), h_n(t)]$, it follows that
$$\liminf\limits_{t\to\infty}\min\limits_{(at,bt)}u(t,x)\geq \liminf\limits_{t\to\infty}\min\limits_{(at,bt)}u_n(t,x)=1_n$$
for any $a<b$.

Since  $1_n\to 1$ as $n\to\infty$ and $\limsup\limits_{t\to\infty}u(t,x)\leq 1$ uniformly in $x$, we thus obtain
$\lim\limits_{t\to\infty}\max\limits_{(at,bt)}|u(t,x)-1|=0$ for such $a$ and $b$.

The remaining cases can be similarly proved, and  we omit the details.
\qed

\medskip

{\bf Remarks:} In the symmetric case $J(x)=J(-x)$, Theorem \ref{th2} was first proved in \cite{du21}. These results have been extended to rather general cooperative systems in \cite{dn22}.

	\section{ Precise rate of acceleration}

We  determine the acceleration rate of the nonlocal free boundary problem \eqref{e1}.
 We  will prove the following result, which is taken from \cite{dn24}.

\begin{theorem}\label{thm1} Suppose that ${\bf(J)}$ and ${\bf(f_{KPP})}$ are satisfied, and {\color{red}$J$ is symmetric}: $J(x)=J(-x)$. Let $(u, g, h)$ be the unique solution of \eqref{e1}, and assume that spreading occurs.
Then the following conclusions hold:

\begin{itemize}\item[(i)]
  If 
\[
\lim_{|x|\to\infty}J(x)|x|^\alpha=\lambda\in (0,\infty)\ {\rm for\ some }\ \alpha\in (1,2],
\]
then
\[
\begin{cases}\dd\lim_{t\to\infty}\frac{h(t)}{t\ln t}=\lim_{t\to\infty}\frac{-g (t)}{t\ln t}=\mu\lambda ,& \mbox{ when } \alpha=2,\\[2.8mm]
\dd\lim_{t\to\infty}\frac{h(t)}{ t^{1/(\alpha-1)}}= \lim_{t\to\infty}\frac{-g (t)}{ t^{1/(\alpha-1)}}=\left[\frac{2^{2-\alpha}}{2-\alpha}\mu\lambda\right]^{1/(\alpha-1)}, & \mbox{ when }  \alpha\in (1,2), \end{cases}
\]
and for any small $\epsilon>0$,
\[
\lim_{t\to\infty}u(t,x)=1 \mbox{ uniformly for } x\in [(1-\epsilon)g(t), (1-\epsilon)h(t)].
\]

 \item[(ii)] If 
\[
\lim_{|x|\to\infty}J(x)|x| (\ln |x|)^\beta=\lambda\in (0,\infty)\ {\rm for\ some }\ \beta\in (1,\yy),
\]
then 
\begin{align*}
\dd  \lim_{t\to\yy}\frac{\ln h(t)}{t^{1/\beta}}= \lim_{t\to\yy}\frac{\ln [-g (t)]}{t^{1/\beta}}= \lf( \frac{2\beta\mu \lambda}{\beta-1}\rr)^{1/\beta},
\end{align*}
namely,
\[-g (t), h(t)=\exp\Big\{\Big[\Big(\frac{2\beta \mu\lambda }{\beta-1}\Big)^{1/\beta}+o(1)\Big]t^{1/\beta}\Big\} \mbox{ as } t\to\infty.
\]
Moreover, for any small $\epsilon>0$,
\[
\lim_{t\to\infty}u(t,x)=1 \mbox{ unifromly for } |x|\leq \exp\big[(1-\epsilon)\big(\frac{2\beta \mu\lambda }{\beta-1}\big)^{1/\beta}t^{1/\beta}\big]
\]
\end{itemize}
\end{theorem}

If $J(x)$ is not symmetric, the rate of acceleration for \eqref{e1} is considered in \cite{DFN2}. We only consider the symmetric case here for simplicity. 

Note that when $J$ is symmetric, in Theorem B we have $c_0=\tilde c_0$, and we will simply say  ${\bf (J_1)}$ holds when
${\bf (J_1^+)}$ (and hence ${\bf (J_1^-)}$) holds. Moreover, if additionally
\[
\lim_{|x|\to\infty}J(x)|x|^\alpha=\lambda\in (0,\infty)\ {\rm for\ some }\ \alpha>0,
\]
then ${\bf (J_1)}$ holds if and only if $\alpha>2$, and ${\bf (J)}$ holds only if $\alpha>1$.
If 
\[
\lim_{|x|\to\infty}J(x)|x| (\ln |x|)^\beta=\lambda\in (0,\infty)\ {\rm for\ some }\ \beta>0,
\]
then ${\bf (J_1)}$ can never hold, and ${\bf (J)}$ holds only if $\beta>1$.
\newpage

 We will consider a more general class of symmetric $J(x)$ than those in Theorem \ref{thm1}; namely $J$ satisfies {\bf (J)} and either 
 {\small \begin{align}\label{8.1}
 \mbox{ for some } \alpha\in (1,2],\  J(x)\sim |x|^{-\alpha},\  i.e.,\	\begin{cases} \underline \lambda:=\liminf_{|x|\to\infty}J(x)|x|^\alpha>0,\\[2mm]
   \bar\lambda:=\limsup_{|x|\to\infty}J(x)|x|^\alpha<\infty,
\end{cases}
\end{align}
or
 \begin{align}\label{jbeta}
	\mbox{ for some } \beta>1,\ J(x)\sim \Big[|x|(\ln |x|)^\beta\Big]^{-1},\ i.e., \	\begin{cases} \underline \lambda:=\liminf_{|x|\to\infty}J(x)|x|(\ln |x|)^{\beta}>0,\\[2mm]
		\bar\lambda:=\limsup_{|x|\to\infty}J(x)|x|(\ln |x|)^{\beta}<\infty.\end{cases}
\end{align}}
We will prove some sharp estimates  under the above assumptions for $J$; Theorem 1.1 is a direct consequence of these 
more general results.

 \subsection{Some preparatory results}
  
\begin{lemma}\label{lemma5.1}
 For $k>1$, $\delta\in [0,1)$, define 
{\small \begin{equation*}
\dd A=A(k,\delta,J):=
\begin{cases}
\dd \int_{-k}^{-\delta k}\int_{0}^\yy  J(x-y)\rd y\rd x &\mbox{ if \eqref{8.1} holds with $ \alpha\in (1,2)$ or if \eqref{jbeta} holds},\\[3mm]
\dd \int_{-k}^{-k^\delta}\int_{0}^\yy  J(x-y)\rd y\rd x &\mbox{ if \eqref{8.1} holds with}\ \alpha=2.
\end{cases}
\end{equation*}
Then 
\[\begin{aligned}
&\begin{cases}
\dd	\liminf_{k\to \yy} \frac A{k^{2-\alpha}}\geq \frac{1-\delta^{2-\alpha}}{(\alpha-1)(2-\alpha)}\underline \lambda,  \\[3mm]
\dd	 \limsup_{k\to \yy} \frac A{k^{2-\alpha}} \leq \frac{1-\delta^{2-\alpha}}{(\alpha-1)(2-\alpha)} \bar \lambda, \end{cases}&& \mbox{ if \eqref{8.1} holds with $ \alpha\in (1,2)$},\\
&	 \begin{cases}
\dd	 \liminf_{k\to \yy} \frac{A}{\ln k}\geq (1-\delta)\underline \lambda, \\[3mm]
\dd	 \limsup_{k\to \yy} \frac{A}{\ln k}	\leq (1-\delta)\bar \lambda , \end{cases}
&&\mbox{ if \eqref{8.1} holds with } \alpha=2,\\
&\begin{cases}
	\dd	\liminf_{k\to\yy}\frac{A}{k(\ln k)^{1-\beta}}\geq  \frac{(1-\delta)\underline\lambda}{\beta-1},  \\[3mm]
	\dd	 	\limsup_{k\to\yy}\frac{A}{k(\ln k)^{1-\beta}}\leq \frac{(1-\delta)\bar\lambda}{\beta-1}, \end{cases}&& \mbox{ if \eqref{jbeta} holds}.
\end{aligned}
\]}
\end{lemma}
\begin{proof}
	
{\bf Case 1}: \eqref{8.1} holds with $\alpha\in (1,2)$. 

Denote
\begin{align}\label{Dm}
	D_\delta:=\frac{1 }{\alpha-1}\int_{0}^{\yy}[(y+\delta)^{1-\alpha}-(y+1)^{1-\alpha}]\rd y.
\end{align}
A direct calculation gives
\[
D_\delta=\lim_{M\to\infty}\frac{(M+\delta)^{2-\alpha}-(M+1)^{2-\alpha}+1-\delta^{2-\alpha}}{(\alpha-1)(2-\alpha)}=\frac{1-\delta^{2-\alpha}}{(\alpha-1)(2-\alpha)}.
\]
Moreover,
\begin{align*}
	A=&\int_{-k}^{-\delta k}\int_{0}^\yy J(x-y)\rd y\rd x=\int_{\delta k}^{k}\int_{0}^\yy  J(x+y)\rd y\rd x\\
	=&\int_{\delta k}^{k}\int_{0}^2  J(x+y)\rd y\rd x+\int_{\delta k}^{k}\int_{2}^\yy  J(x+y)\rd y\rd x=:A_1+A_2,
\end{align*}
and by {\bf (J)},
\begin{align*}
	0\leq A_1\leq \int_{0}^2 1 \rd y\leq 2.
\end{align*}
Clearly,
\begin{align*}
		A_2&= \int_{\delta k}^{k}\int_{2}^\yy  J(x+y)\rd y\rd x=  \int_{2}^\yy \int_{\delta k}^{k}  J(x+y)\rd x \rd y\\
	&=k^{2-\alpha}\left(\int_{2k^{-1}}^{k^{-1/2}}+\int_{k^{-1/2}}^\yy\right) \int_{\delta+y}^{1+y}  \frac{ J(kx)}{(kx)^{-\alpha}}x^{-\alpha}\rd x \rd y.
\end{align*}

We have
\begin{align*}
0\leq &\int_{2k^{-1}}^{k^{-1/2}}\int_{\delta+y}^{1+y}  \frac{ J(kx)}{(kx)^{-\alpha}}x^{-\alpha}\rd x \rd y\leq \sup_{\xi\geq 1}[J(\xi)\xi^\alpha] \int_{2k^{-1}}^{k^{-1/2}}\int_{\delta+y}^{1+y}  x^{-\alpha}\rd x \rd y\\
\leq&\ \frac{\sup_{\xi\geq 1}[J(\xi)\xi^\alpha]}{\alpha-1} \int_{2k^{-1}}^{k^{-1/2}}y^{1-\alpha} \rd y\to 0 \ \mbox{ as $k\to \yy$. }
\end{align*}
By this and \eqref{8.1}, we deduce 
\begin{align*}
	&\limsup_{k\to\yy}\frac{A_2}{k^{2-\alpha}}=\limsup_{k\to\yy}\int_{k^{-1/2}}^\yy\int_{\delta+y}^{1+y}  \frac{ J(kx)}{(kx)^{-\alpha}}x^{-\alpha}\rd x \rd y\\
	&\leq \bar \lambda\int_{0}^\yy \int_{\delta+y}^{1+y}  x^{-\alpha}\rd x \rd y=\frac{\bar\lambda}{\alpha-1}\int_{0}^\yy  [(\delta+y)^{1-\alpha}-(1+y)^{1-\alpha}] \rd y=\bar\lambda D_\delta.
\end{align*}
Thus,
\begin{align*}
		\limsup_{k\to\yy}\frac{A}{k^{2-\alpha}}=\limsup_{k\to\yy}\frac{A_2}{k^{2-\alpha}}\leq \bar\lambda  D_\delta.
\end{align*}

Similarly,
\begin{align*}
		\liminf_{k\to\yy}\frac{A}{k^{2-\alpha}}=\liminf_{k\to\yy}\frac{A_2}{k^{2-\alpha}}\geq \underline\lambda  D_\delta.
\end{align*}

{{\bf Case 2}: \eqref{jbeta} holds. 

Let $A_1$ and $A_2$ be as in Case 1. Clearly, $0\leq A_1\leq 2$. A simple calculation gives
\begin{align*}
		A_2&= \int_{\delta k}^{k}\int_{2+x}^\yy  J(y)\rd y\rd x= \int_{\delta k+2}^{k+2}\int_{\delta k}^{y-2}  J(y)\rd x\rd y+\int_{k+2}^\yy \int_{\delta k}^{k}  J(y)\rd x\rd y\\
			&=\int_{\delta k+2}^{k+2}(y-2-\delta k)  J(y)\rd y+(1-\delta)k\int_{k+2}^\yy  J(y)\rd y.
\end{align*}
By  \eqref{jbeta}, there exists $C>0$ such that for all large $k>0$,
\begin{align*}
&\int_{\delta k+2}^{k+2}(y-2-\delta k)  J(y)\rd y\leq C\int_{\delta k+2}^{k+2} (\ln y)^{-\beta}\rd y\leq C(1-\delta)k\Big[\ln(\delta k+2)\Big]^{-\beta},
\end{align*}
and
\begin{align*}
k\int_{k+2}^\yy  J(y)\rd y\leq  \bar\lambda[1+o_k(1)]k\int_{k+2}^\yy  y^{-1}(\ln y)^{-\beta}\rd y=\frac{\bar\lambda[1+o_k(1)]k}{\beta-1}[\ln (k+2)]^{1-\beta},
\end{align*}
where $o_k(1)\to 0$ as $k\to\infty$. 
Hence,
\begin{align*}
	\limsup_{k\to\yy}\frac{A}{k(\ln k)^{1-\beta}}=\limsup_{k\to\yy}\frac{A_2}{k(\ln k)^{1-\beta}}\leq \frac{(1-\delta)\bar\lambda}{\beta-1}.
\end{align*}

Similarly,
\begin{align*}
		\liminf_{k\to\yy}\frac{A}{k(\ln k)^{1-\beta}}=\liminf_{k\to\yy}\frac{A_2}{k(\ln k)^{1-\beta}}\geq  \frac{(1-\delta)\underline \lambda}{\beta-1}.
\end{align*}}

{\bf Case 3}: \eqref{8.1} holds with $\alpha=2$. 

  By direct calculation,
\begin{align*}
	A=&\int_{-k}^{-k^\delta}\int_{0}^\yy  J(x-y)\rd y\rd x=\int_{k^\delta}^{k}\int_{0}^\yy J(x+y)\rd y\rd x\\
	=&\int_{k^\delta}^{k}\int_{0}^1  J(x+y)\rd y\rd x+\int_{k^\delta}^{k}\int_{1}^\yy  J(x+y)\rd y\rd x=:\wtd A_1+\wtd A_2,
	\end{align*}
and by {\bf (J)},
\begin{align*}
	0\leq \wtd A_1\leq \int_{0}^1 1 \rd y=1.
\end{align*}
By \eqref{8.1}, we have
\begin{align*}
	&\wtd A_2=\int_{1}^{\yy}\int_{k^{\delta}+y}^{k+y}  J(x)\rd x\rd y\leq \bar\lambda[1+o_k(1)] \int_{1}^{\yy}\int_{k^{\delta}+y}^{k+y}  x^{-2}\rd x\rd y
	= \bar\lambda[1+o_k(1)] \ln\left(\frac{k+1}{k^\delta+1}\right),
\end{align*}
where $o_k(1)\to 0$ as $k\to\infty$. Therefore,
\begin{align*}
	\limsup_{k\to \yy}\frac{A}{\ln k}=\limsup_{k\to \yy}\frac{\wtd A_2}{\ln k}\leq \lim_{k\to\infty}\bar\lambda\ \frac{\ln(k+1)-\ln(k^\delta+1)}{\ln k}=(1-\delta)\bar\lambda.
\end{align*}

Similarly,
\begin{align*}
	\liminf_{k\to \yy}\frac{A}{\ln k}=\liminf_{k\to \yy}\frac{\wtd A_2}{\ln k}\geq \lim_{k\to\infty}\underline\lambda\ \frac{\ln(k+1)-\ln(k^\delta+1)}{\ln k}=(1-\delta)\underline\lambda.
\end{align*}
The proof is finished.
\end{proof}

\begin{lemma}\label{lem2.2a}
	Suppose that  $J$ satisfies {\rm \textbf{(J)}} but neither \eqref{8.1} nor \eqref{jbeta} is required. Let $1<\xi(t)<L(t)$ be  functions in $C([0,\infty))$, $\rho\geq 2$ a constant, and define
	\begin{align*}
		\phi(t,x):=\min\left\{1,  \Big[1-\frac{|x|}{L(t)}\Big]^\rho\xi(t)^\rho\right\} \mbox{ for } x\in [-L(t),L(t)],\ t\in [0,\infty).
	\end{align*}
 Then, for any $\epsilon\in (0,1)$, there exists a  constant $\theta^*=\theta^*(\epsilon,J)>1$, such that 	
 \begin{align}\label{2.10b}
		\int_{-L(t)}^{L(t)}J(x-y)\phi (t,y)\rd y\geq (1-\epsilon) \phi(t,x) \mbox{ for }  x\in [-L(t),L(t)],\ t\geq 0
	\end{align}
	provided that
 \begin{align}\label{2.9a}
L(t)\geq \theta^*\xi(t) \ \ \ {\rm for\ all}\ t\geq 0.
 \end{align}
\end{lemma}

\begin{proof} 
	Since $||J||_{L^1}=1$, there is $L_0>0$ depending only on  $J$ and $\epsilon$ such that
	\begin{align}\label{2.4a}
		\int_{-L_0}^{L_0}J(x) \rd x\geq 1-\epsilon/2.
	\end{align}
	Define
	\[
	\psi(t,x):=\phi(t, L(t)x)=\min\left\{1,  (1-|x|)^\rho\xi(t)^\rho\right\},\ \ x\in [-1,1],\ t\geq 0.
	\]
	We note that $\rho\geq 2$ implies that  $\psi(t,x)$ is a convex function of $x$ when
\begin{align*}
	1-\frac 1{\xi(t)}\leq |x|\leq 1.
\end{align*}
 Clearly
  \begin{align*}
		\psi(t,x)=\begin{cases} 1 &\mbox{ for } |x|\leq 1-\xi(t)^{-1},\\
		\big[(1-|x|) \xi(t) \big]^{\rho} &\mbox{ for } 1-\xi(t)^{-1}<|x|\leq 1.
		\end{cases}
\end{align*}
It is also easy to check that
\begin{align*}
	\frac{|\psi(t,x)-\psi(t,y) |}{|x-y|}\leq M(t):=  \rho\xi(t) \mbox{ for } x, y\in [-1,1],\ x\not=y,\ t\geq 0,
\end{align*}
which implies
	\begin{align}\label{2.5a}
		|\phi(t,x)-\phi(t,y)|=|\psi(t, x/L(t))-\psi(t, y/L(t))|\leq \frac{M(t)}{L(t)} |x-y|
	\end{align}
	  for $x,y\in[-L(t), L(t)]$.
	
Since $\psi(t, x)>0$ for $x\in (-1,1)$, $\psi(t,\pm 1)=0$, and $\psi(t,x)$ is convex in $x$ for $x\in [-1, -1+1/\xi(t)]$ and for $x\in [1-1/\xi(t), 1]$, if we  extend $\psi (t,x)$ by $\psi(t,x)=0$ for $|x|>1$, then
 \begin{align*}
 	\psi(t, x)\mbox{ is convex  for $ x\in [1-1/\xi(t),\infty)$ and for $x\in (-\infty, -1+1/\xi(t)]$}.
 \end{align*}

	We now verify \eqref{2.10b} for $x\in [0, L(t)]$; the proof for  $x\in [-L(t), 0]$ is parallel and will be omitted.
	We will divide the proof into two cases:
	\[\mbox{ {\bf (a)}\ $x\in \lf[0, (1-\frac{1}{2\xi(t)})L(t)\rr]$ and  {\bf (b)}\ $x\in \lf[(1-\frac{1}{2\xi(t)})L(t), L(t)\rr]$.}
	\]
	
	{\bf Case (a)}. For
	\begin{align*}
		x\in \lf[0, (1-\frac{1}{2\xi(t)})L(t)\rr],
	\end{align*}
	a direct calculation gives
	\begin{align*}
		\int_{-L(t)}^{L(t)}J(x-y)\phi(t, y)\rd y=\int_{-L(t)-x}^{L(t)-x}J(y)\phi (t, x+y)\rd y\geq \int_{-L_0}^{L_0}J(y)\phi (x+y)\rd y,
	\end{align*}
	where $L_0$ is given by \eqref{2.4a} and we have used
	\[
	L(t)-x\geq \frac{L(t)}{2\xi(t)}\geq L_0, \mbox{ which is guaranteed if we assume } L(t)\geq {2L_0}{\xi(t)}.
	\]
	Then by  \eqref{2.4a}, \eqref{2.5a} and {\bf (J)},
	\begin{align*}
		&\int_{-L_0}^{L_0}J(y)\phi (t, x+y)\rd y\\
		=&\ \int_{-L_0}^{L_0}J(y)\phi (t,x)\rd y+\int_{-L_0}^{L_0}J(y)[\phi (t,x+y)-\phi (t,x)]\rd y\\
		\geq &\ \int_{-L_0}^{L_0}J(y)\phi (t,x)\rd y-\frac{M(t)}{L(t)}\int_{-L_0}^{L_0}J(y)|y|\rd y\\
		\geq &\ (1-\epsilon/2)\phi(t, x)-\frac{M(t)}{L(t)}L_0.
	\end{align*}
	Clearly
	\begin{align*}
		M_1(t):=\min_{x\in [0,(1-\frac{1}{2\xi(t)})L(t)]}\phi(t, x)=\Big(\frac12\Big)^\rho.
	\end{align*}
	  Then from the above calculations we obtain,   for $ x\in [0, (1-\frac{1}{2\xi(t)})L(t)]$,
	\begin{align*}
		\int_{-L(t)}^{L(t)}J(x-y)\phi (t,y)\rd y&\geq (1-\epsilon/2)\phi (t, x)-\frac{M(t)}{L(t)}L_0\\
		&= (1-\epsilon)\phi (t, x)+\frac\epsilon 2 \phi (t, x)  -\frac{M(t)}{L(t)}L_0\\
		&\geq (1-\epsilon)\phi (t, x)+ \frac\epsilon 2 M_1(t)-\frac{M(t)}{L(t)}L_0\geq  (1-\epsilon)\phi (t, x)
	\end{align*}
provided that
\begin{align*}\label{2.8a}
	{L(t)}\geq \frac{2L_0M(t)}{\epsilon M_1(t)}=\frac{2^{\rho+1}L_0\rho}{\epsilon}\xi(t).
\end{align*}

	{\bf Case (b)}. For
	\begin{align*}
		x\in \lf[(1-\frac 1{2\xi(t)})L(t),  L(t)\rr],
	\end{align*}
	we have, using $-L(t)-x<-L_0$ and $\phi(t, x)=0$ for $x\geq L(t)$,
	\begin{align*}
		\int_{-L(t)}^{L(t)}J(x-y)\phi (t, y)\rd y&\geq \int_{-L_0}^{\min\{L_0, L(t)-x\}}J(y)\phi (t, x+y)\rd y\\
			&=\int_{-L_0}^{L_0}J(y)\phi (t, x+y)\rd y\\
				&=\int_{0}^{L_0}J(y)[\phi (t, x+y)+\phi(t,x-y)]\rd y.
\end{align*}
Since $\phi(t,s)$ is convex in $s$ for $s\geq L(t)[1-\xi(t)^{-1}]$, and for $x\in \lf[(1-\frac 1{2\xi(t)})L(t), L(t)\rr]$, $y\in [0, L_0]$,  we have
\[\mbox{ $x+y\geq x-y\geq (1-\frac 1{2\xi(t)})L(t)-L_0\geq (1-\frac 1{\xi(t)})L(t)$ by our earlier assumption  $L(t)\geq 2L_0\xi(t)$.}
\]
Then, we can use the convexity of $\phi(t,\cdot)$ and \eqref{2.4a} to obtain
\begin{align*}
	\int_{0}^{L_0}J(y)[\phi (t,x+y)+\phi (t,x-y)]\rd y\geq 	2\phi(t,x)\int_{0}^{L_0}J(y)\rd y\geq (1-\epsilon/2) \phi(t,x).
\end{align*}
Thus
\begin{align*}
	\int_{-L(t)}^{L(t)}J(x-y)\phi(t,y)\rd y\geq  (1-\epsilon) \phi(t, x).
\end{align*}

Summarising, from the above conclusions in cases {\bf (a)} and {\bf (b)}, we see that \eqref{2.10b} holds if $L(t)\geq \theta^* \xi(t)$ for all $t\geq 0$ with $\theta^*:=\frac{2^{\rho+1}L_0\rho}{\epsilon}>2L_0$.
The proof is finished.
\end{proof}

\subsection{Lower bounds} 
From now on, in all our stated results, we will only list the conclusions for $h(t)$; the corresponding conclusions for $-g (t)$ follow directly by considering the problem with initial function $u_0(-x)$, whose unique solution is given by $(\tilde u(t,x), \tilde g (t), \tilde h(t))=(u(t,-x), -h(t), -g (t))$.

\subsubsection{The case  \eqref{8.1} holds with $\alpha\in (1,2)$ and the  case \eqref{jbeta} holds}

\begin{lemma}\label{lemma5.3} Assume that $J$ satisfies {\bf (J)} and either \eqref{8.1} with $\alpha\in (1,2)$  or \eqref{jbeta}, $f$ satisfies {\bf (f)}, and spreading happens to \eqref{e1}.
  Then 
  \[\begin{cases}
  	\dd  \liminf_{t\to\yy}\frac{h(t)}{t^{1/({\alpha}-1)}}\geq \Big(\frac{2^{2-\alpha}}{2-\alpha}\mu\underline \lambda\Big)^{1/(\alpha-1)} &  \mbox{ if \eqref{8.1} holds with }\ \alpha\in (1,2),\\[2mm]
  	\dd  \liminf_{t\to\yy}\frac{\ln h(t)}{t^{1/\beta}}\geq \Big( \frac{2\beta\mu\underline \lambda}{\beta-1}\Big)^{1/\beta}& \mbox{ if \eqref{jbeta} holds.}
  \end{cases}
  \]
\end{lemma}
\begin{proof}
	We construct a suitable lower solution to \eqref{e1}, which will lead to the desired estimate by the comparison principle.	
		
	Let $\rho> 2$ be a large  constant to be determined. 	For any given small $\epsilon>0$, define for $t\geq0$,
	\begin{equation*}
		\begin{cases}
			\underline{h}(t):=(K_1t+\theta)^{\frac{1}{\alpha-1}},~~\underline{g}(t):=-\underline{h}(t) &\mbox{ if \eqref{8.1} holds with }\alpha\in (1,2),\\[2mm]
			\underline{h}(t):=e^{K_1(t+\theta)^{1/\beta}},~~\underline{g}(t):=-\underline{h}(t) &\mbox{ if \eqref{jbeta} holds},
		\end{cases}
	\end{equation*}
and
\begin{align*}
		\underline{u}(t,x):=K_2	\min\lf\{1,	\Big[K_3\dd \frac{\underline h(t)-|x|}{\underline h(t)}\Big]^{\rho}\rr\} &\mbox{ for } t\geq 0,\ |x|\leq \underline{h}(t),
\end{align*}
where
\begin{equation*}
K_1:=	\begin{cases}
\dd (1-\epsilon)^2(2-\epsilon)^{2-\alpha} D_{\epsilon/(2-\epsilon)}(\alpha-1)\mu\underline \lambda&\mbox{ if \eqref{8.1} holds with } \alpha\in (1,2),\\[2mm]
\dd \Big[(1-\epsilon)^4\frac{2\beta\mu\underline \lambda}{\beta-1}\Big]^{1/\beta}&\mbox{ if \eqref{jbeta} holds},\\
	\end{cases}
\end{equation*}
\begin{align*}
 K_2:=1-\epsilon,\ \  K_3:=1/\epsilon, \ \ \theta\gg 1 \mbox{ and $D_{\epsilon/(2-\epsilon)}$ is given according to \eqref{Dm}. }
\end{align*}

It is easily seen that 
\[
\mbox{\color{red}$\underline u(t,x)\equiv K_2=1-\epsilon$ for $|x|\leq (1-\epsilon)\underline h(t)$}.
\]

Moreover, $\underline u$ is continuous, and $\underline u_t$ exists and is continuous except when  $|x|=(1-\epsilon)\underline h(t)$, where $\underline u_t$ has a jumping discontinuity.	In what follows, we check that $(\underline u, \underline g, \underline h)$ defined above forms a lower solution to \eqref{e1}. 	We will do this in three steps.
	
	{\bf Step 1}.	We  prove the inequality
	\begin{align}\label{5.6}
		\underline{h}'(t)\leq \mu \int_{-\underline{h}(t)}^{\underline{h}(t)}\int^{+\infty}_{\underline{h}(t)}J(y-x)\underline u(t,x)\rd y \rd x,
	\end{align}
	which immediately gives
	\begin{align*}
		\underline{g}'(t)\geq-\mu \int_{-\underline{h}(t)}^{\underline{h}(t)}\int_{-\infty}^{-\underline{h}(t)}J(y-x)\underline u(t,x)\rd y \rd x,
	\end{align*}
	due to $\underline u(t,x)=\underline u(t,-x)$ and $J(x)=J(-x)$.

		Using the definition of $\underline u$, we have
	\begin{align*}
		\mu \int_{-\underline h}^{\underline h} \int_{\underline h}^{+\yy} J(y-x) \underline u(t,x)\rd y\rd x&\geq  (1-\epsilon)\mu  \int_{-(1-\epsilon)\underline h}^{(1-\epsilon)\underline h} \int_{\underline h}^{+\yy} J(y-x) \rd y\rd x\\
		&= (1-\epsilon)\mu  \int_{-(2-\epsilon)\underline h}^{-\epsilon \underline h} \int_{0}^{+\yy} J(y-x) \rd y\rd x.
	\end{align*}
	Using  Lemma \ref{lemma5.1},  we obtain for large $\underline h$ (guaranteed by $\theta\gg 1$),
	\begin{align*}
	  &\dd \int_{-(2-\epsilon)\underline h}^{-\epsilon \underline h} \int_{0}^{+\yy} J(y-x) \rd y\rd x\geq  (1-\epsilon)\underline \lambda D_{\epsilon/(2-\epsilon)}[(2-\epsilon)\underline h]^{2-\alpha}\ \ \mbox{ if \eqref{8.1} holds with }\ \alpha\in (1,2),
	\end{align*}
and
\begin{align*}
	&\dd  \int_{-(2-\epsilon)\underline h}^{-\epsilon \underline h} \int_{0}^{+\yy} J(y-x) \rd y\rd x\geq  (1-\epsilon) \frac{(1-\frac{\epsilon}{2-\epsilon})\underline \lambda}{\beta-1}(2-\epsilon) \underline h \big[\ln (2-\epsilon) \underline h\big]^{1-\beta}\\
	&=(1-\epsilon)^2\frac{2\underline \lambda}{\beta-1} \underline h \big[\ln (2-\epsilon) \underline h\big]^{1-\beta}\geq
	(1-\epsilon)^3\frac{2\underline \lambda}{\beta-1} \underline h (\ln \underline h)^{1-\beta} \ \mbox{ if \eqref{jbeta} holds}.
\end{align*}
	Therefore, by the definition of $K_1$, when \eqref{8.1} holds with $\alpha\in (1,2)$, we have
	\begin{align*}
		&{\mu} \int_{-\underline h(t)}^{\underline h(t)} \int_{\underline h(t)}^{+\yy}  J(y-x) \underline u(t,x)\rd y\rd x\\
		&\geq  (1-\epsilon)^2\mu\underline \lambda D_{\epsilon/(2-\epsilon)}[(2-\epsilon)\underline h(t)]^{2-\alpha}\\
		&= (1-\epsilon)^2\mu\underline \lambda D_{\epsilon/(2-\epsilon)}(2-\epsilon)^{2-\alpha}(K_1t+\theta)^{(2-\alpha)/({\alpha}-1)}\\
		&=  \frac{K_1}{{\alpha}-1} (K_1t+\theta)^{(2-\alpha)/({\alpha}-1)}=\underline h'(t);
	\end{align*}
and when \eqref{jbeta} holds, we have
	\begin{align*}
	&{\mu} \int_{-\underline h(t)}^{\underline h(t)} \int_{\underline h(t)}^{+\yy}  J(y-x) \underline u(t,x)\rd y\rd x\\
	&\geq\mu (1-\epsilon)^4\frac{2\underline \lambda}{\beta-1} \underline h (\ln \underline h)^{1-\beta}\\
	&= \frac{K_1^\beta}{\beta}\underline h (\ln \underline h)^{1-\beta}=\underline h'(t).
\end{align*}
		This proves  \eqref{5.6}.

	{\bf Step 2}.	We prove the following inequality for $t>0$ and $|x|\in [0, \underline h(t)]\setminus\{(1-\epsilon)\underline h(t)\}$,
	\begin{align}\label{5.7}
	\underline u_t\leq  &\ d \int_{-\underline h}^{\underline h}  {J}(x-y) \underline u(t,y)\rd y -d\underline u+f(\underline u).
\end{align}

From the definition of $\underline u$, we see that
\begin{align*}
		\underline{u}_t=0 \ \mbox{ for } |x|< (1-\epsilon)\underline h(t),
\end{align*}
and for $ (1-\epsilon)\underline h(t)<|x|<\underline h(t)$, if \eqref{8.1} holds with $\alpha\in (1,2)$, then
	\begin{align}\label{5.8}
		\underline{u}_t=K_2K_3^{\rho}\rho\lf(\frac{\underline h-|x|}{\underline h}\rr)^{\rho-1}\frac{\underline{h}'|x|}{\underline{h}^2}=\frac{K_1K_2K_3^{\rho}\rho}{\alpha-1}\lf(\frac{\underline h-|x|}{\underline h}\rr)^{\rho-1}\frac{|x|}{\underline{h}}\underline{h}^{1-\alpha},
	\end{align}
	where we have used $\underline h'=\frac{K_1}{\alpha-1}\underline h^{2-\alpha}$;
 and if \eqref{jbeta} holds, then
\begin{align*}
		\underline{u}_t=K_2K_3^{\rho}\rho\lf(\frac{\underline h-|x|}{\underline h}\rr)^{\rho-1}\frac{\underline{h}'|x|}{\underline{h}^2}=\frac{K_1^\beta K_2K_3^{\rho}\rho}{\beta}\lf(\frac{\underline h-|x|}{\underline h}\rr)^{\rho-1}\frac{|x|}{\underline{h}}(\ln \underline h)^{1-\beta},
\end{align*}
	where we have utilized $\underline h'=\frac{K_1^\beta}{\beta}\underline h (\ln \underline h)^{1-\beta}$.

	{\bf Claim}. There is $C_1=C_1(\epsilon)>0$ such that for $x\in [-\underline h(t),\underline h(t)]$ and $t\geq 0$,
\begin{align*}
	d \int_{-\underline h(t)}^{\underline h(t)}  {J}(x-y) \underline u(t,y)\rd y -d \underline u+f(\underline u)
	\geq  C_1\lf[\int_{-\underline h(t)}^{\underline h(t)}  {J}(x-y) \underline u(t,y)\rd y+\underline u\rr].
\end{align*}

The definition of $\underline u$ indicates   $0\leq \underline u(t,x)\leq K_2=1-\epsilon<1$. By the properties of $f$, there exists  $ \wtd C_1:= \wtd C_1(\epsilon)\in (0, d)$ such that 
\begin{align*}
	f(s)\geq  \wtd C_1 s\ \mbox{ for }\ s\in [0,K_2]. 
\end{align*}
Using Lemma \ref{lem2.2a} with
\[
(L(t),\phi(t,x),\xi(t))=(\underline h(t), \underline u(t,x)/K_2,K_3), 
\]
 for any given small $\delta>0$, we can find large $h_*=h_*(\delta,\epsilon)$ such that for $\underline h\geq h_*$ and $|x|\leq \underline h$,
\begin{align*}
	&\int^{\underline h}_{-\underline h}J(x-y)\underline{u}(t,y)dy\geq (1-\delta)\underline u(t,x).
\end{align*}
Hence, due to $d>\wtd C_1$, 
\begin{align*}
		&d \int_{-\underline h}^{\underline h}  {J}(x-y) \underline u(t,y)\rd y -d \underline u(t,x)+f(\underline u(t,x))\\
		&\geq d\int_{-\underline h}^{\underline h}  {J}(x-y) \underline u(t,y)\rd y +(\wtd C_1-d) \underline u(t,x)\\
		&{\geq  \frac{\wtd C_1}3 \int_{-\underline h}^{\underline h}  {J}(x-y) \underline u(t,y)\rd y+(d-\frac{\wtd C_1}3)(1-\delta)\underline u(t,x)+(\wtd C_1-d) \underline u(t,x)}\\
		&\geq {\frac{\wtd C_1}3}\lf[\int_{-\underline h(t)}^{\underline h(t)}  {J}(x-y) \underline u(t,y)\rd y+\underline u(t,x)\rr],
\end{align*}
 provided that $\delta=\delta(\epsilon)>0$ is sufficiently small. Thus the claim holds with $C_1=\wtd C_1/3$.

 To verify \eqref{5.7}, it remains to prove
\begin{align}\label{5.9}
	\underline u_t\leq C_1\lf[\int_{-\underline h}^{\underline h}  {J}(x-y) \underline u(t,y)\rd y+\underline u\rr]\ \mbox{ for } \ |x|\in [0, \underline h(t)]\setminus\{(1-\epsilon)\underline h(t)\}.
\end{align} 
Since $\underline u(x,t)\equiv 1-\epsilon$  for $|x|<(1-\epsilon)\underline h(t)$, \eqref{5.9} holds trivially   for such $x$. Hence we only need to consider  the case of $(1-\epsilon)\underline h(t)<|x|<\underline h(t)$.  

	Since  $\theta\gg 1$ and $0<\epsilon\ll 1$,  for $x\in [7\underline h(t)/8,\underline h(t)]\supset [(1-\epsilon)\underline h(t), \underline h(t)]$, we have
\begin{align*}\label{Jac}
&\int_{-\underline h}^{\underline h}  {J}(x-y) \underline u(t,y)\rd y\geq \int_{-7\underline h/8}^{7\underline h/8}  {J}(x-y) \underline u(t,y)\rd y\geq K_2\int_{-7\underline h/8}^{7\underline h/8}  {J}(x-y) \rd y\nonumber\\
=&\ (1-\epsilon)\int_{-7\underline h/8-x}^{7\underline h/8-x}  {J}(y) \rd y\geq (1-\epsilon)\int_{-\underline h/4}^{-\underline h/8}   {J}(y) \rd y{\color{red}
=(1-\epsilon)\int^{\underline h/4}_{\underline h/8}   {J}(y) \rd y.}
\end{align*}
Hence, when \eqref{8.1} holds with $\alpha\in (1,2)$, we obtain
\begin{align*}
	&\int_{-\underline h}^{\underline h}  {J}(x-y) \underline u(t,y)\rd y\geq  \frac{\underline \lambda}2\int_{\underline h/8}^{\underline h/4} y^{-\alpha}\rd y\nonumber=\ \frac{(8^{\alpha-1}-4^{\alpha-1})\underline \lambda}{2(\alpha-1)}\underline{h}^{1-\alpha}=:C_2\underline{h}^{1-\alpha},
\end{align*}
and when \eqref{jbeta} holds, we have
\begin{align*}
	&\int_{-\underline h}^{\underline h}  {J}(x-y) \underline u(t,y)\rd y\geq  \frac{\underline \lambda}2\int_{\underline h/8}^{\underline h/4} y^{-1} (\ln y)^{-\beta}\rd y\nonumber>{\frac{\underline \lambda\,\underline h}{16} y^{-1} (\ln y)^{-\beta}|_{y=\underline h/4}\geq}  \ \frac{\underline \lambda}{4}(\ln \underline{h})^{-\beta}=:\wtd C_2(\ln \underline{h})^{-\beta}.
\end{align*}
Similar estimates hold for   $x\in [-\underline h(t),-7\underline h(t)/8]$. 

Now, if \eqref{8.1} holds with  $\alpha\in (1,2)$, then for   $|x|\in [(1-C_\epsilon)\underline h(t), \underline h(t)]$ with
\[C_\epsilon:=\Big[\frac{C_1C_2(\alpha-1)}{K_1K_2\rho}\epsilon^{\rho}\Big]^{1/(\rho-1)}, 
\]
 we have
\begin{align*}
\underline	u_t-C_1	\int_{-\underline h}^{\underline h}  {J}(x-y) \underline u(t,y)\rd y&\leq \frac{K_1K_2K_3^{\rho}\rho}{\alpha-1}\lf(\frac{\underline h-|x|}{\underline h}\rr)^{\rho-1}\underline{h}^{1-\alpha}-C_1C_2\underline{h}^{1-\alpha}\\
&\leq \Big[\frac{K_1K_2K_3^{\rho}\rho}{\alpha-1}C_\epsilon^{\rho-1}-C_1C_2\Big]\underline{h}^{1-\alpha}= 0,
\end{align*}
and for  $(1-\epsilon)\underline h(t)<|x|\leq (1-C_\epsilon)\underline h(t)$,   using the definition of $\underline u$, we obtain
\begin{align*}
	\underline	u_t-C_1 \underline	u=&\lf[\frac{K_1\rho}{\alpha-1}\lf(\frac{\underline h-|x|}{\underline h}\rr)^{-1}\frac{|x|}{\underline{h}}\underline{h}^{1-\alpha}-C_1\rr]\underline u\\
	\leq&\lf[\frac{K_1\rho}{C_\epsilon(\alpha-1)} \underline{h}^{1-\alpha}-C_1\rr]\underline u\leq 0
\end{align*}
since $\theta\gg 1$  and $\underline h(t)\geq \theta^{1/(\alpha-1)}$, $1-\alpha<0$. We have thus proved \eqref{5.9}. 
	
We next deal with the case that \eqref{jbeta} holds. If  $|x|$  satisfies 
	\begin{align*}
		\underline h(t)\geq |x|\geq  \Big[1-\frac{\wtd C_\epsilon}{(\ln \underline h(t))^{1/(\rho-1)}}\Big]\underline h(t),
	\end{align*}
	with
	\[
	\wtd C_\epsilon:= \lf[\frac{C_1\wtd C_2\beta}{K_1^\beta K_2K_3^{\rho}\rho}\rr]^{1/(\rho-1)}= \lf[\frac{C_1\wtd C_2\beta\epsilon^\rho}{K_1^\beta K_2\rho}\rr]^{1/(\rho-1)},
	\]
	then $|x|\in [7\underline h(t)/8,\underline h(t)]$ and
	\begin{align*}
		\underline	u_t-C_1	\int_{-\underline h}^{\underline h}  {J}(x-y) \underline u(t,y)\rd y&\leq \frac{K_1^\beta K_2K_3^{\rho}\rho}{\beta}\lf(\frac{\underline h-|x|}{\underline h}\rr)^{\rho-1}(\ln \underline h)^{1-\beta}-C_1\wtd C_2(\ln \underline{h})^{-\beta}\\
		&\leq \Big[\frac{K_1^\beta K_2K_3^{\rho}\rho}{\beta}\wtd C_\epsilon^{\rho-1}-C_1\wtd C_2\Big](\ln \underline{h})^{-\beta}=
		0.
	\end{align*}
	For  $(1-\epsilon)\underline h<|x|\leq [1-\frac{\wtd C_\epsilon}{(\ln \underline h)^{1/(\rho-1)}}]\underline h$,   from the definition of $\underline u$, we deduce
	\begin{align*}
		\underline	u_t-C_1 \underline	u=&\lf[\frac{K_1^\beta \rho}{\beta}\lf(\frac{\underline h-|x|}{\underline h}\rr)^{-1}\frac{|x|}{\underline{h}}(\ln \underline h)^{1-\beta}-C_1\rr]\underline u\\
		\leq&\lf[\frac{K_1^\beta \rho}{\wtd C_\epsilon\beta}(\ln \underline h)^{1-\beta+(\rho-1)^{-1}}-C_1\rr]\underline u\leq 0
	\end{align*}
	since  $\underline h(t)\geq e^{K_1\theta^{1/\beta}}\gg 1$ and we may choose $\rho$ large enough such that $1-\beta+(\rho-1)^{-1}<0$. The desired inequality \eqref{5.9} is thus proved.

	{\bf Step 3}. Completion of the proof by the comparison principle.
	
	Since spreading happens, there is $t_0>0$ large enough such that $[g(t_0),h(t_0)]\supset[-\underline{h}(0),\underline{h}(0)]$, and also
	\begin{align*}
		u(t_0,x)\geq K_2=1-\epsilon \geq \underline{u}(0,x)\ \  \mbox{ for } 	x\in[-\underline{h}(0),\underline{h}(0)].
	\end{align*}
Moreover,	from the definition of $\underline u$, we see $\underline{u}(x,t)=0$ for $x=\pm \underline{h}(t)$ and $t\geq0$. Thus we are in a position to apply  the comparison principle to conclude that
	\[
	-\underline h(t)\geq g(t_0+t),
	\ \ \underline{h}(t)\leq h(t_0+t) \mbox{ for }t\geq0.
	\]
The desired conclusion then follows from the   arbitrariness of $\epsilon>0$ and the fact that  $D_{\epsilon/(2-\epsilon)}\to D_0$ as $\epsilon\to 0$. 	The proof is finished.
\end{proof}

\subsubsection{The case  that \eqref{8.1} holds with ${\alpha}=2$}

\begin{lemma}\label{lemma7.5}
If the conditions in Lemma {\rm \ref{lemma5.3}}  are satisfied except that $J$ satisfies \eqref{8.1} with ${\alpha}=2$, then 
	\begin{align}\label{7.13a}
	\dd \liminf_{t\to\yy}\frac{h(t)}{t\ln t}\geq \dd	 \mu \underline \lambda. 
	\end{align}
\end{lemma}
\begin{proof}
For fixed $\rho\geq 2$,  $0<\epsilon \ll 1$,  $0<\td \epsilon\ll 1$ and $\theta\gg 1$,  define 
	\[
	\begin{cases}
	\underline h(t):=K_1(t+\theta)\ln (t+\theta), & t\geq 0,\\
	\dd	\underline u(t,x):=K_2\min\lf\{1, \lf[\frac{\underline h(t)-|x|}{(t+\theta)^{\td\epsilon}}\rr]^{\rho}\rr\}, & t\geq 0,\ x\in [-\underline h(t), \underline h(t)],
	\end{cases}
	\]
	where
	\begin{align*}
		K_1:=(1-\epsilon)^3 (1-\td\epsilon)\mu \underline \lambda,\ \ K_2:=1-\epsilon.
	\end{align*}
	Note that
	\[
	{\color{red} \underline u(t,x)=K_2=1-\epsilon \mbox{ for } |x|\leq \underline h(t)-(t+\theta)^{\td\epsilon}.}
	\]
	Obviously, for any $t> 0$, $\partial_t \underline u(t,x)$ exists for $x\in [-\underline h(t), \underline h(t)]$  except 
	when $|x|=\underline h(t)-(t+\theta)^{\td\epsilon}$. However, the one-sided partial derivatives  $\partial_t\underline u(t\pm 0, x)$ always exist.

	{\bf Step 1.}	 We show that for  $\theta\gg 1$,
	\begin{align}\label{5.13}
	&\underline h'(t)\leq   \mu \int_{-\underline h(t)}^{\underline h(t)} \int_{\underline h(t)}^{+\yy} J(y-x) \underline u(t,x)\rd y\rd x\  \mbox{ for }  t> 0,
	\end{align}
	which clearly implies,  due to $\underline u(t,x)=\underline u(t,-x)$ and ${J}(x)={J}(-x)$, that
	\[
	-\underline h'(t)\geq -  \mu\int_{-\underline h(t)}^{\underline h(t)} \int_{-\yy}^{-\underline h(t)} J(y-x) \underline u(t,x)\rd y\rd x\   \mbox{ for }   t> 0.
	\]

Making use of the definition of $\underline u$ and 
\begin{align*}
\big[-2(1-\epsilon)\underline h,-[2(1-\epsilon)\underline h]^{\td\epsilon}\big]\subset 	[-2\underline h+(t+\theta)^{\td\epsilon},-(t+\theta)^{\td\epsilon}]
\end{align*}
for  $\theta\gg 1$, we obtain
	\begin{align*}
	&\mu \int_{-\underline h}^{\underline h} \int_{\underline h}^{+\yy} J(y-x) \underline u(t,x)\rd y\rd x\geq  (1-\epsilon)\mu  \int_{-\underline h+(t+\theta)^{\td\epsilon}}^{\underline h-(t+\theta)^{\td\epsilon}} \int_{\underline h}^{+\yy} J(y-x) \rd y\rd x\\
	=&\ (1-\epsilon)\mu  \int_{-2\underline h+(t+\theta)^{\td\epsilon}}^{-(t+\theta)^{\td\epsilon}} \int_{0}^{+\yy} J(y-x) \rd y\rd x\geq (1-\epsilon)\mu  \int_{-2(1-\epsilon)\underline h}^{-[2(1-\epsilon)\underline h]^{\td\epsilon}} \int_{0}^{+\yy} J(y-x) \rd y\rd x.
\end{align*}
	Thanks to Lemma \ref{lemma5.1},   for large $\underline h$ (which is guaranteed by $\theta\gg1$),
	\begin{align*}
	\int_{-2(1-\epsilon)\underline h}^{-[2(1-\epsilon)\underline h]^{\td\epsilon}} \int_{0}^{+\yy} J(y-x) \rd y\rd x\geq (1-\epsilon)(1-\td\epsilon)\underline \lambda\ln [2(1-\epsilon)\underline h].
	\end{align*}
	Hence, with $\theta\gg1$, we have
	\begin{align*}
		&{\mu} \int_{-\underline h(t)}^{\underline h(t)} \int_{\underline h(t)}^{+\yy}  J(y-x) \underline u(t,x)\rd y\rd x\\
		&\geq (1-\epsilon)^2\mu(1-\td\epsilon)\underline \lambda\ln [2(1-\epsilon)\underline h]\\
		&= (1-\epsilon)^2\mu(1-\td\epsilon)\underline \lambda\big\{\ln (t+\theta)+\ln [ \ln(t+\theta)]+\ln [2(1-\epsilon)K_1]\big\}\\
		&\geq K_1\ln (t+\theta)+K_1=\underline h'(t) \mbox{ for  all $t>0$,}
	\end{align*}
	 which 	proves \eqref{5.13}.

	{\bf Step 2.} We show that	  for $t>0$ and $x\in [-\underline h(t),\underline h(t)]$ with $|x|\not=\underline h(t)-(t+\theta)^{\td\epsilon}$, 
	\begin{align}\label{5.14}
	\underline u_t(t,x)\leq  d \int_{-\underline h(t)}^{\underline h(t)}  {J}(x-y) \underline u(t,y)\rd y -d\underline u(t,x)+f(\underline u(t,x))
	\end{align}
 for  $\theta\gg1$.
	
	From the definition of $\underline u$, we obtain by direct calculation that, for $t>0$,
	\begin{align}\label{5.15}
	\underline u_t(t,x)=\begin{cases} 
		\rho K_2^{\rho^{-1}} \underline u^{1-\rho^{-1}}\lf[K_1\frac{(1-\td\epsilon)\ln (t+\theta)+1}{(t+\theta)^{\td\epsilon}} +\frac{\td\epsilon |x|}{(t+\theta)^{1+\td\epsilon}}\rr]& \mbox{ if } \underline h(t)-(t+\theta)^{\td\epsilon}<|x|\leq \underline h(t),\\
	0 & \mbox{ if } 0\leq |x|< \underline h(t)-(t+\theta)^{\td\epsilon}.
	\end{cases}
	\end{align}
	
	Making use of  Lemma \ref{lem2.2a}  with 
	\[
	(L(t),\phi(t,x),\xi(t))=(\underline h(t),\underline u(t,x)/K_2, \frac{\underline h(t)}{(t+\theta)^{\td\epsilon}}), 
	\]
	   for any given small $\delta>0$, we can find a large $\theta_*=\theta_*(\delta,\epsilon)$ such that for $\theta\geq \theta_*$ and $|x|\leq \underline h(t)$,
	\begin{align*}
		&\int^{\underline h(t)}_{-\underline h(t)}J(x-y)\underline{u}(y,t)dy\geq (1-\delta)\underline u(x,t).
	\end{align*}
Then, a similar analysis as in the proof of Lemma \ref{lemma5.3} shows  that there exists $C_1>0$, depending on $\epsilon$ and $\delta$, such that for $\theta\gg 1$, $x\in [-\underline h(t),\underline h(t)]$ and $t\geq 0$,
\begin{align*}
	d \int_{-\underline h(t)}^{\underline h(t)}  {J}(x-y) \underline u(t,y)\rd y -d \underline u+f(\underline u)
	\geq  C_1\lf[\int_{-\underline h(t)}^{\underline h(t)}  {J}(x-y) \underline u(t,y)\rd y+\underline u\rr].
\end{align*}
Hence, to verify \eqref{5.14}, we only need to show that 
\begin{align}\label{5.16}
	\underline u_t\leq C_1\lf[\int_{-\underline h(t)}^{\underline h(t)}  {J}(x-y) \underline u(t,y)\rd y+\underline u\rr] \ \mbox{ for } \ |x|\in [0, \underline h(t)]\setminus\{ \underline h(t)-(t+\theta)^{\td\epsilon}\}.
\end{align}

Clearly, \eqref{5.16} holds trivially for $0\leq |x|<\underline h(t)-(t+\theta)^{\td\epsilon}$ due to $\underline u_t=0$ for such $x$. We next consider the remaining  case $\underline h(t)-(t+\theta)^{\td\epsilon}<|x|<\underline h(t)$.

Denote $\eta=\eta(t):=(t+\theta)^{{\td\epsilon}}$. 
	Using $\theta\gg1$ and \eqref{8.1},  we obtain, for $x\in [\underline h(t)-\eta(t),\underline h(t)]$,
\begin{align*}
	&\int_{-\underline h}^{\underline h}  {J}(x-y) \underline u(t,y)\rd y\geq \int_{-\underline h+\eta}^{\underline h-\eta}  {J}(x-y) \underline u(t,y)\rd y= K_2\int_{-\underline h+\eta}^{\underline h-\eta}   {J}(x-y) \rd y\nonumber\\
	&=K_2\int_{-\underline h+\eta-x}^{\underline h-\eta-x}  {J}(y) \rd y\geq K_2\int_{-\underline h}^{-\eta}   {J}(y) \rd y\geq  \frac{K_2\underline\lambda}2\int_{\eta}^{\underline h} y^{-2}\rd y\nonumber\\
	&=\frac{K_2\underline\lambda}{2}(\eta^{-1}-\underline{h}^{-1})\geq \frac{(1-\epsilon)\underline\lambda}{4}\eta^{-1}=:C_2(t+\theta)^{-\td\epsilon}.
\end{align*}

The same estimate also holds for   $x\in [-\underline h(t),-\underline h(t)+\eta(t)]$.  Therefore, for $|x|\in [\underline h(t)-\eta(t),\underline h(t)]$, due to $\rho>2$ and $0<\td\epsilon\ll1$,
we have
\begin{align*}
	&\underline u_t(t,x)-C_1\int_{-\underline h}^{\underline h}  {J}(x-y) \underline u(t,y)\rd y\\
	&\leq  	\rho K_2^{1/\rho} \underline u^{(\rho-1)/\rho}\lf[K_1\frac{(1-{\td\epsilon})\ln (t+\theta)+1}{(t+\theta)^{{\td\epsilon}}} +\frac{{\td\epsilon} \underline h}{(t+\theta)^{1+{\td\epsilon}}}\rr]-C_1C_2(t+\theta)^{-\td\epsilon}\\
	&\leq 2K_1	\rho K_2^{1/\rho} \underline u^{(\rho-1)/\rho}\frac{\ln (t+\theta)}{(t+\theta)^{{\td\epsilon}}} -C_1C_2(t+\theta)^{-{\td\epsilon}}\\
	&=\frac{2K_1	\rho K_2 [(\underline h-|x|)/(t+\theta)^{{\td\epsilon}}]^{\rho-1}\ln (t+\theta)-C_1C_2}{(t+\theta)^{{\td\epsilon}}} 
		\leq 0
\end{align*}
if $|x|$ further satisfies 
\begin{align*}
	|x|\geq \underline h(t)-\lf(\frac{C_1C_2}{2K_1	\rho K_2}\rr)^{1/(\rho-1)}\frac{(t+\theta)^{\td\epsilon}}{[\ln (t+\theta)]^{1/(\rho-1)}}=:\underline h(t)-C_3\frac{(t+\theta)^{{\td\epsilon}}}{[\ln (t+\theta)]^{1/(\rho-1)}}.
\end{align*}

On the other hand, for $\underline h(t)-(t+\theta)^{\td\epsilon}<|x|<\underline h(t)-C_3{(t+\theta)^{{\td\epsilon}}}/[\ln (t+\theta)]^{1/(\rho-1)}$, using \eqref{5.15} and $0<\td\epsilon\ll 1$, $\theta\gg1$,  we deduce
\begin{align*}
	\underline u_t-C_1\underline u&
	\leq 2K_1	\rho K_2^{1/\rho} \underline u^{(\rho-1)/\rho}\frac{\ln (t+\theta)}{(t+\theta)^{{\td\epsilon}}} -C_1\underline u\\
	&=\underline u \lf(\frac{2K_1	\rho  [(\underline h-|x|)/(t+\theta)^{{\td\epsilon}}]^{-1/\rho}\ln (t+\theta)}{(t+\theta)^{{\td\epsilon}}}-C_1\rr)\\
	&\leq \underline u \lf(\frac{2K_1	\rho [\ln (t+\theta)]^{1+\frac1{\rho(\rho-1)}}}{C_3^{1/\rho}(t+\theta)^{{\td\epsilon}}}-C_1\rr)<0.
\end{align*}
Hence, \eqref{5.16} holds true. This concludes Step 2. 

\smallskip
	
	{\bf Step 3.} We finally prove \eqref{7.13a}.
	
The definition of $\underline u$ clearly gives  $\underline u(t,\pm \underline h(t))=0$ for $t\geq 0$.  Since spreading happens for $(u,g,h)$ and $K_2=1-\epsilon<1$, there is a large constant $t_0>0$ such that 
	\begin{align*}
	&[-\underline h(0),\underline h(0)]\subset (g (t_0),h(t_0)),\\
	&\underline u(0,x)\leq K_2\leq u(t_0,x)\ \mbox{ for } \ x\in [-\underline h(0),\underline h(0)].
	\end{align*}
	 It follows that
	\begin{align*}
	&[-\underline h(t),\underline h(t)]\subset [g(t+t_0),h(t+t_0)] \ && \mbox{ for } t\geq 0,\\
	& \underline u(t,x)\leq u(t+t_0,x)&& \mbox{ for } t\geq 0,\ x\in [-\underline h(t),\underline h(t)],
	\end{align*}
which implies
\begin{align*}
		\dd \liminf_{t\to\yy}\frac{h(t)}{t\ln t}\geq (1-\epsilon)^3 (1-{\td\epsilon})\mu\underline \lambda.
\end{align*}
Since $\epsilon>0$ and ${\td\epsilon}>0$ can be arbitrarily small,  we thus obtain	\eqref{7.13a}  by letting $\epsilon\to 0$ and ${\td\epsilon}\to 0$. This completes the proof of the lemma.	
\end{proof}

\subsection{Upper bounds} 
Recall that we will only state and prove the conclusions for $h(t)$, as the corresponding conclusion for $-g (t)$ follows directly by considering the problem with initial function $u_0(-x)$.
\begin{lemma}\label{lem2.2}
Assume that $J$ satisfies {\bf (J)} and one of the conditions \eqref{8.1} and \eqref{jbeta}, $f$ satisfies {\bf (f)}, and spreading happens to \eqref{e1}.  Then
	\begin{equation}\label{8.2a}
 \begin{cases} 
 	\dd \limsup_{t\to\yy}\frac{h(t)}{t^{1/({\alpha}-1)}}\leq \Big(\frac{2^{2-\alpha}}{2-\alpha} \mu\bar \lambda\Big)^{1/(\alpha-1)} &\mbox{ if \eqref{8.1} holds with }\ {\alpha}\in (1,2),\smallskip\\
\dd \limsup_{t\to\yy}\frac{h(t)}{t\ln t}\leq \dd	 \mu \bar \lambda &\mbox{ if  \eqref{8.1} holds with } \alpha=2,\\[2mm]
	\dd  \limsup_{t\to\yy}\frac{\ln h(t)}{t^{1/\beta}}\leq \lf( \frac{2\beta\mu\overline \lambda}{\beta-1}\rr)^{1/\beta} & \mbox{ if \eqref{jbeta} holds.}
\end{cases}
\end{equation}
\end{lemma}
\begin{proof} For any given small $\epsilon>0$, 
	define, for $t\geq 0$,
	\begin{align*}
	&\bar h(t):=\begin{cases}
	(Kt+\theta)^{1/({\alpha}-1)}&\mbox{ if \eqref{8.1} holds with}\ {\alpha}\in (1,2],\\
	K(t+\theta)\ln(t+\theta)&\mbox{ if \eqref{8.1} holds with }\ {\alpha}=2,\\
	e^{K(t+\theta)^{1/\beta}}&\mbox{ if \eqref{jbeta} holds},
	\end{cases}
\\
&	 \ol u(t,x):=1+\epsilon,\ \ \ x\in [-\bar h(t), \bar h(t)],
	\end{align*}
	where  $\theta\gg 1$ and
	\begin{equation}
K:=\begin{cases}
\dd (1+\epsilon)^3\frac{2^{2-\alpha}}{2-\alpha}\mu\bar \lambda& \mbox{ if \eqref{8.1} holds with }  \alpha\in (1, 2),\\[3mm]
\dd (1+\epsilon)^3\mu \bar \lambda & \mbox{ if \eqref{8.1} holds with }\  \alpha=2,\\[3mm]
 \dd \Big[\frac{2(1+\epsilon)^3\beta\mu\overline \lambda}{\beta-1}\Big]^{1/\beta}&\mbox{ if \eqref{jbeta} holds},
\end{cases}
	\end{equation}

	We verify that for $t>0$,
	\begin{align}\label{8.2}
	\bar h'(t)\geq \mu \int_{-\bar h(t)}^{\bar h(t)} \int_{\bar h(t)}^{+\yy} J(y-x) \bar u(t,x)\rd y\rd x,
	\end{align}
	which clearly implies
	\[
	-\bar h'(t)\leq - {\mu} \int_{-\bar h(t)}^{\bar h(t)} \int_{-\yy}^{-\bar h(t)} J(y-x) \bar u(t,x)\rd y\rd x
	\]
since $\ol u(t,x)=\ol u(t,-x)$ and $J(x)=J(-x)$.
	
Using $\bar u=1+\epsilon$, we have
\begin{align*}
	\mu \int_{-\bar h}^{\bar h} \int_{\bar h}^{+\yy} J(y-x) \bar u(t,x)\rd y\rd x&= (1+\epsilon)\mu  \int_{-\bar h}^{\bar h} \int_{\bar h}^{+\yy} J(y-x) \rd y\rd x\\	
	&= (1+\epsilon)\mu  \int_{-2\bar h}^{0} \int_{0}^{+\yy} J(y-x) \rd y\rd x.
\end{align*}
By Lemma \ref{lemma5.1} with $\delta=0$,  we see that for large $\bar h$, which is guaranteed by $\theta\gg 1$,
\begin{align*}
\begin{cases}
\dd	\int_{-2\bar h}^{0} \int_{0}^{+\yy} J(y-x) \rd y\rd x\leq (1+\epsilon)\frac{\bar \lambda}{(\alpha-1)(2-\alpha)}(2\bar h)^{2-\alpha}, & \mbox{ if \eqref{8.1} holds with } \alpha\in (1,2),\\[3mm]
	\dd	\int_{-2\bar h}^{0} \int_{0}^{+\yy} J(y-x) \rd y\rd x\leq (1+\epsilon)\bar \lambda\ln (2\bar h), & \mbox{ if \eqref{8.1} holds with } \alpha=2,\\[3mm]
	\dd	\int_{-2\bar h}^{0} \int_{0}^{+\yy} J(y-x) \rd y\rd x\leq \dd (1+\epsilon)(2\bar h) [\ln (2\bar h) ]^{1-\beta}\frac{\bar \lambda}{\beta-1} &\mbox{ if \eqref{jbeta} holds}.
		\end{cases}
\end{align*}
	Therefore, when \eqref{8.1} holds with $\alpha\in (1,2)$, by the definition of $K$, we have
	\begin{align*}
	{\mu} \int_{-\bar h}^{\bar h} \int_{\bar h}^{+\yy}  J(y-x) \bar u(t,x)\rd y\rd x
	&\leq (1+\epsilon)^2\mu\frac{\bar \lambda }{(\alpha-1)(2-\alpha)}(2\bar h)^{2-\alpha}\\
	&= (1+\epsilon)^2\mu\frac{\bar \lambda }{(\alpha-1)(2-\alpha)}
	2^{2-\alpha}(Kt+\theta)^{(2-\alpha)/({\alpha}-1)}\\
	&\leq   \frac{K}{{\alpha}-1} (Kt+\theta)^{(2-\alpha)/({\alpha}-1)}=\bar h'(t).
	\end{align*}
	
 When \eqref{8.1} holds with  ${\alpha}=2$, we similarly obtain, due to $\theta\gg 1$,
	\begin{align*}
	{\mu} \int_{-\bar h}^{\bar h} \int_{\bar h}^{+\yy}  J(y-x) \bar u(t,x)\rd y\rd x
	&\leq (1+\epsilon)^2\mu \bar \lambda \ln (2\bar h)\\
	&= (1+\epsilon)^2\mu \bar \lambda \big\{\ln (t+\theta)+\ln [ \ln(t+\theta)]+\ln 2K\big\}\\
	&\leq K\ln (t+\theta)+K=\bar h'(t).
	\end{align*}
	
Finally, when \eqref{jbeta} holds, we have 
	\begin{align*}
			{\mu} \int_{-\bar h}^{\bar h} \int_{\bar h}^{+\yy}  J(y-x) \bar u(t,x)\rd y\rd x
		&\leq (1+\epsilon)^2\mu (2\bar h) [\ln (2\bar h) ]^{1-\beta}\frac{\bar \lambda}{\beta-1} \\
		&\leq(1+\epsilon)^3\mu (2\bar h) (\ln \bar h)^{1-\beta}\frac{\bar \lambda}{\beta-1}\\
		&= \frac{K^\beta}{\beta}\underline h (\ln \underline h)^{1-\beta}=\underline h'(t).
	\end{align*}

		Thus \eqref{8.2} always holds. 
	
	Recalling that  $\ol u\geq 1$ is a constant,  we get, 
	for $t>0$, $x\in [-\bar h(t),\bar h(t)]$, 
	\begin{align*}
	\ol u_t(t,x)\equiv 0\geq  d \int_{-\bar h}^{\bar h}  {J}(x-y) \ol u(t,y)\rd y -d\ol u(t,x)+f(\ol u(t,x)).
	\end{align*}
	
Note that condition {\bf (f)} implies, by simple comparison with ODE solutions,  
\[\limsup_{t\to\yy} \max_{x\in [g(t),h(t)]}u(t,x)\leq 1;
\]
 hence there is $t_0>0$ such that 
\begin{align*}
	u(t_0,x)\leq 1+\epsilon= \ol u(t_0,x) \ \mbox{ for } x\in [g(t_0),h(t_0)]\subset [-\bar h(0),\bar h(0)]
\end{align*}
with the last part holding for large $\theta$. 

We are now in a position  to use  the comparison principle  (Theorem 1.3) to conclude that
	\begin{align*}
	&[g(t+t_0),h(t+t_0)]\subset  [-\bar h(t),\bar h(t)]\ {\rm for}\ t\geq 0,\ \\
	& u(t+t_0,x)\leq \ol u(t,x)\ {\rm for}\ t\geq 0,\ x\in [ g (t+t_0), h(t+t_0)].
	\end{align*}
By the arbitrariness of $\epsilon>0$, we get  \eqref{8.2a}.  The proof is finished.
\end{proof}

{\bf Proof of Theorem \ref{thm1}:} The conclusions for $g(t)$ and $h(t)$ follow directly from the above lower and upper bounds. The conclusion on $\lim_{t\to\infty} u(t,x)$ follows from the definitions of the lower and upper solutions.

	\end{document}